\documentclass{amsart}

\usepackage{amsmath}
\usepackage{amsfonts}
\usepackage{amssymb}
\usepackage{amsthm}
\usepackage{comment}
\usepackage{epsfig}
\usepackage{psfrag}
\usepackage{mathrsfs}
\usepackage{amscd}
\usepackage[all]{xy}
\usepackage{rotating}
\usepackage{lscape}
\usepackage{amsbsy}
\usepackage{verbatim}
\usepackage{moreverb}

\newtheorem{theorem}{Theorem}[section]
\newtheorem{prop}[theorem]{Proposition}
\newtheorem{lemma}[theorem]{Lemma}
\newtheorem{cor}[theorem]{Corollary}
\newtheorem{conj}[theorem]{Conjecture}

\theoremstyle{definition}
\newtheorem{definition}[theorem]{Definition}

\theoremstyle{remark}
\newtheorem{remark}[theorem]{Remark}
\newtheorem{example}[theorem]{Example}

\DeclareMathOperator{\Aut}{Aut}

\DeclareMathOperator*{\colim}{colim}

\DeclareMathOperator{\Der}{Der}

\DeclareMathOperator{\Hom}{Hom}
\DeclareMathOperator{\End}{End}
\DeclareMathOperator{\ulhom}{\underline{\Hom}}
\DeclareMathOperator{\Fun}{Fun}

\DeclareMathOperator{\Map}{Map}
\DeclareMathOperator{\Mor}{Mor}
\DeclareMathOperator{\Tot}{Tot}

\DeclareMathOperator{\Spec}{Spec}

\DeclareMathOperator{\spaces}{Spaces}
\DeclareMathOperator{\stab}{Stab}

\DeclareMathOperator{\ind}{Ind}

\DeclareMathOperator{\Ker}{Ker}

\DeclareMathOperator{\free}{Free}

\DeclareMathOperator{\Cat}{Cat_\infty}

\DeclareMathOperator{\qc}{QC}
\DeclareMathOperator{\m}{Mod}

\DeclareMathOperator{\comod}{Comod}

\DeclareMathOperator{\alg}{\mathrm{-alg}}
\DeclareMathOperator{\coalg}{\mathrm{-coalg}}

\DeclareMathOperator{\Alg}{\mathrm{Alg}}
\DeclareMathOperator{\op}{\mathrm{op}}
\DeclareMathOperator{\Coalg}{Coalg}
\DeclareMathOperator{\Calg}{\mathrm{CAlg}}
\DeclareMathOperator{\algd}{\mathrm{-algebroid}}

\DeclareMathOperator{\gl}{\mathrm{GL}_1}

\DeclareMathOperator{\Top}{\mathrm{Top}}
\DeclareMathOperator{\Emb}{\mathrm{Emb}}

\DeclareMathOperator{\Space}{\mathrm{Spaces}}
\DeclareMathOperator{\spectra}{\mathrm{Spectra}}
\DeclareMathOperator{\mfld}{Mflds}

\DeclareMathOperator{\BTop}{{\mathit B}\negthinspace\Top}
\DeclareMathOperator{\BAut}{{\mathit B}\negthinspace\Aut}

\DeclareMathOperator{\Lie}{Lie}

\def\ot{\otimes}
\def\fr{\rm fr}

\DeclareMathOperator{\fin}{Fin}

\DeclareMathOperator{\oo}{\infty}

\DeclareMathOperator{\hh}{HH}

\DeclareMathOperator{\id}{id}

\DeclareMathOperator{\ba}{Bar}

\newcommand{\ra}{\rightarrow}

\newcommand{\xra}{\xrightarrow}

\newcommand{\squig}{\rightsquigarrow}

\def\cA{\mathcal A}\def\cC{\mathcal C}\def\cD{\mathcal D}
\def\cE{\mathcal E}\def\cF{\mathcal F}

\def\cM{\mathcal M}\def\cO{\mathcal O}\def\cP{\mathcal P}
\def\cQ{\mathcal Q}\def\cT{\mathcal T}
\def\cX{\mathcal X}

\def\EE{\mathbb E}\def\FF{\mathbb F}\def\GG{\mathbb G}

\def\PP{\mathbb P}
\def\RR{\mathbb R}

\def\bH{\mathbf H}

\def\fB{\mathfrak B}

\begin{document}

\title{The tangent complex and Hochschild cohomology of $\cE_n$-rings}
\author{John Francis}
\address{Department of Mathematics\\Northwestern University\\Evanston, IL 60208-2370}
\email{jnkf@northwestern.edu}


\thanks{The author was supported by the National Science Foundation under award number 0902974.}


\maketitle

\begin{abstract}
In this work, we study the deformation theory of $\cE_n$-rings and the $\cE_n$ analogue of the tangent complex, or topological Andr\'e-Quillen cohomology. We prove a generalization of a conjecture of Kontsevich, that there is a fiber sequence $A[n-1] \ra T_A\ra \hh^*_{\cE_{\!n}}(A)[n]$, relating the $\cE_n$-tangent complex and $\cE_n$-Hochschild cohomology of an $\cE_n$-ring $A$. We give two proofs: The first is direct, reducing the problem to certain stable splittings of configuration spaces of punctured Euclidean spaces; the second is more conceptual, where we identify the sequence as the Lie algebras of a fiber sequence of derived algebraic groups, $B^{n-1}A^\times\ra \Aut_A\ra \Aut_{\fB^nA}$. Here $\fB^nA$ is an enriched $(\oo,n)$-category constructed from $A$, and $\cE_n$-Hochschild cohomology is realized as the infinitesimal automorphisms of $\fB^nA$. These groups are associated to moduli problems in $\cE_{n+1}$-geometry, a {\it less} commutative form of derived algebraic geometry, in the sense of To\"en-Vezzosi and Lurie. Applying techniques of Koszul duality, this sequence consequently attains a nonunital $\cE_{n+1}$-algebra structure; in particular, the shifted tangent complex $T_A[-n]$ is a nonunital $\cE_{n+1}$-algebra. The $\cE_{n+1}$-algebra structure of this sequence extends the previously known $\cE_{n+1}$-algebra structure on $\hh^*_{\cE_{\!n}}(A)$, given in the higher Deligne conjecture. In order to establish this moduli-theoretic interpretation, we make extensive use of factorization homology, a homology theory for framed $n$-manifolds with coefficients given by $\cE_n$-algebras, constructed as a topological analogue of Beilinson-Drinfeld's chiral homology. We give a separate exposition of this theory, developing the necessary results used in our proofs. 

\end{abstract}

\tableofcontents

\section{Introduction}

In this paper, we study certain aspects of $\cE_n$-algebra, that is, algebras with multiplication maps parametrized by configuration spaces of $n$-dimensional disks inside a standard $n$-disk. We focus on the deformation theory of $\cE_n$-algebras, which is controlled by an operadic version of the tangent complex of Grothendieck and Illusie. One of our basic results is a relation between this $\cE_n$-tangent complex and $\cE_n$-Hochschild cohomology. This result generalizes a theorem of Quillen in the case of $n=1$ in \cite{quillen}, and was first conjectured by Kontsevich in \cite{motives}. Before stating our main theorem, we first recall some important examples and motivations in the theory of $\cE_n$-algebra.

The $\cE_n$ operads interpolate between the $\cE_1$ and $\cE_\infty$ operads, and as a consequence, the categories of $\cE_n$-algebras provide homotopy theoretic gradations of {\it less} commutative algebra, interpolating between noncommutative and commutative algebra. Since the second space of the operad $\cE_n(2)$ is homotopy equivalent to $S^{n-1}$ with its antipodal action by $\Sigma_2$, one can intuitively imagine an $\cE_n$-algebra as an associative algebra with multiplications parametrized by $S^{n-1}$ as a $\Sigma_2$-space, in which the antipodal map on $S^{n-1}$ exchanges an algebra structure with its opposite algebra structure. The spaces $S^{n-1}$ become more connected as $n$ increases, and for this reason one may think that an $\cE_n$-algebra is more commutative the larger the value of $n$.

For the special case of $n=1$, the space $\cE_1(2)\simeq S^0$ has two components, which reflects the fact that an algebra and its opposite need not be isomorphic. The quotient $\cE_1(2)_{\Sigma_2} \simeq S^0_{\Sigma_2} =\ast$ is equivalent to a point, and as a consequence the theory of $\cE_1$-algebras is equivalent to that of strictly associative algebras. For $n=\oo$, the space $\cE_\infty(2)\simeq S^\infty$ is contractible, corresponding to an essentially unique multiplication, but the quotient $S^{\oo}_{\Sigma_2} \cong \RR\PP^\infty \simeq B\Sigma_2$ is not contractible, and this distinguishes the theory of $\cE_\infty$-algebras from that of strictly commutative algebras in general. The rational homology ${\rm H}_*(\RR\PP^\infty, \FF)$ is trivial if $\FF$ is a field of characteristic zero, in contrast, and this has the consequence that the theories of $\cE_\infty$-algebras and strictly commutative algebras agree over a field of characteristic zero. Otherwise, the homotopy theory of strictly commutative algebras is often ill-behaved, so one might interpret this to mean that, away from characteristic zero, commutativity wants to be a structure, rather than a property.

We now consider six occurrences of $\cE_n$, each serving to motivate the study of $\cE_n$-algebra:

{\bf Iterated loop spaces}: Historically, the theory of the $\cE_n$ operad and its algebras first developed in the setting of spaces, where Boardman and Vogt originally defined $\cE_n$ in order to describe the homotopy theoretic structure inherent to an $n$-fold loop space, \cite{bv}. $\cE_n$-algebras were first used to give configuration space models of mapping spaces, and then May proved the more precise result that $n$-fold loop spaces form a full subcategory of $\cE_n$-algebras in spaces, up to homotopy, \cite{may}.

{\bf Ring spectra}: $\cE_n$-structures next arose in the study of ring spectra in algebraic topology. For instance, various of the important spectra in topology do not support $\cE_\infty$-ring structures, but do admit an $\cE_n$-algebra structure for lesser $n$, which allows for certain advantageous manipulations (such as defining the smash product of $A$-module spectra). For example, the Morava $K$-theories $K(n)$ admit a unique $\cE_1$-algebra structure, \cite{vigleik}; the Brown-Peterson spectra $BP$ are presently only known to admit an $\cE_4$-algebra structure, \cite{bp}; a Thom spectrum $Mf$ classified by a map $f: X\ra BO$ obtains an $\cE_n$-ring structure if the map $f$ is an $n$-fold loop map, which is the case for the spectra $X(n)$ in Devinatz-Hopkins-Smith's proof of Ravenel's conjectures, Thom spectra for the Bott map $\Omega SU(n) \ra BU$, a 2-fold loop map.

{\bf Quantum groups}: A very different source of $\cE_n$ structures arose in the 1980s, with the advent of the theory of quantum groups. The Hopf algebras $U_q(\frak g)$ of Drinfeld and Jimbo have an invertible element $R$ in $U_q(\frak g)^{\ot 2}$ which satisfies the Yang-Baxter equation. This gives the category of $U_q(\frak g)$-modules the structure of a braided monoidal category, or, equivalently, an $\cE_2$-algebra in categories, using the fact that the spaces $\cE_2(k)$ are classifying spaces for the pure braid groups $P_k$ on $k$ strands. This braided structure on the category gives rise to invariants of knots and 3-manifolds, such as the Jones polynomial.

{\bf Conformal and topological field theory}: $\cE_n$-algebras are topological analogues of Beilinson-Drinfeld's chiral algebras, algebro-geometric objects encoding the operator product expansions in conformal field theory. That is, via the Riemann-Hilbert correspondence, $\cE_n$-algebras bear the same relation to chiral algebras as constructible sheaves bear to D-modules. Consequently, $\cE_n$-algebras play a role in topological field theory analogous to that of chiral algebras in conformal field theory. For instance, if $\cF$ is a topological field theory in dimension $d+1$, i.e., a symmetric monoidal functor on the cobordism category of $d$-manifolds, $\cF: {\rm Cob}_{d+1}\ra \cC$, then the value $\cF(S^d)$ on the $d$-sphere has the structure of a Frobenius $\cE_{d+1}$-algebra in $\cC$, and this encodes an important slice of the structure of the field theory. In the case $d=1$ and $\cC$ is vector spaces, this augmented $\cE_2$-algebra $\cF(S^1)$ is a strictly commutative Frobenius algebra, and the field theory is determined by this algebraic object.

{\bf Homology theories for $n$-manifolds}: One can consider the notion of a {\it homology theory} for framed $n$-manifolds with coefficients in a symmetric monoidal $\oo$-category $\cC^{\ot}$. This can be defined as a symmetric monoidal functor $H: \mfld_n^{\fr} \ra \cC$ from framed $n$-manifolds, with framed embeddings as morphisms, to $\cC$. A homology theory must additionally satisfy an analogue of excision: If a manifold $M$ is decomposed along a trivialized neighborhood of a codimension-1 submanifold, $M \cong M_0 \cup_{N}M_1$, then the value $H(M)$ should be  equivalent to the two-sided tensor product $H(M_0)\ot_{H(N)}H(M_1)$.\footnote{This is equivalent to the usual excision axiom if $\cC$ is chain complexes and the monoidal structure is the coproduct, which can formulated as the assertion that ${\rm C}_*(M)$ is homotopy equivalent to ${\rm C}_*(M_0)\coprod_{{\rm C}_*(N)}{\rm C}_*(M_1)$.} There is then an equivalence $\bH(\mfld_n^{\fr}, \cC)\simeq \cE_n\alg(\cC)$ between homology theories with values in $\cC$ and $\cE_n$-algebras in $\cC$. A detailed discussion will be forthcoming in \cite{facthomology}.

{\bf Quantization}: The deformation theory of $\cE_n$-algebras is closely related to deformation quantization, going back to \cite{motives}. For instance, for a translation-invariant classical field theory $\cF$ with $A = \cO(\cF(\RR^n))$ the commutative algebra of observables, then certain $\cE_n$-algebra deformations of $A$ over a formal parameter $\hbar$ give rise to quantizations of the theory $\cF$, see \cite{kevinowen}. 

This final example provides especial impetus to study the deformation theory of $\cE_n$-algebras, our focus in the present work. In classical algebra, the cotangent complex and tangent complex play a salient role in deformation theory: The cotangent complex classifies square-zero extensions; the tangent complex $T_A$ has a Lie algebra structure, and in characteristic zero the solutions to the Maurer-Cartan equation of this Lie algebra classify more general deformations. Consequently, our study will be devoted the $\cE_n$ analogues of these algebraic structures.

We now state the main theorem of this paper. Let $A$ be an $\cE_n$-algebra in a stable symmetric monoidal $\oo$-category $\cC$, such as chain complexes or spectra. $T_A$ denotes the $\cE_n$-tangent complex of $A$, $\hh_{\cE_{\!n}}^*(A)$ is the $\cE_n$-Hochschild cohomology, $A^\times$ is the derived algebraic group of units in $A$, and $\fB^nA$ is a $\cC$-enriched $(\oo,n)$-category constructed from $A$. $\fB^nA$ should be thought of as having a single object and single $k$-morphism $\phi^k$ for $1\leq k\leq n-1$, and whose collection of $n$-morphisms is equivalent to $A$, $\Hom_{\fB^nA}(\phi^{n-1}, \phi^{n-1}) \simeq A$; this generalizes the construction of a category with a single object from a monoid. Then we prove the following:

\begin{theorem}\label{big} There is a fiber sequence $$A[n-1]\longrightarrow T_A \longrightarrow \hh^*_{\cE_{\!n}}(A)[n]$$ of Lie algebras in $\cC$. This is the dual of a cofiber sequence of $\cE_n$-$A$-modules $$\int_{S^{n-1}}\negthinspace A\thinspace[1-n] \longleftarrow L_A \longleftarrow A[-n]$$ where $L_A$ is the $\cE_n$-cotangent complex of $A$ and $\int_{S^{n-1}}A$ is the factorization homology of the $(n-1)$-sphere with coefficients in $A$. This sequence of Lie algebras may be also obtained from a fiber sequence of derived algebraic groups $$B^{n-1}A^\times \longrightarrow \Aut_A \longrightarrow \Aut_{\fB^n A}$$ by passing to the associated Lie algebras. In particular, there are equivalences: $$\xymatrix{\Lie(B^{n-1}A^\times) \simeq A[n-1], & \Lie(\Aut_A)\simeq T_A,  & \Lie(\Aut_{\fB^n A}) \simeq \hh^\ast_{\cE_{\!n}}(A)[n].\\}$$ This sequence, after desuspending by $n$, has the structure of a fiber sequence of nonunital $\cE_{n+1}$-algebras $$A[-1]\longrightarrow T_A[-n] \longrightarrow \hh^*_{\cE_{\!n}}(A)$$ arising as the tangent spaces associated to a fiber sequence of $\cE_{n+1}$-moduli problems.
\end{theorem}

\begin{remark} In the case of $n=1$, this theorem specializes to Quillen's theorem in \cite{quillen}, which says that for $A$ an associative algebra, there is a fiber sequence of Lie algebras $A\ra T_A\ra \hh^*(A)[1]$, where the Lie algebra structure, at the chain complex level, is given by \cite{scst}. The final part of the result above seems new even in the $n=1$, where it says the there is an fiber sequence of nonunital $\cE_2$-algebras $A[-1]\ra T_A[-1]\ra \hh^*(A)$. The existence of a nonunital $\cE_{n+1}$-algebra structure on $A[-1]$ is perhaps surprising; see Conjecture \ref{last} for a discussion of this structure. In general, the $\cE_{n+1}$-algebra structure on $\hh^*_{\cE_{\!n}}(A)$ presented in the theorem is that given by the higher Deligne conjecture of \cite{motives}. In that paper, Kontsevich separately conjectured that $A\ra T_A \ra \hh^*_{\cE_{\!n}}(A)[n]$ is a fiber sequence of Lie algebras and that $\hh^*_{\cE_{\!n}}(A)$ admits an $\cE_{n+1}$-algebra structure; the statement that $A[-1]\ra T_A[-n] \ra \hh^*_{\cE_{\!n}}(A)$ is a fiber sequence of nonunital $\cE_{n+1}$-algebras is thus a common generalization of those two conjectures.
\end{remark}


We now summarize the primary contents of this work, section by section:

\medskip

{\bf Section 2} presents a general theory of the cotangent complex for algebras over an operad via {\it stabilization}: From this homotopy theoretic point of view, the assignment of the cotangent complex to a commutative ring is an algebraic analogue of the assignment of the suspension spectrum to a topological space. We begin with a brief review of the basic constructions in this subject, similar to the presentations of Basterra-Mandell \cite{bm}, Goerss-Hopkins \cite{gh}, and especially Lurie \cite{dag4}. The first main result result of this section, Theorem \ref{main1}, is a cofiber sequence describing the cotangent complex of an $\cE_n$-algebra $A$ as an extension of a shift of $A$ itself and the associative enveloping algebra of $A$ (and this gives the fiber sequence in the statement of Theorem \ref{big}, but without any algebraic structure); the proof proceeds from a hands-on analysis of the cotangent complex in the case of a free $\cE_n$-algebra $A$, where the core of the result obtains from a stable splitting of configuration spaces due to McDuff, \cite{mcduff}. The second main focus of this section involves the algebraic structure obtained by the cotangent and tangent space of an augmented $\cE_n$-algebra; after some standard generalities on Koszul duality in the operadic setting, \`a la Ginzburg-Kapranov \cite{gk}, Theorem \ref{main1} is then used to prove the next central result, Theorem \ref{koszulen}, which states that the tangent space $TA$ at the augmentation of an augmented $\cE_n$-algebra $A$ has the structure of a nonunital $\cE_n[-n]$-algebra; i.e., $TA[-n]$ is a nonunital $\cE_n$-algebra. The idea that this result should hold dates to the work of Getzler-Jones \cite{getzlerjones}; the result has been known in characteristic zero to experts for a long time due to the formality of the $\cE_n$ operad, see \cite{motives} and \cite{formality}, which implies that the derived Koszul dual of ${\rm C}_*(\cE_n, \RR)$ can be calculated from the comparatively simple calculation of the classical Koszul dual of the Koszul operad ${\rm H}_*(\cE_n,\RR)$, as in \cite{getzlerjones}.

\medskip

{\bf Section 3}, which can be read independently of the preceding section, gives a concise exposition of the {\it factorization homology} of topological $n$-manifolds, a homology theory whose coefficients are given by $\cE_n$-algebras (and, more generally, $\cE_B$-algebras). This theory been recently developed in great detail by Lurie in \cite{dag6}, though slightly differently from our construction. Factorization homology is a topological analogue of Beilinson-Drinfeld's chiral homology theory, \cite{bd}, constructed using ideas from conformal field theory for applications in representation theory and the geometric Langlands program. This topological analogue is of interest in manifold theory quite independent of the rest of the present work, a line of study we pursue in \cite{facthomology}. A key result of Section 3 is Proposition \ref{excision}, a gluing, or excision, property of factorization homology: This is used extensively in our work, both retroactively in Section 2 (to calculate the relation of the $n$-fold iterated bar construction $\ba^{(n)} A$ of an augmented $\cE_n$-algebra $A$ and its cotangent space $LA$) and later in Section 4.

\medskip

{\bf Section 4} studies $\cO$-moduli problems, or formal derived geometry over $\cO$-algebras, to then apply to $\cE_n$-algebra. Using Gepner's work on enriched $\oo$-categories in \cite{gepner}, we obtain the natural fiber sequence of derived algebraic groups $B^{n-1}A^\times \ra \Aut_A\ra\Aut_{\fB^nA}$ relating the automorphisms of $A$ with the automorphisms of an enriched $(\oo,n)$-category $\fB^nA$. The tangent complexes of these moduli problems are then calculated. The main result of this section, Theorem \ref{liehh}, is the identification of the tangent complex of $\Aut_{\fB^nA}$ with a shift of the $\cE_n$-Hochschild cohomology of $A$; the proof hinges on an $\cE_n$ generalization of a theorem of \cite{qcloops}, and it fundamentally relies the $\ot$-excision property of factorization homology. The proof of Theorem \ref{big} is then completed by showing that this moduli-theoretic construction of the fiber sequence $A[-1] \ra T_A[-n] \ra \hh^*_{\!\cE_n}(A)$ automatically imbues it with the stated $\cE_{n+1}$-algebraic structure: This is consequence of Proposition \ref{extend}, a general result in Koszul duality likely familiar to experts, which, together with Theorem \ref{koszulen}, shows that the tangent space of an $\cE_{m}$-moduli problem satisfying a technical Schlessinger-type condition obtains an $\cE_{m}[-m]$-algebra structure.

\begin{remark} In this work, we use the {\it quasicategory} model of $\oo$-category theory, first developed in detail by Joyal, \cite{joyal}, and then by Lurie in \cite{topos}, which is our primary reference. Most of the arguments made in this paper would work as well in a sufficiently nice model category or a topological category. For several, however, such as constructions involving categories of functors or monadic structures, $\oo$-categories offer substantial technical advantages. The reader uncomfortable with this language can always substitute the words ``topological category" for ``$\oo$-category" wherever they occur in this paper to obtain the correct sense of the results, but with the proviso that technical difficulties may then abound in making the statements literally true. The reader only concerned with algebra in chain complexes, rather than spectra, can likewise substitute ``pre-triangulated differential graded category" for ``stable $\oo$-category" wherever those words appear, with the same proviso. See the first chapter of \cite{topos} or section 2.1 of \cite{qcloops} for a more motivated introduction to this topic.
\end{remark}

\subsection{Acknowledgements} This work is based on my 2008 PhD thesis, \cite{thez}, and I thank my advisor, Michael Hopkins, from whose guidance and insight I have benefitted enormously. I am indebted to Jacob Lurie for generously sharing his ideas, which have greatly shaped this work. I thank David Gepner for his help with enriched $\oo$-categories and for writing \cite{gepner}. I am thankful to Kevin Costello, Dennis Gaitsgory, Paul Goerss, Owen Gwilliam, David Nadler, Bertrand T\"oen, and Xinwen Zhu for helpful conversations related to this paper. I thank Gr\'egory Ginot, Owen Gwilliam, and Geoffroy Horel for finding numerous errors, typos, and expository faults in earlier drafts of this paper.

\section{The Operadic Cotangent Complex}

An essential role in the classical study of a commutative ring is played by the module of K\"ahler differentials, which detects important properties of ring maps and governs aspects of deformation theory. The module $\Omega_A$ of K\"ahler differentials of a commutative ring $A$ is defined as quotient $I/I^2$, where $I$ is the kernel of the multiplication $A\otimes A \ra A$, and $I^2$ is the ideal in $I$ of elements that products of multiple elements. $\Omega_A$ has the property that it corepresents derivations, i.e., that there is a natural equivalence $\Hom_A (\Omega_A , M) \simeq \Der (A, M)$. If $A$ is not smooth, then the assignment $M \squig \Der(A,M)$ is not right exact. Grothendieck had the insight that $\Omega_A$ has a derived enhancement, the cotangent complex $L_A$, which corepresents the right derived functor of derivations. Quillen fitted this concept to a very general model category framework of taking the left derived functor of abelianization. We first give a brief review of the rudiments of operadic algebra in $\oo$-categories; for further details and proofs we refer to \cite{dag3} or \cite{thez}. We will then discuss the appropriate version of the cotangent complex for algebras over an operad.

For $\cO$ a topological operad, we will also denote by $\cO$ the symmetric monoidal $\oo$-category whose objects are finite sets and whose morphism spaces are $\Map_\cO(J, I) = \coprod_{\pi: J\ra I} \prod_I\cO(J_i)$, where $\pi$ is a map of sets and $J_i =\pi^{-1}\{i\}$ is the inverse image of $i$. (This category is also known as the PROP associated to the operad $\cO$.) Note that there is a natural projection of $\cO\ra \fin$ from $\cO$ to the $\oo$-category of finite sets (i.e., the nerve of the category of finite sets).

\begin{definition} An $\cO$-algebra structure on $A$, an object of a symmetric monoidal $\oo$-category $\cC$, is a symmetric monoidal functor $\tilde A: \cO \ra \cC$ with an equivalence $\tilde A(\{ 1\}) \simeq A$ between $A$ and value of $A$ on the set with a single element. $\cO$-algebras in $\cC$, $\cO\alg(\cC)$, is the $\oo$-category of symmetric monoidal functors $\Fun^\ot(\cO,\cC)$.
\end{definition}

There is an intrinsic notion of a module for an $\cO$-algebra, which we will use extensively. In order to formulate this notion, we will need to use a version of the $\oo$-category $\cO$ using based sets. Let $\fin_\ast :=\fin^{\ast/}$ denote the (nerve of the) category of based finite sets.

\begin{definition} The $\oo$-category $\cO_\ast$ is the pullback in the following Cartesian diagram
$$\xymatrix{\cO_\ast \ar[r]\ar[d]&\cO\ar[d]\\
\fin_\ast\ar[r]&\fin\\}$$\noindent where $\fin_\ast \ra \fin$ is the forgetful functor, forgetting the distinguished nature of the basepoint $\ast$.

\end{definition}

Note that $\cO_\ast$ is acted on by $\cO$ under disjoint union, where $\cO\times\cO_\ast\ra \cO_\ast$ sends $(I, J_\ast)$ to $(I\sqcup J)_\ast$. Second, note that a symmetric monoidal functor $A:\cE \ra \cF$ makes $\cF$ an $\cE$-module. Thus, an $\cO$-algebra, $A:\cO\ra \cC$, makes $\cC$ an $\cO$-module, with the action map $\cO\times \cC \ra \cC$ given by the intuitive formula $(I, M)\squig A^{\ot I} \ot M$.

\begin{definition} For an $\cO$-algebra $A:\cO\ra \cC$, the $\oo$-category of $\cO$-$A$-modules is $\m_A^{\cO}(\cC)=\Fun_\cO(\cO_\ast,\cC)$, functors from $\cO_\ast$ to $\cC$ which are $\cO$-linear.\end{definition}

\begin{remark} If $\cO$ has a specified map from the operad $\cE_1$, so that an $\cO$-algebra can be regarded as an $\cE_1$-algebra by restriction along this map, then for an $\cO$-algebra $A$ we write $\m_A(\cC)$ for the $\oo$-category of left $A$-modules, with respect to this $\cE_1$-algebra structure on $A$. Note the distinction from $\m_A^{\cO}(\cC)$; for instance, in the case $\cO= \cE_1$, $\m_A^{\cE_1}(\cC)$ is equivalent to $A$-bimodules in $\cC$, rather than left $A$-modules.
\end{remark}

Evaluation on the point $\ast$ defines a functor $\m_A^\cO(\cC)\ra\cC$, which is the underlying object of an $\cO$-$A$-module $M$. We have a natural equivalence $M(J_\ast) \simeq A^{\ot J}\ot M$. So, applying the functor $M$ to the map $J_\ast \ra \ast$ produces a map $\cO(J_\ast) \ra \Map_{\cC}(A^{\ot J} \ot M, M)$, subject to certain compatibility conditions, and this is the usual notion of an $\cO$-$A$-module, \cite{gk}.

The collection of $\oo$-categories $\m_A^\cO(\cC)$, as $A$ varies, assembles to form an $\oo$-category $\m^{\cO}(\cC)$ of all $\cO$-algebras and their operadic modules, see \cite{dag3} or \cite{thez}, so that the following is a pullback diagram: $$\xymatrix{
\m_A^{\cO}(\cC)\ar[d] \ar[r]&\m^{\cO}(\cC)\ar[d]\\
\{A\}\ar[r] &\cO\alg(\cC)\\}$$

The structure of an $\cO$-$A$-module is equivalent to the structure of a left module for a certain associative algebra $U_A$ in $\cC$, the enveloping algebra of $A$. That is, the forgetful functor $G: \m_A^{\cO}(\cC)\ra \cC$ preserves limits and consequently has a left adjoint, $F$:

\begin{definition} $U_A = F(1_\cC)$ is the free $\cO$-$A$-module generated by the unit of $\cC$.
\end{definition}

Note that the monad structure on the composite functor $GF$ gives $U_A$ an associative algebra structure.

If $\cC$ is a stable symmetric monoidal $\oo$-category whose monoidal structure distributes over direct sums, then an $\cO$-$A$-module structure on an object $M$ is exactly the structure necessary give the direct sum $A\oplus M$ an $\cO$-algebra structure over $A$: This is the {\it split square-zero} extension of $A$ by $M$, in which the restriction of the multiplication to $M$ is trivial.

Recall that, classically, a derivation $d$ of a commutative ring $A$ into an $A$-module $M$ consists of a map $d: A\ra M$ satisfying the Leibniz rule, $d(ab)=ad(b)+bd(a)$. This can be reformulated in an enlightening way: A map $d: A \ra M$ is a derivation if and only if the map $\id + d: A \ra A\oplus M$,  from $A$ to the split square-zero extension of $A$ by $M$, is a map of commutative algebras. This reformulation allows for a general operadic notion of a derivation:

\begin{definition} Let $\cC$ be a stable presentable symmetric monoidal $\oo$-category whose monoidal structure distributes over colimits. For $M$ an $\cO$-$A$-module in $\cC$, and $B\ra A$ a map of $\cO$-algebras, then the module of $A$-derivations of $B$ into $M$ is the mapping object $$\Der (B, M) :=\Map_{\cO\alg_{/A}}(B, A\oplus M).$$
\end{definition}

Since the monoidal structure of $\cC$ is closed, then it is evident from the definition that derivations defines a bifunctor with values in $\cC$
$$\xymatrix{\Der: (\cO\alg(\cC)_{/A})^{\rm op} \times \m_A^\cO(\cC) \ar[r]& \cC}.$$ Under modest hypotheses on the $\oo$-category $\cC$, the functor of derivations out of $A$ preserves small limits. Thus, one could ask that it be corepresented by a specific $A$-module. This allows us to formulate the definition of the cotangent complex.

\begin{definition} The absolute cotangent complex of an $\cO$-algebra $A\in \cO\alg(\cC)$ consists of an $\cO$-$A$-module $L_A$ together with a derivation $d: A \ra A \oplus L_A$ such that the induced natural transformation of functors $$\xymatrix{\Map_{\m^\cO_A} (L_A , -)\ar[r]& \Der(A, -)}$$is an equivalence, where the map $\Map_{\m^\cO_A} (L_A , M) \ra \Der (A, M)$ is defined by sending a map $\alpha:L_A \ra M$ to the derivation $\alpha \circ d$.
\end{definition}

In other words, the absolute cotangent complex of $A$ is the module corepresenting the functor of $A$-derivations $\Der(A, -): \m_A^\cO(\cC) \longrightarrow \cC$. From the definition, it is direct that if $L_A$ exists, then it is unique up to a natural equivalence. We now describe this object more explicitly.

\begin{lemma} The functor $A\oplus-$ that assigns to a module $M$ the corresponding split square-zero extension $A\oplus M$, $\m_A^\cO(\cC) \ra \cO\alg(\cC)_{/A}$, is conservative and preserves small limits.
\end{lemma}
\begin{proof} As established earlier, the forgetful functor $G: \cO\alg(\cC)\ra \cC$ preserves limits, and therefore the functor $G: \cO\alg(\cC)_{/A} \ra \cC_{/A}$ is also limit preserving. This gives us the following commutative diagram
\[\xymatrix{
\m_A^\cO(\cC) \ar[d]\ar[r]^{A\oplus-}& \cO\alg(\cC)_{/A}\ar[d]\\
\cC \ar[r]^{A\times-}& \cC_{/A}\\}\] Since the bottom and vertical arrows are all limit preserving and conservative, the functor on the top must be limit preserving and conservative.
\end{proof}

\begin{prop} If $\cC$ is stable presentable symmetric monoidal $\oo$-category whose monoidal structure distributes over colimits, then the functor $\xymatrix{A\oplus-: \m_A^\cO(\cC) \ar[r]& \cO\alg(\cC)_{/A}}$ has a left adjoint, which we will denote $\L_A$.
\end{prop}
\begin{proof} Both $\m_A^\cO(\cC)$ and $\cO\alg(\cC)_{/A}$ are presentable $\oo$-categories under the hypotheses above, \cite{thez} and \cite{dag3}. The functor $A\oplus-$ is therefore a limit preserving functor between presentable $\oo$-categories. To apply the $\oo$-categorical adjoint functor theorem, \cite{topos}, it suffices to show that $A\oplus -$ additionally preserves filtered colimits. However, the forgetful functor $\cO\alg(\cC) \ra \cC$ preserves filtered colimits, see \cite{thez} or \cite{dag3}, so in both the source and target of $A\oplus -$ filtered colimits are computed in $\cC$.
\end{proof}

As a consequence we obtain the existence of the cotangent complex of $A$ as the value of the left adjoint $\L$ on $A$. In other words, since there is an equivalence $L_A \simeq \L_A({\rm id}_A)$ and the functor $\L_A$ exists, therefore the cotangent complex $L_A$ exists.

We now consider the following picture that results from an $\cO$-algebra map $f: B\ra A$.\footnote{The curved arrows are left adjoint and straight arrows are right adjoints; we maintain this convention throughout.}
\[\xymatrix{
\cO\alg(\cC)_{/\negthinspace B}\ar@/^1pc/[rr]^{f}\ar@/_1pc/[dd]_{\L_B}&&\cO\alg(\cC)_{/A}\ar[ll]^{B\times_A -}\ar@/_1pc/[dd]_{\L_A}\\
\\
\m_B^\cO(\cC)\ar[uu]_{B\oplus-}\ar@/^1pc/[rr]^{f_!}&&\m_A^\cO(\cC)\ar[ll]^{f^!}\ar[uu]_{A\oplus-}\\}\] It is evident that the compositions of right adjoints commute, i.e., that for any $\cO$-$A$-module $M$ there is an equivalence $B\times_A(A\oplus M)\simeq B\oplus f^! M$, where $f^!$ denotes the forgetful functor from $A$-modules to $B$-modules, and $f_!$ is its left adjoint, which can be computed by the relative tensor product $f_! \simeq U_A\ot_{U_B}(-)$.

As a consequence, we obtain that the value of the $\L_A$ on $f \in \cO\alg(\cC)_{/A}$ can be computed in terms of the absolute cotangent complex of $B$ and the corresponding induction functor on modules. That is, for $f: B\ra A$ an $\cO$-algebra over $A$, there is a natural equivalence of $\cO$-$A$-modules $\L_A(f) \simeq f_! L_B.$ This follows from the commutativity of the left adjoints in the above diagram, which commute because their right adjoints commute.

We now consider a relative version of the cotangent complex $\L_{A|B}$ for a map $f: B\ra A$, in which we view the $\cO$-$A$-module $L_{A|B}$ as a linear approximation to the difference between $B$ and $A$. If $L_B$ is an analogue of the cotangent bundle of a smooth manifold $M$, then $L_{A|B}$ is analogous to the bundle of cotangent vectors along the fibers of a submersion $M\ra N$. This will reduce to the case of the absolute cotangent complex already discussed when $A$ is the unit $k$ of $\cC$.

\begin{definition} For $B$ an $\cO$-algebra over $A$, the relative cotangent complex $L_{A|B}$ is an $\cO$-$A$-module corepresenting the functor of derivations $\m_A^\cO \ra \Space$ sending $M$ to the space of $B$-linear $A$-derivations from $A$ to $M$, $\Der_{A|B}(A, M) := \Map_{\cO\alg^{B/}_{/\negthinspace A}}(A, A\oplus M)$.
\end{definition}

As with the absolute cotangent complex, the relative cotangent complex $L_{A|B}$ is a value of a linearization functor $\L_{A|B}$ on the $\oo$-category of $\cO$-algebras over $A$ and under $B$. $\L_{A|B}$ is the left adjoint to the functor $\m_A^\cO(\cC)\ra \cO\alg(\cC)^{B/}_{/\negthinspace A}$ that assigns to an $\cO$-$A$-module the square-zero extension $A\oplus M$, equipped with a map from $B$ and a map to $A$. This obtains the following diagram.
\[\xymatrix{
&\cO\alg(\cC)^{B/}_{/\negthinspace A}\ar@/_1.5pc/[dl]_{A\amalg_B-}\ar@/_1.5pc/[dd]_{\L_{A|B}}\\
\cO\alg(\cC)^{A/}_{/\negthinspace A}\ar@/_1.5pc/[dr]_{\L_{A|A}}\ar[ru]\\
&\m_A^\cO(\cC)\ar[ul]\ar[uu]_{A\oplus-}\\}\] where $A\amalg_B-$ denotes the coproduct in $\cO$-algebras under $B$. So for any $C \in \cO\alg(\cC)^{A/}_{/\negthinspace B}$, the value of the relative cotangent complex on $C$ is $\L_{A|B}(C) \simeq L_{A|A}(A\amalg_B C)$. The $\cO$-$A$-module $L_{A|B}$ is obtained as the value $\L_{A|B}(A)$.

\begin{prop} There is a cofiber sequence $f_! L_B \ra L_A \ra L_{A|B}$ in the $\oo$-category of $\cO$-$A$-modules.
\end{prop}

\begin{proof} To check that $L_{A|B}$ is the cofiber of the natural map $f_! L_B \ra L_A$, it suffices to check, for any $M$ in $\cO$-$A$-modules, that $\Map_{\m_A^\cO}(L_{A|B},M)$ is the fiber of the natural map $\Map_{\m_A^\cO}(L_A, M)\ra \Map_{\m_A^\cO}(f_! L_B , M)$. Note that using that $f_!$ is the left adjoint to the forgetful functor $\m_A^\cO(\cC) \ra \m_B^\cO(\cC)$, we obtain the equivalence $\Map_{\m_A^\cO}(f_! L_B, M) \simeq \Map_{\m_B^\cO}(L_B, M) \simeq \Map_{\cO\alg_{/B}}(B, B\oplus M)$. We have thereby reduced the argument to evaluating the fiber of
$$\Map_{\cO\alg_{/\negthinspace A}}(A, A\oplus M) \ra \Map_{\cO\alg_{/B}}(B, B\oplus M)$$which is exactly $\Map_{\cO\alg^{B/}_{/\negthinspace A}}(A, A\oplus M)$.
\end{proof}

More generally, we have the following, known as the transitivity sequence: There is a natural cofiber sequence $f_! L_{B|C} \ra L_{A|C} \ra L_{A|B}$ for any sequence of $\cO$-algebras $C \ra B \xra f A$.

A particularly interesting case of the relative cotangent complex functor is that where both $A$ and $B$ are the unit $k = 1_\cC$ of $\cC$, that is, the relative cotangent complex $\L_{k|k}$ of augmented $\cO$-algebras in $\cC$. We will refer to the value value $\L_{k|k}(D) \simeq L_{k|D}[-1]$ of an augmented $\cO$-algebra $D$ the cotangent \emph{space} of the $\cO$-algebra $D$ at the point  of $D$ given by the augmentation $\epsilon: D\ra k$. This is equivalent to the case of the absolute cotangent complex of the non-unital $\cO$-algebra ${\rm Ker}(\epsilon)$, which is the $\cO$-indecomposables functor.
\medskip

We now turn to the question of describing more concretely what the cotangent complex $L_A$ actually looks like. For starters, the functor $A\oplus -: \m_A^{\cO} \ra \cO\alg_{/A}$ factors through the $\oo$-category of augmented $A$-algebras. We thus obtain a corresponding factorization of $\L_A$ through a relative cotangent complex $\L_{A|A}$. We will discuss relative cotangent complexes in more detail in the next section, but in the meantime it suffices to say that $\L_{A|A}$ is a functor from the $\oo$-category of $\cO$-algebras augmented over $A$ to $\cO$-$A$-modules fitting into the following picture.
\[\xymatrix{
\cO\alg(\cC)_{/A} \ar@/_1pc/[dr]_{\L_A}\ar@/^.75pc/[rr]^{A\amalg_{\cO} -}&& \cO\alg(\cC)_{/A}^{A/}\ar[ll]\ar@/_.75pc/[ld]_{\L_{A|A}}\\
&\m_A^\cO(\cC)\ar[ur]_{A\oplus -}\ar[ul]_{A\oplus -}\\}\]

The functor $\L_{A|A}$ is closely related to the notion of the indecomposables of a non-unital algebra. In the case of a discrete commutative non-unital ring $J$, the indecomposables ${\rm Indec} (J)$ are defined as the kernel of the multiplication map of $J$. Thus, there is a left exact sequence ${\rm Indec}(J) \ra J \ot J \ra J$. In the $\oo$-categorical setting, it is just as convenient to define the functor of indecomposables in terms of the cotangent complex. I.e., the $\cO$-indecomposables ${\rm Indec}(J)$ of a non-unital $\cO$-$A$-algebra $J$ is given as ${\rm Indec}(J) = \L_{A|A}(A\oplus J)$, where $A\oplus J$ is the split extension of $A$ by $J$.

The formula $L_A \simeq \L_{A|A}(A\amalg A) \simeq {\rm Indec}({\rm Ker}(A\amalg A \ra A))$, however, is not an especially convenient description. For instance, the coproduct $A \amalg A$ in $\cO$-algebras is potentially wild. Although the coproduct of $\cE_\infty$-algebras is very well-behaved, since it is just given by the tensor product, the coproduct of associative or $\cE_n$-algebras is more complicated. Further, the indecomposables functor ${\rm Indec}$ is similarly inconvenient, since it cannot be computed as just a kernel of a multiplication map as in the associative case.

However, in the case of $\cE_n$-algebras we will see that the composition cancels out some of this extra complication, and that for $n$ finite the $\cE_n$-cotangent complexes have a slightly simpler description not enjoyed by $\cE_\infty$-cotangent complexes.

We will now give a more explicit description of the cotangent complex in the case of a free $\cO$-algebra $A\simeq \free_{\cO}X$, which can expressed by the formula $\coprod_{k\geq 0} \cO(k)\ot_{\Sigma_k}X^{\ot k}$, see \cite{dag3} or \cite{thez}.

\begin{lemma} For $A$ a free $\cO$-algebra on an object $X$ in $\cC$, the cotangent complex of $A$ is equivalent to $U_A\otimes X$.
\end{lemma}
\begin{proof} The proof is obtained by tracing the adjunctions \[\Map_{\cO\alg_{/A}}(A, A\oplus M) \simeq \Map_{\cC_{/A}}(X, A\oplus M) \simeq \Map_\cC(X, M) \simeq \Map_{\m_A^{\cO}}(U_A\otimes X, M).\]We obtain that $U_A\otimes X$ corepresents derivations, implying the equivalence $U_A\otimes X\simeq L_A$.
\end{proof}

This reduces the problem of describing the cotangent complex of a free algebra to that of describing the enveloping algebra of a free algebra. Note that we have a functor $\psi: \cO \ra \cO_\ast$, which adds a basepoint. On morphisms, $\psi$ maps the space $\Hom_\cO(J,I)\ra \Hom_{\cO_\ast}(J_\ast, I_\ast)$ to the subspace of maps for which the preimage of $\ast$ is exactly $\ast$.

We now have the following description of the enveloping algebra $U_A$ of an $\cO$-algebra $A$.

\begin{lemma}\label{kanext}

Let $\cC$ be  a presentable symmetric monoidal $\oo$-category whose monoidal structure distributes over colimits, and let $A$ be an $\cO$-algebra in $\cC$, defined by a symmetric monoidal functor $ A: \cO \ra \cC$. Then the enveloping algebra $U_A$ in $\cC$ is equivalent to the value on $\ast$ of the left Kan extension of $ A$ along $\psi$, $U_A \simeq \psi_!  A(\ast)$.

\end{lemma}
\begin{proof}

There is a forgetful functor
$$\xymatrix{
\m^\cO(\cC)\ar[d]&\m_A^\cO(\cC)\ar[l]\ar[d]\\
\cO\alg(\cC)\times\cC \ar@/^1pc/[u]^U& \{A\} \times \cC\ar[l]\ar@/^1pc/[u] \\}$$\noindent
and the left adjoint $U$ sends the pair $(A, 1_\cC)$ to $(A, U_A)\in \m^\cO(\cC)$, where $U_A$ is the free $\cO$-$A$-module generated by $1_\cC$. This restriction functor is exactly that given by restriction along $\psi: \cO\sqcup\{\ast\} \ra \cO_*$. The left adjoint of this restriction is calculated by the Kan extension, which can be seen to linear, and therefore gives the enveloping algebra $U_A$ as the value on the basepoint.
\end{proof}

In the case where $A$ is an $\cO$-algebra in vector spaces or chain complexes, the formula for the Kan extension recovers the pointwise description of the enveloping algebra $U_A$ given in \cite{gk}. The previous lemma allows a simple expression for the enveloping algebra in the special case of free $\cO$-algebras, see Fresse in \cite{fresse}, by using the formula for a left Kan extension, $i_!A(\ast) \simeq \colim_{J\in\cO/\ast}A^{\ot J_\ast}$.

\begin{cor}\label{enveloping} Let $A$ be the free $\cO$-algebra on $X$, as above, then the universal enveloping algebra $U_A$ is equivalent to $\coprod_{n\geq 0} \cO(n+1) \ot_{\Sigma_n}X^{\ot n}$, where $\Sigma_n$ acts on $\cO(n+1)$ by an inclusion $\Sigma_n \hookrightarrow \Sigma_{n+1}$.
\end{cor}
\begin{proof} Let $\Sigma$ denote the groupoid of finite sets and bijections. The free $\cO$-algebra generated by $X$ is calculated by the coend $\cO\ot_\Sigma X$, where $X:\Sigma\ra \cC$ is the functor assigning $J\squig X^{\ot J}$, and $\cO$ is regarded as a symmetric sequence assigning $J\squig \cO(J)$. Kan extending a coend is then computed as another coend, we obtain that the enveloping algebra of the free $\cO$-algebra on $X$ is equivalent to the coend $\cO_\ast\ot_\Sigma X$, where $\cO_\ast$ is regarded as a symmetric sequence assigning $J \squig \cO(J_\ast)$. Writing out the formula for the coend, $\cO_\ast\ot_\Sigma X \simeq \colim_\Sigma \cO(J_\ast)\ot X^J \simeq \coprod_j \cO(j+1)\ot_{\Sigma_j}X^{\ot j}$.
\end{proof}

\subsection{Stabilization of $\cO$-Algebras}

In this section, we will see that the $\cO$-algebra cotangent complex is part of a more general theory of stabilization. Stabilization and costabilization are $\oo$-categorical analogues of passible to abelian group and abelian cogroup objects in ordinary categories. Since Quillen realized Grothendieck-Illusie's cotangent complex as a derived functor of abelianization (i.e., Andr\'e-Quillen homology), one would then hope that stabilization should have an analogous relation to the cotangent complex in the $\oo$-categorical setting; this is the case, as we next see.

\begin{definition} Let $\cC$ be a presentable $\oo$-category, and let $\cC_* = \cC^{\ast/}$ be the pointed envelope of $\cC$. The stabilization of $\cC$ is a stable presentable $\oo$-category $\stab(\cC)$ with a colimit-preserving functor $\Sigma^\infty : \cC_*\ra \stab(\cC)$ universal among colimit preserving functors from $\cC_*$ to a stable $\oo$-category. $\Sigma^\infty_*$ is the composite $\cC\ra \cC_* \ra \stab(\cC)$, given by first taking the coproduct with the final object, $C \rightsquigarrow C\sqcup *$, and then stabilizing.
\end{definition}

\begin{example} If $\cC$ is the $\oo$-category of spaces, then $\cC_\ast$ is pointed spaces, $\stab(\cC)$ is the $\oo$-category of spectra, and $\Sigma^\infty$ is the usual suspension spectrum functor.
\end{example}

\begin{remark} The $\oo$-category $\stab(\cC)$ can be explicitly constructed as {\it spectra} in $\cC$, \cite{dag1}.

\end{remark}

We denote the right adjoint of the stabilization functor by $\Omega^\infty$; objects in the image of $\Omega^\infty$ attain the structure of infinite loop objects in $\cC_*$, hence the notation.

The rest of this section will establish the following result on the stabilization of $\cO$-algebras. Our discussion will mirror that of \cite{dag4}, where these results are established in the commutative algebra setting.

\begin{theorem}\label{stab} Let $\cC$ be a stable presentable symmetric monoidal $\infty$-category whose monoidal structure distributes over colimits. For $A$ an $\cO$-algebra in $\cC$, the stabilization of the $\oo$-category of $\cO$-algebras over $A$ is equivalent to the $\oo$-category of $\cO$-$A$-modules in $\cC$, i.e., there is a natural equivalence
$$\stab (\cO\alg(\cC)_{/A}) \simeq \m^\cO _A (\cC)$$and equivalences of functors $\Sigma_*^{\oo}\simeq \L_A$ and $\Omega^\infty \simeq A\oplus (-)$.
\end{theorem}

\begin{remark} An equivalent result, in the case where $\cC$ is spectra, was previously proved by Basterra-Mandell in \cite{bm}.
\end{remark}

In proceeding, it will be useful to consider operadic algebras for more general operads, not in spaces. We complement our previous definition:

\begin{definition} For $\cO$ an operad in $\cC$, then $\cO\alg(\cC)$ is the full $\oo$-subcategory of $\m_\cO(\cC^\Sigma)$ consisting of left $\cO$-modules $M$ which are concentrated in degree 0 as a symmetric sequence: $M(J) = 0$ for $J\neq \O$.

\end{definition}

\begin{remark}
Under the hypotheses above, $\cC$ is tensored over the $\oo$-category of spaces: There is an adjunction $k \ot(-): \spaces \leftrightarrows \cC: \Map_\cC(k, -)$, where the left adjoint sends a space $X$ to the tensor with the unit, $k \ot X$. If $\cE$ is an operad in spaces, then $k \ot \cE$ defines an operad in $\cC$. There is then an equivalence between our two resulting notions of $\cE$ algebras in $\cC$: $\cE\alg(\cC) \simeq (k\ot \cE)\alg(\cC)$.
\end{remark}

We will require the following lemma from the Goodwillie calculus, which is a familiar fact concerning derivatives of split analytic functors. See \cite{calc3} for a further discussion of Goodwillie calculus.

\begin{lemma} Let $T$ be a split analytic functor on a stable monoidal $\oo$-category $\cC$ defined by a symmetric sequence $\cT \in \cC^\Sigma$ with $\cT(0) \simeq *$, so that $T(X) = \coprod_{n\geq 1} \cT(n)\ot_{\Sigma_n}X^{\ot n}$. The first Goodwillie derivative $DT$ is equivalent to $DT(X) \simeq \cT(1)\ot X$.
\end{lemma}
\begin{proof} We calculate the following, $$D T (X) \simeq \varinjlim \Omega^i T(\Sigma^i X) \simeq \varinjlim \Omega^i \Bigl(\coprod_{n\geq 1} \cT(n)\ot_{\Sigma_n} (\Sigma^i X)^{\ot n}\Bigr) \simeq$$
\[\varinjlim \Omega^i (\cT(1)\ot \Sigma^i X) \ \oplus \ \coprod_{n\geq 2} \varinjlim \Omega^i (\cT(n)\ot_{\Sigma_n} (\Sigma^i X)^{\ot n}),\]using the commutation of $\Omega$ with the infinite coproduct and the commutation of filtered colimits and infinite coproducts. However, we can now note that the higher terms are $n$-homogeneous functors for $n>1$, and hence they have trivial first Goodwillie derivative. This obtains that $DT(X) \simeq \varinjlim \Omega^i (\cT(1)\ot \Sigma^i X) \simeq \cT(1)\ot X$.
\end{proof}

We will now prove the theorem above in the special case where $A$ is just $k$, the unit of the monoidal structure on $\cC$. In this case, $\cO$-algebras over $A$ are literally the same as augmented $\cO$-algebras in $\cC$, $\cO\alg^{\rm aug}(\cC) \simeq \cO\alg(\cC)_{/k}$. There is an adjunction between augmented and non-unital $\cO$-algebras
$$\xymatrix{
\cO\alg^{\rm nu}(\cC)\ar@/_1pc/[d]_{k\oplus (-)}\\
\cO\alg^{\rm aug}(\cC)\ar[u]_I \\}$$where $I$ denotes the augmentation ideal functor, with left adjoint given by adjoining a unit. The adjunction above is an equivalence of $\oo$-categories, since the unit and counit of the adjunction are equivalences when $\cC$ is stable. We now formulate a special case of the theorem above. First, recall that the first term $\cO(1)$ of an operad $\cO$ has the structure of an associative algebra.

\begin{prop} There is a natural equivalence $\stab(\cO\alg^{\rm nu}(\cC)) \simeq \m_{\cO(1)}(\cC)$.
\end{prop}
\begin{proof} Let $T$ denote the monad associated to non-unital $\cO$-algebras, so that there is a natural equivalence $\cO\alg^{\rm nu}(\cC) \simeq \m_T(\cC)$. We may thus consider stabilizing this adjunction, to produce another adjunction:
$$\xymatrix{
\m_T(\cC)\ar[d] \ar@/^.7pc/[r]^{\Sigma^\infty} & \stab(\m_T\cC)\ar[d]^g\ar[l]^{\Omega^\infty} \ar[r] &\m_{gf}(\cC)\\
\cC\ar@/^1pc/[u] \ar[r]^\simeq&\stab(\cC)\ar@/^1pc/[u]^f\\}$$The stabilization of $\m_T(\cC)$ is monadic over $\cC$, \cite{dag4}, since the right adjoint is conservative, preserves split geometric realizations, and hence satisfies the $\oo$-categorical Barr-Beck theorem. The resulting monad $g\circ f$ on $\cC$ is the first Goodwillie derivative of $T$, which by the above lemma is computed by $\cO(1)\ot(-)$, with the monad structure of $g\circ f$ corresponding to the associative algebra structure on $\cO(1)$. Thus, the result follows.
\end{proof}

Note that if the operad $\cO$ is such that  $\cO(1)$ is equivalent to the unit of the monoidal structure, then there is an equivalence $\m_{\cO(1)}(\cC)\simeq \cC$, so the theorem then reduces to the statement of the equivalence $\stab(\cO\alg)\simeq \cC$. In particular, the functor $\ind_\eta$ of induction along the augmentation $\eta: \cO\ra 1$ is equivalent to the stabilization functor $\Sigma^\infty$. 

To complete the proof of the main theorem, we will reduce it to the proposition above. Consider $\cO_A$, the universal enveloping operad of $A$, defined by the property that $\cO_A\alg(\cC)$ is equivalent to $\cO$-algebras under $A$. The existence of $\cO_A$ is assured by the existence of the left adjoint to the forgetful functor $\cO\alg(\cC)^{A/} \ra \cC$; $\cO_A$ can be explicitly constructed as the Boardman-Vogt tensor product of $\cO$ and $U_A$. Likewise, we have that non-unital $\cO$-$A$-algebras is equivalent to non-unital $\cO_A$-algebras. Since the $\oo$-category of $\cO$-algebras augmented over $A$ is again equivalent to $\cO_A\alg^{\rm nu}(\cC)$, we reduce the argument to considering this case.

Thus, we obtain that $\stab(\cO_A\alg^{\rm nu}(\cC))$ is equivalent to $\m_{\cO_A(1)}(\cC)$. Since the first term of the enveloping operad $\cO_A(1)$ is equivalent to the enveloping algebra $U_A$, and $\m_{U_A}(\cC) \simeq \m_A^\cO(\cC)$, this implies that the equivalence $\stab(\cO_A\alg^{\rm nu}(\cC)) \simeq \m_A^\cO(\cC)$.

By definition, the stabilization of an unpointed $\oo$-category $\cX$ is the stabilization of its pointed envelope $\cX_*$, the $\oo$-category of objects of $\cX$ under $*$, the final object. Thus the pointed envelope of $\cO$-algebras over $A$ is $\cO$-algebras augmented over and under $A$. This is the stabilization we have computed, which completes our proof of the theorem.

\subsection{The $\cE_n$-Cotangent Complex}

We now specialize to the case of $\cO$ an $\cE_n$ operad, for $n< \oo$, in which case a certain splitting result further simplifies the description of the enveloping algebra of a free algebra in Corollary \ref{enveloping}. First, we briefly review some of the geometry of the configuration spaces $\cE_n(k)$. The map $\cE_n(k+1) \ra \cE_n(k)$, given by forgetting a particular $n$-disk, is a fiber bundle with fibers given by configurations of a disk in a standard disk with $k$ punctures, which is homotopy equivalent to a wedge of $k$ copies of the $(n-1)$-sphere. A standard fact is that, after suspending, this fiber bundle splits: $\Sigma\cE_n(k+1) \simeq \Sigma( \cE_n(k)\times \bigvee_{k}S^{n-1}$). Iterating, there is then a stable equivalence between the space $\cE_n(k+1)$ and the product $\prod_{1\leq j \leq k}\bigvee_j S^{n-1}$. The map $\cE_n(k+1)\ra \cE_n(k)$ is equivariant with respect to the action of $\Sigma_k$ on both sides, so one can ask that this splitting be arranged so as to be equivariant with respect to this action. The following lemma can be proved either directly, by explicit analysis of the equivariant splittings of configuration spaces, or as a consequence of McDuff's theorem in \cite{mcduff}:

\begin{lemma}\label{splitting} There is a $\Sigma_k$-equivariant splitting \[\Sigma^\infty_*\cE_n(k+1)\simeq\coprod_{i+j = k} \ind_{\Sigma_i}^{\Sigma_k}(\Sigma^{(n-1)j}\Sigma^\infty_*\cE_n(i)),\] where $\ind_{\Sigma_i}^{\Sigma_k}$ is induction from $\Sigma_i$-spectra to $\Sigma_k$-spectra\end{lemma}\qed

This has the following consequence. Again, assume $\cC$ is a stable presentable symmetric monoidal $\oo$-category whose monoidal structure distributes over colimits. Denote by $\free_{\cE_1}$ the free $\cE_1$-algebra functor.

\begin{prop} \label{envfree} Let $A$ be the free $\cE_n$-algebra on an object $X$ in $\cC$. There is a natural equivalence $$U_A \simeq A \ot \free_{\cE_1}(X[n-1]).$$
\end{prop}
\begin{proof}

By Corollary \ref{enveloping}, the enveloping algebra $U_A$ is equivalent to $\coprod \cE_n(k+1) \ot_{\Sigma_k} X^{\ot k}$. By the description of the spaces $\cE_n(k+1)$ in Lemma \ref{splitting}, we may rewrite this as \[\coprod_k \cE_n(k+1)\ot_{\Sigma_k}X^{\ot k} \simeq \coprod_k \coprod_{i+j =k}\ind_{\Sigma_i}^{\Sigma_k}\Sigma^\infty_* \cE_n(i)[(n-1)j]\ot_{\Sigma_k} X^{\ot k}\]
\[\simeq \coprod_k \coprod_{i+j =k}(\cE_n(i)\ot_{\Sigma_i}X^{\ot i})[(n-1)j]\ot X^{\ot j}\simeq \coprod_k \coprod_{i+j=k} (\cE_n(i)\ot_{\Sigma_i}X^{\ot i})\ot X[n-1]^{\ot j}\]
\[\simeq \biggl(\coprod_i \cE_n(i)\ot_{\Sigma_i} X^{\ot i}\biggr) \otimes \biggl(\coprod_j X[n-1]^{\ot j}\biggr)\simeq A \ot \free_{\cE_1}(X[n-1]).\]
\end{proof}

This brings us to the main result of this section, which in the stable setting gives a description of the absolute cotangent complex of an $\cE_n$-algebra.

\begin{theorem}\label{main1} Let $\cC$ be a stable presentable symmetric monoidal $\oo$-category whose monoidal structure distributes over colimits. For any $\cE_n$-algebra $A$ in $\cC$, there is a natural cofiber sequence
$$U_A \longrightarrow A \longrightarrow L_A[n]$$in the $\oo$-category of $\cE_n$-$A$-modules.
\end{theorem}

\begin{remark} This result has a more familiar form in the particular case of $\cE_1$-algebras, where the enveloping algebra $U_A$ is equivalent to $A\ot A^{\op}$. The statement above then becomes that there is a homotopy cofiber sequence $L_A \ra A\ot A^{\op} \ra A$, which is a description of the associative algebra cotangent complex dating back to Quillen for simplicial rings and Lazarev \cite{lazarev} for $A_\infty$-ring spectra.
\end{remark}

\begin{proof}

We will prove the theorem as a consequence of an equivalent statement formulated in terms of the $\oo$-category of all $\cE_n$-algebras and their $\cE_n$-modules, $\m^{\cE_{\!n}}(\cC)$. We first define the following functors, $L$, $U$, and $\imath$, from $\cE_n\alg(\cC)$ to $\m^{\cE_{\!n}}(\cC)$: $L$ is the cotangent complex functor, assigning the pair $(A, L_A)$ in $\m^{\cE_{\!n}}(\cC)$ to an $\cE_n$-algebra $A$. $U$ is the composite $\cE_n\alg(\cC) \times \{1\}\ra \cE_n\alg(\cC) \times \cC \ra \m^{\cE_{\!n}}(\cC)$, where the functor $\cE_n\alg(\cC)\times \cC \ra \m^{\cE_{\!n}}(\cC)$ sends an object $(A, X)\in \cE_n\alg(\cC)\times\cC$ to $(A, U_A\ot X)$, the free $\cE_n$-$A$-module generated by $X$. Finally, the functor $\imath: \cE_n\alg(\cC)\ra \m^{\cE_{\!n}}(\cC)$ sends $A$ to the pair $(A,A)$, where $A$ is regarded as an $\cE_n$-$A$-module in the canonical way. We will now show that there is a cofiber sequence of functors $U\ra \imath\ra \Sigma^n L$.

The first map in the sequence can be defined as follows. Denote $\m^{\cE_n,\ell}(\cC)$ the $\oo$-category of $\cE_n$-algebras with left modules in $\cC$: I.e., an object of $\m^{\cE_n,\ell}(\cC)$ roughly consists of a pair $(A, K)$ of an $\cE_n$-algebra $A$ and a left $A$-module $K$. Every $\cE_n$-module has a left module structure by choice of a 1-dimensional subspace of $\RR^n$, so we have a forgetful functor $\m^{\cE_{\!n}}(\cC) \ra \m^{\cE_n,\ell}(\cC)$. This functor has a left adjoint in which, for fixed $\cE_n$-algebra $A$, there is an adjunction $F:\m_A(\cC) \leftrightarrows \m_A^\cO(\cC):G$; the left adjoint $F$ is computed as the bar construction $U_A \ot_A(-)$. Consequently, the counit of this adjunction, $FG\ra {\rm id}$, applied to $A\in \m_A^{\cE_{\!n}}(\cC)$, gives the desired map $U_A\simeq U_A\ot_A A \simeq FG(A)\ra A$. The functoriality of the counit map thus defines a natural transformation of functors $U\ra \imath$. We will identify $\Sigma^n L$ as the cokernel of this map. We first prove this in the case that $A$ is the free algebra on an object $X$, so that we have $A\simeq \coprod \cE_n(i)\ot_{\Sigma_i} X^{\ot i}$ and $U_A \simeq \coprod \cE_n(i+1)\ot_{\Sigma_i} X^{\ot i}$. The map $U_A \ra A$ defined above is concretely realized by the operad structure maps $\cE_n(i+1) \xra {\circ_{i+1}} \cE_n(i)$ given by plugging the $i+1$ input of $\cE_n(i+1)$ with the unit of $\cC$. The map $\circ_{i+1}$ is $\Sigma_i$-equivariant, since it respects the permutations of the first $i$ inputs of $\cE_n(i+1)$, so this gives an explicit description of the map
$$U_A \simeq \coprod_{i\geq 0} \cE_n(i+1)\ot_{\Sigma_i}X^{\ot i} \longrightarrow \coprod_{i\geq 0} \cE_n(i)\ot_{\Sigma_i} X^{\ot i} \simeq A.$$Using the previous result that $U_A \simeq A \ot \free_{\cE_1}(X[n-1])$, we may rewrite this as
\[\biggl(\coprod_{j\geq 0} \cE_n(j) \ot_{\Sigma_j} X^{\ot j} \biggr)\ot \biggl(\coprod_{k\geq 0} X[n-1]^{\ot k}\biggr) \simeq \coprod_{i\geq 0} \cE_n(i+1) \ot_{\Sigma_i}X^{\ot i} \longrightarrow \coprod_{i\geq 0} \cE_n(i)\ot_{\Sigma_i} X^{\ot i}.\]The kernel of this map exactly consists of the direct sum of all the terms $\cE_n(j)\ot_{\Sigma_j}X^{\ot j}\ot X[n-1]^{\ot k}$ for which $k$ is greater than zero. So we obtain a fiber sequence
\[\xymatrix{\biggl(\coprod_{j\geq 0} \cE_n(j) \ot_{\Sigma_j} X^{\ot j}\biggr) \ot \biggl(\coprod_{k\geq 1} X[n-1]^{\ot k}\biggr) \ar[d]\ar[r]&\biggl(\coprod_{j\geq 0} \cE_n(j) \ot_{\Sigma_j} X^{\ot j}\biggr) \ot \biggl(\coprod_{k\geq 0} X[n-1]^{\ot k}\biggr)\ar[d]\\
 0\ar[r]&\coprod_{i\geq 0} \cE_n(i)\ot_{\Sigma_i} X^{\ot i}\\}\] of $\cE_n$-$A$-modules. It is now convenient to note the equivalence $\coprod_{k\geq 1} X[n-1]^{\ot k} \simeq X[n-1] \ot \coprod_{k\geq 0} X[n-1]^{\ot k}$. That is, the fiber in the sequence above is equivalent to $U_A \ot X[n-1]$. Thus, whenever $A$ is the free $\cE_n$-algebra on an object $X$ of $\cC$, we obtain a fiber sequence
$$U_A \ot X[n-1] \longrightarrow U_A \longrightarrow A.$$However, we can recognize the appearance of the cotangent complex, since we saw previously that the cotangent complex of a free algebra $A$ is equivalent to $U_A \ot X$. Thus, we now obtain the statement of the theorem, that there is a fiber sequence $L_A[n-1] \ra U_A \ra A$, in the special case where $A$ is a free $\cE_n$-algebra.

We now turn to the general case. Denote the functor $J: \cE_n\alg(\cC) \ra \m^{\cE_{\!n}}(\cC)$ defined objectwise as the cokernel of the map $U\ra \imath$. We will show that the functor $J$ is colimit preserving, a property which we will then use to construct a map from $L$ to $J$. To show a functor preserves all colimits, it suffices to verify the preservation of geometric realizations and coproducts. Since geometric realizations commute with taking cokernels, we may show that $J$ preserves geometric realizations by showing that both the functor $U$ and $\imath$ preserve them. 

First, consider the functor $U$: The inclusion $\cE_n\alg(\cC)\ra \cE_n\alg\times \cC$ preserves geometric realizations; additionally, the free $\cE_n$-$A$-module functor $\cE_n\alg\times\cC \ra \m^{\cE_{\!n}}(\cC)$ is a left adjoint. $U$ is thus the composite of a left adjoint and a functor that preserves geometric realizations, hence $U$ preserves geometric realizations. Secondly, consider the functor $\imath$. Given a simplicial object $A_\bullet$ in $\cE_n\alg(\cC)$, the realization of $|\imath A_\bullet|$ is equivalent to $(|A_\bullet|, |U_A\ot_{U_{A_\bullet}}A_\bullet|)$. We now use the general result: For $R_\bullet$ a simplicial algebra, $M_\bullet$ an $R_\bullet$-module, and $R_\bullet \ra S$ an algebra map, then there is an equivalence $|S\ot_{R_\bullet}M_\bullet|\simeq S\ot_{|R_\bullet|} |M_\bullet|$. Applying this in our example gives that $|U_A\ot_{U_{A_\bullet}}A_\bullet|$ is equivalent to $U_A\ot_{|U_{A_\bullet}|}|A_\bullet|$. The geometric realization $|U_{A_\bullet}|$ is equivalent to $U_A$, since by the description of $U_A$ as a left Kan extension it preseves these geometric realizations. Thus, we obtain that $\imath$ does preserve geometric realizations and as a consequence $J$ does as well.

Now, we show that $J$ preserves coproducts. First, if a functor $F:\cE_n\alg(\cC) \ra \cD$ preserves geometric realizations and coproducts of free $\cE_n$-algebras, then $F$ also preserves arbitrary coproducts. We see this as follows: Let $A_i$, $i\in I$, be a collection of $\cE_n$-algebras in $\cC$, and let $C_\bullet A_i$ be the functorial simplicial resolution of $A_i$ by free $\cE_n$-algebras, where $C_n A_i := \free_{\cE_{\!n}}^{\circ (n+1)}(A_i)$. Since geometric realizations commute with coproducts, there is a natural equivalence of $F(\coprod_I A_i) \simeq F(\coprod_I |C_\bullet A_i|)$ with $F(|\coprod_I C_\bullet A_i|)$. Applying our assumption that $F$ preserves coproducts of free algebras and geometric realizations, we thus obtain equivalences $$F\Bigl|\coprod_I C_\bullet A_i\Bigr|\simeq \Bigl|\coprod_I F(C_\bullet A_i)\Bigr|\simeq \coprod_I F(|C_\bullet A_i|)\simeq \coprod_I F(A_i)$$\noindent
where the second equivalence again follows from $F$ preserving geometric realizations. Thus, we obtain that $F$ preserves arbitrary coproducts given the previous assumption. We now demonstrate that $J$ preserves coproducts of free $\cE_n$-algebras, which will consequently imply that $J$ preserves all colimits. Note that the functor $L$ is a left adjoint, hence it preserves all colimits. We showed, above, that for free algebras $A = \free_{\cE_{\!n}}(X)$, there is an equivalence $J(A) \simeq L_A[n]$. Let $\{A_i\}$ be a collection of free $\cE_n$-algebras; since the coproduct of free algebras is again a free algebra, we obtain that $J(\coprod_I A_i) \simeq L_{\coprod_I A_i}[n] \simeq \coprod_I L_{A_i}[n] \simeq \coprod_I J(A_i)$. Thus, $J$ preserves coproducts of free algebras, hence $J$ preserves all colimits.

The universal property of the cotangent complex functor $L$ proved in Theorem \ref{stab} now applies to produce a map from $L$ to $J$: The stabilization functor $\L_A: \cE_n\alg(\cC)_{/A} \ra \m_A^{\cE_{\!n}}(\cC)$, from Theorem \ref{stab}, has the property that for any colimit preserving functor $F$ from $\cE_n\alg(\cC)_{/A}$ to a stable $\oo$-category $\cD$, there exists an essentially unique functor $F': \m_A^{\cE_{\!n}}(\cC) \ra \cD$ factorizing $F'\circ \L_A \simeq F$. Choose $F$ to be the composite $J_A: \cE_n\alg(\cC)_{/A} \ra \m^{\cE_{\!n}}(\cC)_{/A} \ra \m_A^{\cE_{\!n}}(\cC)$, where the first functor is $J$ and the second functor sends a pair $(B\xra f A, M)$, where $M$ is an $\cE_n$-$B$-module, to the $\cE_n$-$A$-module $U_A\ot_{U_B}M$. $J_A$ preserves colimits, since it is a composite of two functors each of which preserve colimits. The universal property now applies to show that there is an equivalence of functors $\jmath\circ \L_A \simeq J_A$, for some colimit preserving functor $\jmath$. However, we have shown there is also an equivalence $J_A(B)\simeq \L_A(B)[n]$ whenever $B$ is a free $\cE_n$-algebra. Since cotangent complexes of free algebras generate $\m_A^{\cE_{\!n}}(\cC)$ under colimits, we may conclude that the functor $\jmath$ is therefore the $n$-fold suspension functor. Thus, we obtain the equivalence of functors $J_A \simeq \Sigma^n\L_A$. Since this equivalence holds for every $A$, we finally have an equivalence of functors $J \simeq \Sigma^nL$ and a cofiber sequence of functors $U\ra \imath \ra \Sigma^n L$.
\end{proof}

One may think of the result above as saying that the shifted $\cE_n$-$A$-module $A[-n]$ is very close to being the cotangent complex of $A$. There is an interesting interpretation of the difference between the functors corepresented by $A[-n]$ and $L_A$, however, which we discuss later in this paper.

\begin{remark} Since the preceding proof was written in \cite{thez}, both Lurie and I separately realized that a more conceptual proof of this theorem is possible in terms of the theory of higher categories. Lurie's proof is in \cite{dag6}, and we present a closely related proof later in this paper. I have still included this proof, however, since its nuts-and-bolts character offers a complementary understanding.
\end{remark}

Let us apply the previous analysis of the absolute cotangent complex in the $\cE_n$ setting to obtain a similar description of the cotangent space of an augmented $\cE_n$-algebra $A$.

\begin{cor}\label{pointversion} Let $A$ be an augmented $\cE_n$-algebra in $\cC$, as above, with augmentation $f: A\ra k$. Then there exists a cofiber sequence in $\cC$, $k \ra k\ot_{U_A} A \ra L_{k|A}[n-1]$, where $L_{k|A}$ is the relative cotangent complex of $f$.
\end{cor}
\begin{proof}

Recall from the previous theorem the cofiber sequence $U_A \ra A \ra L_A[n]$ of $\cE_n$-$A$-modules. Given an $\cE_n$-ring map $f: A \ra B$, we can apply the induction functor to obtain $U_B \ra f_! A \ra f_! L_A[n]$, a cofiber sequence of $\cE_n$-$B$-modules. Specializing to where $f:A \ra k$ is the augmentation of $A$, this cofiber sequence becomes $k \ra f_! A \ra f_! L_A [n]$. Note that since there is an equivalence between $\cE_n$-$k$-modules in $\cC$ and $\cC$ itself, the enveloping algebra of the unit $k$ is again equivalent to $k$. So we have an equivalence $f_! A \simeq k\ot_{U_A} A$.

Finally, we can specialize the cofiber sequence $f_! L_{A|k} \ra L_{B|k} \ra L_{B|A}$ to the case of $B=k$, to obtain a cofiber sequence $f_! L_{A} \ra L_{k|k} \ra L_{k|A}$. Since $L_{k|k}$ is contractible, this gives an equivalence $L_{k|A}[-1] \simeq f_! L_A$. Substituting into $k \ra k\ot_{U_A}A \ra f_! L_A [n]$, we obtain a cofiber sequence $k \ra k\ot_{U_A}A \ra L_{k|A}[n-1]$ as desired.
\end{proof}

\begin{remark} The object $k\ot_{U_A} A$ may be thought as the infinitesimal $\cE_n$-Hochschild homology of $A$, or the $\cE_n$-Hochschild homology with coefficients in the augmentation, i.e., $k\ot_{U_A}A =: \hh_*^{\cE_{\!n}}(A, k)$. This result is then saying that, modulo the unit, the cotangent space is equivalent to a shift of the infinitesimal $\cE_n$-Hochschild homology. In the case $n=1$ of usual algebra, the enveloping algebra $U_A$ is equivalent to $A\ot A^{\op}$, and we have the chain of equivalences $\hh_*^{\cE_1}(A,k) = k\ot_{A\ot A^{\op}}A \simeq k\ot_A\ot A\ot_{A^{\op}} k \simeq k\ot_A k$, so that the infinitesimal Hochschild homology of an algebra is given by the bar construction: $\hh_*^{\cE_1}(A,k) \simeq k\ot_A k$.
\end{remark}

\subsection{Structure of the Cotangent Complex and Koszul Duality}

Until this point, we have studied the cotangent complex solely as an object of the stable $\oo$-category, such as $\m_A^{\cO}(\cC)$ or $\cC$. However, one may also ask what structure $f_!L_B$ obtains by the fact that it is born as a linear approximation to a map $f:A\ra B$. For instance, taking as geometric motivation the case of a submersion $M\ra N$, we have that the bundle of tangents along the fibers $T_{M|N}$ has the structure of a Lie algebroid on $M$. Before proceeding, we first provide the obvious notion of the $\cO$-tangent complex. Note that our conditions on $\cC$ imply that $\m_A^{\cO}(\cC)$ is tensored and enriched over $\cC$, which allows the following definition.

\begin{definition} The relative tangent complex $T_{A|B}\in\cC$ of an $\cO$-algebra map $B\ra A$ is the dual of $L_{A|B}$ as an $\cO$-$A$-module: $T_{A|B} = \Hom_{\m_A^{\cO}}(L_{A|B}, A)$.
\end{definition}

We will prove the following, to give a sense of a direction of this section:

\begin{prop}\label{koszul} Let $\cO$ be an augmented operad in $\cC$ with $\cO(1)\simeq k$, the unit of $\cC$. Then the tangent space of an augmented $\cO$-algebra naturally defines a functor $$\cO\alg^{\rm aug}(\cC)^{\op}\longrightarrow \cO^!\alg^{\rm nu}(\cC)$$ where $\cO^!$ is the derived Koszul dual operad of $\cO$.
\end{prop}

This will be largely an application of $\oo$-categorical Barr-Beck thinking, \cite{dag2}. For simplicity, we begin with the case of augmented $\cO$-algebras (and for which $\cO(1)$ is the unit, which we will henceforth assume). For an augmented $\cO$-algebra $\epsilon: A\ra k$, we will denote the tangent space at the augmentation by $TA:=\Hom_{\cC}(\epsilon_! L_A, k)$, and refer to $TA$ simply as the tangent {\it space} of $A$ at $\epsilon$ (i.e., at the $k$-point $\Spec k \ra \Spec A$, in the language of section 4).

It is convenient to now use a slightly different description of operads and their algebras. Let $\cC$ be a symmetric monoidal presentable $\oo$-category for which the monoidal structure distributes over colimits. Recall the $\oo$-category of symmetric sequences $\cC^\Sigma$ (i.e., functors from finite sets with bijections to $\cC$): There is a functor $\cC^\Sigma \ra \Fun(\cC,\cC)$ given by assigning to a symmetric sequence $X$ the endofunctor $C \rightsquigarrow \coprod_{i\geq}X(i)\ot_{\Sigma_i}C^{\ot i}$. There is a monoidal structure on $\cC^\Sigma$ agreeing with the composition of the associated endofunctor, so that the preceding functor $\cC^\Sigma\ra \Fun(\cC,\cC)$ is monoidal. Operads are exactly associative algebras in $\cC^\Sigma$ with respect to this monoidal structure.

We refer to \cite{thez} for the following, which relies on a description of free algebras in a monoidal category due to Rezk in \cite{rezk}. Let $\cX$ be a monoidal $\oo$-category for which the monoidal structure distributes over geometric realizations and left distributes over colimits.

\begin{prop} For $\cX$ as above, the bar construction defines a functor $\Alg^{\rm aug}(\cX) \ra \Coalg^{\rm aug}(\cX)$, sending an augmented algebra $A$ to $\ba A = 1_\cX\ot_A 1_\cX$, the geometric realization of the two-sided bar construction $\ba(1_\cX,A,1_\cX)$.

\end{prop}

The conditions on $\cX$ are satisfied for symmetric sequences $\cC^\Sigma$ equipped with the composition monoidal structure (which, it is worthwhile to note, does not distribute over colimits on the {\it right}). The following was first proved in the setting of model categories in \cite{ching} when $\cC$ is spectra.

\begin{cor} Let $\cC$ be a symmetric monoidal $\oo$-category, which stable and presentable, and for which the monoidal structure distributes over colimits. Then the bar construction $\ba: \Alg^{\rm aug}(\cC^\Sigma)\ra \Coalg^{\rm aug}(\cC^\Sigma)$ defines a functor from augmented operads in $\cC$ to augmented cooperads in $\cC$, i.e., coaugmented coalgebras for the composition monoidal structure on $\cC^\Sigma$.
\end{cor}

\begin{definition} For an operad $\cO\in \Alg^{\rm aug}(\cC^\Sigma)$, as above, the derived Koszul dual operad $\cO^!$ is the dual of the cooperad given by the bar construction $\ba \cO$: I.e., $\cO^! = (\ba \cO)^\vee$.

\end{definition}

We now apply this to our study of $\cO$-algebras. For a unital operad $\cO$, let $\cO^{\rm nu}$ denote the associated operad without degree zero operation, so that there is an equivalence $\cO^{\rm nu}\alg(\cC) \simeq \cO\alg^{\rm nu}(\cC)$.

\begin{lemma} For $\epsilon: A\ra k$ an augmented $\cO$-algebra in $\cC$, with augmentation ideal $I_A$, there is an equivalence $\epsilon_! L_A \simeq 1\circ_{\cO^{\rm nu}} A :=|\ba(1,\cO^{\rm nu},I_A)|$ between the cotangent space of $A$ at the $k$-point $\epsilon$ and the two-sided bar construction of $\cO^{\rm nu}$ with coefficients in the left $\cO$-module $I_A$ and the unit symmetric sequence 1, regarded as a right $\cO^{\rm nu}$-module.
\end{lemma}
\begin{proof} For a map of operads $\cP\ra \cQ$, the bar construction $\cQ\circ_\cP(-)$ computes the left adjoint to the restriction $\cQ\alg(\cC)\ra \cP\alg(\cC)$. (See \cite{thez} for a discussion of this fact in the $\oo$-category setting.) Applying this to $\cQ=1$ the unit symmetric sequence, we find that $1\circ_{\cO^{\rm nu}}(-)$ computes the left adjoint to the functor $\cC \ra \cO^{\rm nu}\alg(\cC)$ assigning an object of $\cC$ the trivial $\cO$-algebra structure. Thus, the cotangent space functor and the bar construction are both left adjoints to equivalent functors, hence they are equivalent.
\end{proof}

We now have the following picture:

\begin{cor}\label{Okoszul} The cotangent space $\epsilon_! L_A$ of an augmented $\cO$-algebra $A$ naturally has the structure of an $1\circ_\cO 1$-comodule in $\cC$. That is, there is a commutative diagram:
$$\xymatrix{
\cO\alg^{\rm aug}(\cC)\ar[dr]_{\L}\ar[rr]^{1\circ_\cO-}&&\comod_{1\circ_\cO1}(\cC)\ar[dl]^{\rm forget}\\
&\cC\\}$$
\end{cor}
\begin{proof} The comonad underlying $1\circ_\cO1$ is that associated to the adjunction between $\L$ and the trivial functor, so every object $1\circ_\cO A$ obtains a left $1\circ_\cO1$-comodule structure.
\end{proof}

\begin{remark} Left comodules in $\cC$ for the cooperad $1\circ_\cO1$ form a type of coalgebra for the cooperad $1\circ_\cO1$. However, there are two important distinctions between these objects and usual coalgebras (i.e., $1\circ_\cO1$-algebras in $\cC^{\op}$): These objects are automatically ind-nilpotent coalgebras, and they have an extra structure, analogous to divided power maps. Thus, $1\circ_\cO1$-comodules could instead be termed ind-nilpotent $1\circ_\cO1$-coalgebras with divided powers, as they are in \cite{fg}.
\end{remark}

And the dual of a coalgebra is an algebra:

\begin{lemma} Dualizing defines a functor $\comod_{1\circ_\cO1}(\cC) \ra \cO^!\alg^{\rm nu}(\cC)$.

\end{lemma}
\begin{proof} For a $1\circ_\cO1$-comodule $C$, there is a map $C\ra\coprod_{i\geq 1} (1\circ_\cO1)(i)\ot_{\Sigma_i} C^{\ot i}$. Dualizing gives a map $\prod_{i\geq 1}\cO^!(i)\ot^{\Sigma_i}(C^{\ot i})^\vee \ra C^\vee$. We have a composite map $$\coprod_{i\geq 1}\cO^!(i)\ot_{\Sigma_i}(C^\vee)^{\ot i} \longrightarrow\prod_{i\geq 1}\cO^!(i)\ot^{\Sigma_i}(C^\vee)^{\ot i} \longrightarrow \prod_{i\geq 1}\cO^!(i)\ot^{\Sigma_i}(C^{\ot i})^\vee \longrightarrow C^\vee$$ where the map $\cO^!(i)\ot_{\Sigma_i}(C^\vee)^{\ot i} \ra \cO^!(i)\ot^{\Sigma_i}(C^\vee)^{\ot i}$, from the $\Sigma_i$-invariants of the diagonal action to the $\Sigma_i$-coinvariants of the action, is the norm map. This gives $C^\vee$ a nonunital $\cO^!$-algebra structure.
\end{proof}
\begin{remark} The dual of a $1\circ_\cO1$-comodule actually obtains more structure than just that of an $\cO^!$-algebra: The factorization of the action maps $\cO^!(i)\ot_{\Sigma_i}(C^\vee)^{\ot i} \ra C^\vee$ through the norm map is an $\cO$-analogue of a divided power structure on a commutative algebra, or a restricted structure on a Lie algebra. Thus, the tangent space of an $\cO$-algebra should be a pro-nilpotent restricted $\cO^!$-algebra. See \cite{fresse} for an extended treatment of this structure specific to simplicial algebra. However, in the particular case of the spaces $\cE_n(i)$, for $n$ finite, the above norm map is actually a homotopy equivalence: This is a consequence of the fact that $\cE_n(i)$ are finite CW complexes with a free action of $\Sigma_i$. Thus, one does not obtain any extra restriction structure in the case of $\cE_n$, our case of interest, and so we ignore this extra structure for the present work.
\end{remark}

We now restrict to the special case of $\cE_n$, in which something special happens: The $\cE_n$ operad is Koszul self-dual, up to a shift. That is, there is an equivalence of operads in spectra, $\cE_n^!\simeq \cE_n[-n]$. Unfortunately, a proof of this does exist in print. That this is true at the level of homology dates to Getzler-Jones, \cite{getzlerjones}, and a proof at the chain level has recently been given by Fresse, \cite{fressekoszul}: Fresse shows that there is an equivalence $C_\ast(\cE_n,\FF)[-n] \simeq \Tot[{\rm Cobar}(C^\ast(\cE_n,\FF))]$. In chain complexes, we can therefore apply Fresse's theorem to obtain our next result. Lacking a direct calculation of the operad structure on $\cE_n^!$ in full generality to feed into Corollary \ref{Okoszul}, we will produce the following more directly: 

\begin{theorem}\label{koszulen} The $\cE_n$-tangent space $T$ defines a functor $$\cE_n\alg^{\rm aug}(\cC)^{\op}\longrightarrow \cE_n[-n]\alg^{\rm nu}(\cC).$$
\end{theorem}

The key input to this construction will be a theorem of Dunn in \cite{dunn}, upgraded by Lurie to the $\oo$-category context in \cite{dag6}, roughly saying that $\cE_n$-algebra is equivalent to $n$-times iterated $\cE_1$-algebra:

\begin{theorem}[\cite{dunn}, \cite{dag6}]\label{dunn} Let $\cX$ be a symmetric monoidal $\oo$-category. Then there is a natural equivalence $\cE_{m+1}\alg(\cX) \simeq \cE_1\alg(\cE_{m}\alg(\cX))$. Iterating, there is an equivalence $\cE_1\alg^{(n)}(\cC)\simeq \cE_n\alg(\cC)$ for all $n$ and any symmetric monoidal $\oo$-category $\cC$.

\end{theorem}

\begin{remark} One can derive intuition for this result from the theory of loop spaces. An $(n+1)$-fold loop space $\Omega^{n+1}X$ is precisely the same thing as a loop space in $n$-fold loop spaces, $\Omega^{n+1}X = \Omega(\Omega^nX)$. By May's theorem in \cite{may}, $(n+1)$-fold loop spaces are equivalent to grouplike $\cE_{n+1}$-algebras in spaces. Applying May's theorem twice, with the observation above, we obtain that $(n+1)$-fold loop spaces are also equivalent to the subcategory of grouplike objects of $\cE_1\alg(\cE_n\alg(\spaces))$. As a consequence, we have an equivalence $\cE_{n+1}\alg(\spaces)^{\rm gp}\simeq \cE_1\alg(\cE_n\alg(\spaces))^{\rm gp}$, due to their shared equivalence to $n$-fold loop spaces. This is almost a proof of the theorem: If one could remove the condition of being grouplike, then the result for spaces would imply it for general $\cX$. 
\end{remark}

Let $\ba^{(n)}$ denote the $n$-times iterated bar construction, which defines a functor $\cE_1\alg^{(n)}_{\rm aug}(\cC)\ra \cE_1\coalg^{(n)}_{\rm aug}(\cC)$ from $n$-times iterated $\cE_1$-algebra to $n$-times iterated $\cE_1$-coalgebas. Dunn's theorem has the following corollary:

\begin{cor} The dual of $\ba^{(n)}$ defines a functor $\cE_n\alg^{\rm aug}(\cC)^{\op} \ra \cE_n\alg^{\rm aug}(\cC)$.

\end{cor}

To establish Theorem \ref{koszulen}, we are now only required to do a calculation to show that the dual of the bar construction above calculates the tangent space (modulo the unit and after a shift). Our computational input will be the following:

\begin{lemma}\label{bar} There is an equivalence between the $n$-times iterated bar construction and the infinitesimal $\cE_n$-Hochschild homology of an augmented $\cE_n$-algebra: $$\ba^{(n)}A \simeq \hh^{\cE_{\!n}}_*(A, k):=k\ot_{U_A}A.$$ 
\end{lemma}

The functor $\ba$ is iterative, by definition: $\ba^{(m+1)}A \simeq k\ot_{\ba^{(m)}A}k$. In order to prove the above assertion we will need a likewise iterative description of $\cE_n$-Hochschild homology, which will rely on a similar equivalence $U^{\cE_{m+1}}_A \simeq A\ot_{U^{\cE_m}_A}A$, where $U^{\cE_i}_A$ is the associative enveloping algebra of $A$ regarded as an $\cE_i$-algebra.

Postponing the proof of Lemma \ref{bar} to our treatment of factorization homology in the following section (where we will use the $\ot$-excision property of factorization homology in the proof), we can now prove Theorem \ref{koszulen} quite succinctly:

\begin{proof}[Proof of Theorem \ref{koszulen}]

By Corollary \ref{pointversion}, there is an cofiber sequence $k\ra \hh^{\cE_{\!n}}_*(A,k) \ra L_{k|A}[n-1]$. Using the equivalences $\hh^{\cE_{\!n}}_*(A,k) \simeq \ba^{(n)}A$ and $L_{k|A} \simeq \epsilon_!L_A[1]$, we obtain a sequence $k\ra  \ba^{(n)}A \ra \epsilon_!L_A[n]$, and thus there is an equivalence $\epsilon_!L_A[n]\simeq \Ker(\ba^{(n)}A \ra k)$ between the suspended cotangent space and the augmentation ideal of the augmented $\cE_n$-coalgebra $ \ba^{(n)}A$. Dualizing, we obtain $TA[-n] \simeq \Ker(\ba^{(n)}A \ra k)^\vee$: I.e., $TA$, after desuspending by $n$, has the structure of a nonunital $\cE_n$-algebra. Thus, we can define the lift of the functor $T:\cE_n\alg^{\rm aug}(\cC)^{\op}\ra \cE_n[-n]\alg^{\rm nu}(\cC)$ by the equivalence $T\simeq\Ker(\ba^{(n)}(-) \ra k)^\vee[n]$.
\end{proof}

The $\oo$-categorical version of Dunn's theorem has an important consequence, which we note here for use later in Section 4. Let $\cC$ be presentable symmetric monoidal $\oo$-category whose monoidal structure distributes over colimits.

\begin{theorem}[\cite{dag6}]\label{monoidal} For $A$ an $\cE_{n+1}$-algebra in $\cC$, the $\oo$-category $\m_A(\cC)$ of left $A$-modules in $\cC$ obtains the structure of an $\cE_n$-monoidal $\oo$-category. That is, $\m_A(\cC)$ is an $\cE_n$-algebra in $\m_\cC({\rm Cat}_\infty^{\rm Pr})$, the $\oo$-category of presentable $\oo$-categories tensored over $\cC$.
\end{theorem}
\begin{proof}[Sketch proof] The functor of left modules $\m: \cE_1\alg(\cC) \ra \m_\cC({\rm Cat}_\infty^{\rm Pr})$ is symmetric monoidal: There is a natural equivalence $\m_{A\ot B}(\cC)\simeq \m_A(\cC) \ot \m_B(\cC)$. Since monoidal functors preserve all algebra structures, therefore $\m$ defines a functor $\cE_1\alg(\cO\alg(\cC)) \ra \cO\alg({\rm Cat}_\infty^{\rm Pr})$ for any topological operad $\cO$. Setting $\cO=\cE_n$ and applying Theorem \ref{dunn} then gives the result.
\end{proof}

This concludes our discussion of the structure on the tangent space of an augmented $\cO$-algebra. A more subtle problem is to describe the exact structure on the absolute operadic cotangent and tangent complexes $L_A$ and $T_A$; these structures are, in some sense, global, rather than local. We now briefly discuss this issue, deferring a more involved discussion to a future work; the following will not be put to use in this work. The following discussion can be summarized as:
\begin{itemize}
\item Local case -- the cotangent space $\epsilon_! L_{R|A}$ of an augmented $\cO$-algebra $\epsilon: R\ra A$ is a $1\circ_\cO1$-comodule in $\m_A^{\cO}$;
\item Global case -- the cotangent complex $f_! L_B$ of an $\cO$-algebra $f: B\ra A$ over $A$ is a $1\circ_\cO1$-comodule in $\m_A^\cO$, but with an additional structure of a coaction of $L_A$, and this additional datum is equivalent to a generalization of a {\it coalgebroid} structure. 
\end{itemize}

Spelling this out, we have the following commutative diagram, obtained by describing the comonads of the adjunctions associated to the stabilization of $\cO$-algebras:
\[\xymatrix{
\cO\alg_{/A}\ar[rr]^{\L_A}\ar[d]_{A\amalg_\cO-}&&\comod_{L_A}(\comod_{1\circ_\cO1}(\m^\cO_A))\ar[d]^{\rm forget}\\
\cO\alg^{\rm aug}_A\ar[dr]_{\L_{A|A}}\ar[rr]^{{ \L}_{A|A}}&& \comod_{1\circ_\cO1}(\m^\cO_A)\ar[dl]^{\rm forget}\\
&\m^{\cO}_A\\}\]
\noindent That is, taking the coproduct with $A\amalg_\cO A$ has the structure of a comonad in $\cO\alg^{\rm aug}(\m_A^{\cO})$, and this gives the functor given by taking the product with absolute cotangent complex $L_A$ the structure of a comonad in $\comod_{1\circ_\cO1}(\m_A^{\cO})$; moreover, every cotangent complex $f_! L_B$ obtains a coaction of $L_A$, for $f:B\ra A$ an $\cO$-algebra map, since there is a coaction of $A\amalg_\cO A$ on $A\amalg_\cO R$. Reiterating, the stabilization of the natural map $B \ra A \amalg_\cO B$ in $\cO\alg_{/A}$, which is the counit of the adjunction between $\cO\alg_{/A}$ and $\cO\alg^{\rm aug}_A$, after stabilizing, gives rise to a map \[f_!L_B\longrightarrow L_A \oplus f_! L_B\] and which is part of a structure of a coaction on $f_! L_B$ of a comonad structure on the functor $L_A \oplus -$. It is tempting to then dualize to obtain an algebraic structure on $T_A$, but the full resulting : For instance, the object $T_A$ no longer has an $\cO$-$A$-module structure, in general, though it should have a Lie algebra structure.

There is one case where this works out quite cleanly, and in which dualizing is unproblematic: where $\cO$ is the $\cE_\infty$ operad. For simplicity, and to make the connection between this story and usual commutative/Lie theory, we shall assume $\cC$ is of characteristic zero. The bar construction $1\circ_{\cE_\infty}1\simeq \Lie[1]^\vee$ produces the shifted Lie cooperad. The situation is summarizes by the following commutative diagram:
\[\xymatrix{
\Calg_{/A}\ar[rr]^{\L_A}\ar[d]_{A\ot-}&&\comod_{L_A}(\Lie[-1]\coalg_A)\ar[d]^{\rm forget}\ar[r]&\m_{T_A}(\Lie[-1]\alg_A)^{\op}\ar[d]^{\rm forget}\\
\Calg^{\rm aug}_A\ar[dr]_{\L_{A|A}}\ar[rr]^{{ \L}_{A|A}}&& \Lie[-1]\coalg_A\ar[dl]^{\rm forget}\ar[r]&\Lie[-1]\alg_A^{\op}\\
&\m_A\\}\]
Thus, $T_A$ obtains a $\Lie[1]$-algebra structure in $\m_A(\cC)$ from this construction; equivalently, $T_A[-1]$ is a Lie algebra in $\m_A(\cC)$, and this Lie algebra structure generalizes that given by the Atiyah class when $A$ is a smooth commutative algebra. Further, the functor given by taking the product with $T_A[-1]$ is a monad in $A$-linear Lie algebras, and for every commutative algebra map $f:R\ra A$, the dual $(f_!L_R)^\vee[-1] \simeq \Hom_R(L_A, R)[-1]$ is a Lie algebra, and has an action of the monad $T_A[-1]$ by $A$-linear Lie algebra maps.

This $\oo$-category $\m_{T_A[-1]}(\Lie\alg_A)$ is actually something very familiar, namely Lie $A$-algebroids, in a slightly altered guise. Let $A$ be a commutative algebra over a field of characteristic zero, and let $\Lie\algd_A$ be the $\oo$-category of Lie $A$-algebroids (obtained from the usual category of differential graded Lie $A$-algebroids by taking the simplicial nerve of the Dwyer-Kan simplicial localization). Then we have the following comparison:

\begin{prop} For $A$ a commutative algebra over a field of characteristic zero, there is an equivalence of $\oo$-categories between Lie $A$-algebroids and $A$-linear Lie algebras with an action of the monad $T_A[-1]$, $$\Lie\algd_A \longrightarrow \m_{T_A[-1]}(\Lie\alg_A),$$ given by taking kernel of the anchor map.
\end{prop}
\begin{proof} We apply the Barr-Beck formalism, \cite{dag2}. For a Lie $A$-algebroid $L$ with anchor map $\rho: L\ra T_A$, the derived kernel of the anchor map $\Ker(\rho)$ naturally has a Lie structure, and the bracket is $A$-linear. Thus, we obtain a functor $\Ker: \Lie\algd_A \ra \Lie\alg_A$ from Lie $A$-algebroids to $A$-linear Lie algebras. This functor preserves limits and has a left adjoint, namely the functor $\Lie\alg_A\ra\Lie\algd_A$ that assigns to an $A$-linear Lie algebra $\frak g$ the Lie $A$-algebroid with zero anchor map, $\frak g \ra 0 \ra T_A$. The composite functor $F$ on $\Lie\alg_A$, given by assigning to $\frak g$ the kernel of the zero anchor map, takes values $F\frak g\simeq T_A[-1]\times \frak g$. $F$ has the structure of a monad on $\Lie\alg_A$, and this monad structure corresponds to that of $T_A[-1]$. We then have a Barr-Beck situation
$$\xymatrix{
\Lie\algd_A\ar[dr]_{\Ker}\ar[rr]&&\m_{T_A[-1]}(\Lie\alg_A)\ar[dl]\\
&\Lie\alg_A&\\}$$ in which the functor $\Ker: \Lie\algd_A \ra \m_{T_A[-1]}(\Lie\alg_A)$ is an equivalence if and only if the functor $\Ker$ is conservative and preserves $\Ker$-split geometric realizations. Firstly, $\Ker$ is clearly conservative, since a map of complexes over $T_A$ is an equivalence if and only if it is an equivalence on the kernel. Secondly, the forgetful functor to $\m_A$ from both $\Lie\algd_A$ and $\Lie\alg_A$ preserves all geometric realizations, and, in particular, $G$-split ones. Thus, the Barr-Beck theorem applies.
\end{proof}

\begin{remark} The previous proposition generalizes to arbitrary $\oo$-categories $\cC$, not of characteristic zero, with the appropriate adjustment in the definition of Lie $A$-algebroids. The proof is identical. We intend to study the homotopy theory of algebroids and the $\cE_n$ analogues in a later work.

\end{remark}

\section{Factorization Homology and $\cE_n$-Hochschild Theories}

\subsection{$\cE_n$-Hochschild Cohomology}

We now consider the notion of the operadic Hochschild cohomology of $\cE_n$-algebras. The following definitions are sensible for general operads, but in this work we will only be concerned with the $\cE_n$ operads.

\begin{definition} Let $A$ be an $\cO$-algebra in $\cC$, and let $M$ be an $\cO$-$A$-module. Then the $\cO$-Hochschild cohomology of $A$ with coefficients in $M$ is $$\hh^*_\cO(A, M) = \Hom_{\m_A^\cO}(A, M).$$
\end{definition}
When the coefficient module $M$ is the algebra $A$ itself, we will abbreviate $\hh^*_\cO(A, A)$ by $\hh^*_\cO(A)$.

\begin{remark} When $A$ is an associative algebra, the $\oo$-category $\m_A^{\cE_1}$ is equivalent to $A$-bimodules, and thus $\cE_1$-Hochschild cohomology is equivalent to usual Hochschild cohomology. In constrast, when $\cO$ is the $\cE_\infty$ operad, and $A$ is an $\cE_\infty$-algebra, the $\oo$-category $\m_A^{\cE_\infty}$ is equivalent to usual left (or right) $A$-modules. Consequently, the resulting notion of $\cE_\infty$-Hochschild cohomology is fairly uninteresting: $\hh^*_{\cE_\infty}(A, M)= \Hom_{\m_A}(A,M)$ is equivalent to $M$.

\end{remark}

\begin{remark} The preceding definition does not require that $\cC$ is stable. A particular case of interest in when $\cC = \Cat$, the $\oo$-category of $\oo$-categories, in which case this notion of Hochschild cohomology categories offers derived analogues to the classical theory of Drinfeld centers, a topic developed in \cite{qcloops}.
\end{remark}

In the case that $\cC$ is stable, the $\cE_n$-Hochschild cohomology is closely related to our previously defined notion of $\cE_n$-derivations and the cotangent complex. We have the following corollary of Theorem \ref{main1}.

\begin{cor} Let $M$ be an $\cE_n$-$A$-module in $\cC$, with $A$ and $\cC$ as above. There is then a natural fiber sequence in $\cC$ $$\Der(A, M)[-n]\longrightarrow \hh_{\cE_{\!n}}^*(A, M) \longrightarrow M.$$
\end{cor}

\begin{proof} Mapping the cofiber sequence $U_A \ra A\ra L_A[n]$ into $M$, we obtain fiber sequences \[\xymatrix{\Hom_{\m_A^{\cE_{\!n}}}(U_A, M)\ar[d]^\simeq & \Hom_{\m_A^{\cE_{\!n}}}(A, M)\ar[l]\ar[d]^\simeq &\ar[l] \Hom_{\m_A^{\cE_{\!n}}}(L_A, M)[-n]\ar[d]^{\simeq}\\
M& \hh^*_{\cE_{\!n}}(A, M)\ar[l]&\ar[l] \Der(A, M)[-n] \\}\] which obtains the stated result.
\end{proof}

A particular case of the corollary above establishes a conjecture of Kontsevich in \cite{motives}. Kontsevich suggested that for an $\cE_n$-algebra $A$ in chain complexes, there is an equivalence between the quotient of the tangent complex $T_A$ by $A[n-1]$ and an $\cE_n$ version of Hochschild cohomology of $A$ shifted by $n-1$. This follows from the above by setting $M=A$, since the tangent complex of $A$ is equivalent to $\Der(A, A)$, we thus obtain
$$\xymatrix{\hh_{\cE_{\!n}}^*(A) [n-1]\ar[r]& A[n-1] \ar[r]& \Der(A, A)\simeq T_A,}$$ implying the equivalence of complexes $\hh_{\cE_{\!n}}^*(A)[n] \simeq  T_A/(A[n-1])$.\footnote{This is the statement of the second claim of \cite{motives}, where Kontsevich terms $T_A$ the deformation complex, which he denotes ${\rm Def}(A)$. This statement was later called a conjecture by Kontsevich and Soibelman in their book on deformation theory, \cite{KS}.} The infinitesimal analogue of this result was earlier proved by Hu in \cite{hu},\footnote{In the terminology of \cite{hu}, the result says that the based $\cE_n$-Hochschild cohomology is equivalent to a shift of the based Quillen cohomology of augmented $\cE_n$-algebras.} under the assumption that $k$ is a field of characteristic zero, which we now generalize:

\begin{cor} Let $A$ be an augmented $\cE_n$-algebra in $\cC$. Then there is a fiber sequence in $\cC$ given by $T_{k|A}[1-n] \ra \hh^*_{\cE_{\!n}}(A, k) \ra k$, where $T_{k|A}$ denotes the relative tangent complex at the augmentation $f: A \ra k$.
\end{cor}
\begin{proof} As in the previous corollary, we obtain this result by dualizing a corresponding result for the cotangent space . From a previous proposition, we have a cofiber sequence $k \ra f_! A \ra L_{k|A}[n-1]$. We now dualize, which produces a fiber sequence $\Hom_\cC (L_{k|A}[n-1], k) \ra \Hom_\cC(f_! A, k) \ra \Hom_\cC(k, k)$. Since $\cC$ is presentable and the monoidal structure distributes over colimits, $\cC$ is closed, implying the equivalence $k \simeq \Hom_\cC(k, k)$. Also, since $f_!$ is the left adjoint to the functor $\cC \ra \m_A^{\cE_{\!n}}(\cC)$ given by restriction along the augmentation $f$, we have an equivalence $\Hom_\cC(f_! A, k) \simeq \Hom_{\m_A^{\cE_{\!n}}}(A, k)$. This is the infinitesimal $\cE_n$-Hochschild cohomology of $A$, $\hh^*_{\cE_{\!n}}(A, k)$, by definition. Thus, we can rewrite our sequence as $\Hom_\cC (L_{k|A}, k)[1-n] \ra \hh^*_{\cE_{\!n}}(A, k) \ra k$, which proves the result.
\end{proof}

\subsection{Factorization Homology and $\cE_n$-Hochschild Homology}

This preceding notion of the operadic Hochschild cohomology of unital $\cO$-algebras is readily available for any operad $\cO$, and this suggests one should look for a companion notion of Hochschild homology. Such a notion appears unknown for a completely general $\cO$. However, there is a notion of Hochschild homology in the case of the $\cE_n$ operad, given by {\it factorization homology}, a topological analogue of Beilinson-Drinfeld's homology of factorization coalgebras, see \cite{bd} and \cite{fg}. This topic has also been developed in depth by Lurie in \cite{dag6}, where he calls it {\it topological chiral homology},\footnote{We offer conflicting terminology with reluctance. The term ``chiral," however, is potentially misleading, since the relation to the chiral sector of a conformal field theory, or other uses of the term, is quite tentative.} and a closely related construction was given by Salvatore \cite{salvatore} in the example in which the target category $\cC$ is topological spaces. We include the present treatment because a shorter discussion of the topic from a slightly simpler perspective might also be of use, and because we require specific results, such as Proposition \ref{excision} and Proposition \ref{commutes}, for our proof of Theorem \ref{big}. A more involved treatment of this and related issues will be forthcoming in \cite{facthomology} and \cite{aft}. More recent work on this subject includes \cite{ricardo}, \cite{kevinowen}, and \cite{gtz}. The notion of factorization homology appears very close to Morrison-Walker's blob complex \cite{blob} (at least for closed $n$-manifolds).

Let $\BTop(n)$ be the classifying space for the group of homeomorphisms of $\RR^n$, and let $B$ be a space with a map $B \ra \BTop(n)$. For $M$ a topological manifold of dimension $n$, $M$ has a topological tangent bundle classified by a map $\tau_M: M \ra \BTop(n)$. A topological manifold $M$ has a $B$-framing given a classifying map $M \ra B$ lifting $\tau_M$.

\begin{definition} Given a map $B\ra \BTop(n)$, the $B$-framed (colored) operad, $\cE_B$, is the symmetric monoidal $\oo$-category whose objects are finite disjoint unions of  $B$-framed $n$-disks and whose morphisms are $B$-framed embeddings.
\end{definition}

If $B$ is a connected space, then $\cE_B$ is equivalent to the PROP associated to an operad, defined as follows: First, choose a $B$-framing of the standard $n$-disk. Now, define the space $\cE_B(I) = \Emb^B(\coprod_I D^n ,D^n)$ of $B$-framed embeddings of $n$-disks: A point in this space consists of an embedding $f: \coprod_{I}D^n_i\ra D^n$, with a homotopy between the $B$-framing of $T_{D^n_i}$ and the $B$-framing on the isomorphic bundle $f_i^{-1}T_{D^n}$, for each $i$. Given a surjection of finite sets $J\ra I$, the usual insertion maps $\cE_B(I) \times \prod_I \cE_B(J_i) \ra \cE_B(J)$ give the collection of spaces $\{\cE_B(I)\}$ an operad structure. The $\oo$-category $\cE_B$ is equivalent to the PROP associated to this operad. 

\begin{lemma} There is a homotopy equivalence $\Emb^B(D^n, D^n) \simeq \Omega B$ of topological monoids, where $\Omega B$ is the based loop space of $B$.
\end{lemma}
\begin{proof} By definition, the space $\Emb^B(D^n, D^n)$ sits in a homotopy pullback square:
$$\xymatrix{
\Emb^B(D^n, D^n) \ar[r]\ar[d]& \Map_{/\negthinspace B}(D^n, D^n)\ar[d]\\
\Emb^{\Top}(D^n, D^n) \ar[r]&\Map_{/\negthinspace\BTop(n)}(D^n, D^n)\\}$$\noindent
There are evident homotopy equivalences $\Map_{/\negthinspace\BTop(n)}(D^n, D^n)\simeq \Omega \BTop(n)\simeq \Top(n)$ and likewise $\Map_{/\negthinspace B}(D^n, D^n)\simeq \Omega B$. By the Kister-Mazur theorem, \cite{kister}, the inclusion of $\Top(n)$ into $\Emb^{\Top}(D^n, D^n)$ is a homotopy equivalence. The map defined by the bottom row is a homotopy inverse to this map, therefore it is a homotopy equivalence. The top map in the diagram is therefore the pullback of a homotopy equivalence, and thus it is also a homotopy equivalence.
\end{proof}

Thus, a choice of basepoint in $B$ similarly defines a map  $ \cE_n(I) \times \Omega B^I \ra \cE_B(I)$, which is a homotopy equivalence for $B$ connected.



\begin{remark}The connectedness assumption on $B$ is not essential in what follows, but we will include it for simplicity and because it holds in virtually all cases of interest, e.g., when $B$ is one of $\ast$, $B{\rm O}(n)$, $B{\rm Spin}(n)$, $B{\rm PL}(n)$ or $B=M$, a connected topological $n$-manifold $M$.
\end{remark}

\begin{example} Consider $B=E\negthinspace\Top(n)\ra \BTop(n)$, a homotopy point of $\BTop(n)$. An $E\!\Top(n)$ structure on an $n$-manifold $M$ is then equivalent to a topological framing of $\tau_M$. The operad $\cE_{E\!\Top(n)}$ is homotopy equivalent to the usual $\cE_n$ operad, because there is a natural homotopy equivalence $\cE_n(I) \xra\sim \cE_{E\!\Top(n)}(I)$, sending a rectilinear embeddings to a framed embedding. By smoothing theory, framed topological manifolds are essentially equivalent to framed smooth manifolds (except possibly in dimension 4).
\end{example}

\begin{example}
For $B=B{\rm O}(n)$, with the usual map $B{\rm O}(n)\ra \BTop(n)$, then $\cE_{B{\rm O}(n)}$ is equivalent to the ribbon $\cE_n$ operad. See \cite{sw} for a treatment of this operad.
\end{example}

\begin{definition} Let $B\ra \BTop(n)$ be as above, and let $\cC$ be a symmetric monoidal $\oo$-category. Then $\cE_B\alg(\cC)$ is $\oo$-category $\Fun^\ot(\cE_B, \cC)$ of symmetric monoidal functors from $\cE_B$ to $\cC$.
\end{definition}

With this setting, we may give a construction of a topological version of factorization homology. Recall that an $\cE_B$-algebra $A$ in $\cC$ is a symmetric monoidal functor $A:\cE_B\ra \cC$, so that there is an equivalence $A(I) \simeq A^{\ot I}$. Let $M$ be a $B$-framed topological $n$-manifold. $M$ defines a contravariant functor $\EE_M: \cE_B^{\op} \ra \Space$, given by $\EE_M(I) = \Emb^B(\coprod_I D^n, M)$. That is, $\EE_M$ is the restriction of the Yoneda embedding of $M$ to the $\oo$-subcategory of $n$-disks.

\begin{definition} As above, given $B\ra \BTop(n)$, let $A$ be an $\cE_B$-algebra in $\cC$, and let $M$ be a topological $n$-manifold with structure $B$. The factorization homology of $M$ with coefficients in $A$ is the homotopy coend of the functor $\EE_M\ot A: \cE_B^{\op}\times \cE_B \ra \Space \times \cC \xra \ot \cC$: \[\int_M A := \EE_M\ot_{\cE_B} A.\]

\end{definition}

\begin{remark} Although formulated slightly differently, this construction is equivalent to the construction of topological chiral homology by Lurie in \cite{dag6}, which we will explain in \cite{facthomology}.
\end{remark}

The following example demonstrates how factorization homology specializes to the case of usual homology:

\begin{example} Let $\cC^\oplus$ be the $\oo$-category of chain complexes equipped with the direct sum monoidal structure. Since every complex $V$ has a canonical and essentially unique map $V\oplus V\ra V$, there is an equivalence $\cE_n\alg(\cC^\oplus)\simeq \cC$. The factorization homology of a framed $n$-manifold $M$ with coefficients in a complex $V$ is then equivalent to $\int_MV \simeq C_\ast(M,V)$, the complex of singular chains on $M$ tensored with $V$.

\end{example}

The $B$-framed manifold $M$, and hence the functor it defines, has an action of the group $\Top^B(M)$, the group of $B$-framed homeomorphisms of $M$. Consequently, the factorization homology $\int_M A$ inherits an action of $\Top^B(M)$. More generally, factorization homology defines a functor \[\xymatrix{{\rm Mflds}_n^B \times \cE_B\alg(\cC) \ar[r]^{ \ \ \ \ \ \ \  \ \ \ \ \int}& \cC\\}\] \noindent
where ${\rm Mflds}_n^B$ is the $\oo$-category of $B$-framed topological $n$-manifolds with morphisms given by embeddings, $\Hom(M, N) := \Emb^B(M, N)$. If $M$ is a topological $k$-manifold with a $B$-framing structure on $M\times \RR^{n-k}$, then we will write $\int_M(-)$ for $\int_{M\times\RR^{n-k}}(-)$.

\begin{remark} There is alternative construction, which we briefly sketch. Let $M$, $\cE_B$, and $A$ be as above. The functor $\EE_M: \cE_n^{\op} \ra \spaces$ defines a symmetric sequence with terms $\EE_M(I):= \Emb^B(\coprod_I D^n, M)$, the space of $B$-framed topological embeddings of the disjoint union of disks $\coprod_I D^n$ into $M$. There are natural maps $\EE_M(I) \times \prod_I \cE_B(J_i) \ra \EE_M(J)$ for every surjection of finite sets $J\ra I$. These maps give $\EE_M$ the structure of a right $\cE_B$-module in symmetric sequences. Since $A$ is an $\cE_B$-algebra, it can be consider as a left $\cE_B$-module in symmetric sequences (concentrated in sequence degree zero). Then one can define $\int_M A = \EE_M\circ_{\cE_B}A$, the geometric realization of the two-sided bar construction of $\cE_B$ with coefficients in the left module $A$ and the right module $\EE_M$.

\end{remark}

We now present several illustrative computations.

\begin{prop}\label{chiralenveloping} Let $A$ be an $\cE_n$-algebra in $\cC$. Then there is a natural equivalence $U_A \simeq \int_{S^{n-1}\times\RR}A$ between the enveloping algebra $U_A$ and the factorization homology of the $(n-1)$-sphere with coefficients in $A$.
\end{prop}
\begin{proof}
By Lemma \ref{kanext}, the enveloping algebra $U_A$ is computed by the left Kan extension of the $A: \cE_n \ra \cC$ along the functor $ \psi:\cE_n \ra {\cE_n}_\ast$, defined by adding a distinguished element $\ast$ to a set $I\in \cE_n$. This Kan extension, $i_! A(\ast)$, is equivalent to the colimit of $A$ over the overcategory $\colim_{{\cE_n}_{/\ast}} A$. The colimit of a diagram $A: \cX \ra \cC$ can be computed as the geometric realization of the simplicial diagram
\[\xymatrix{\coprod_{\{K\}\in \cX_0} A(K) & \ar@<-.5ex>[l]\ar@<.5ex>[l] \coprod_{\{J\ra I\} \in\cX_1}A(J) &\ar@<-1ex>[l]\ar@<1ex>[l]\ar[l] \coprod_{\{F\ra E \ra D\}\in \cX_1\times_{\cX_0}\cX_1} A(F)\ldots \\}\] where $\cX_0$ is the space of objects of $\cX$, $\cX_1$ is the space of morphisms, and $\cX_1\times_{\cX_0} \cX_1$ is the space of composable morphisms. If $\cC$ is tensored over the $\oo$-category of spaces, then, e.g., the first object in this simplicial diagram can be written as $\coprod_{\cX_0} A \simeq \coprod_{[K]\in\pi_0 \cX_0} \Aut(K) \ot A(K)$, where $\Aut(K)$ is the space of automorphisms of the object $K$ in $\cX$, and $\pi_0 \cX_0$ is the collection of equivalence classes of objects of $\cX$. There is a similar description of the higher terms in the above simplicial object.

Applying this to the case of the composite functor $A: {\cE_n}_{/\ast}\ra \cE_n\ra \cC$, we can compute $\colim_{{\cE_n}_{/\ast}}A$ as the geometric realization of a simplicial object
\[\xymatrix{\coprod_K \cE_n(K_\ast) \ot A^{\ot K} &\ar@<-.5ex>[l]\ar@<.5ex>[l] \coprod_{\Hom_{\cE_{\!n}}(J, I)} \cE_n(I_\ast)\ot A^{\ot J}&\ar@<-1ex>[l]\ar@<1ex>[l]\ar[l] \ldots \\}\]
\noindent
The factorization homology $\int_{S^{n-1}\times\RR}A$ is defined as a coend, computed as the colimit of a simplicial object
\[\xymatrix{ \coprod _{K}\Emb^{\rm fr}(\coprod_K D^n, S^{n-1}\times\RR)\ot A^{\ot K} &\ar@<-.5ex>[l]\ar@<.5ex>[l] \coprod_{\Hom_{\cE_{\!n}}(J, I)}\Emb^{\rm fr}(\coprod_I D^n, S^{n-1}\times\RR)\ot  A^{\ot J}&\ar@<-1ex>[l]\ar@<1ex>[l]\ar[l] \ldots\\}\]
We have a map $\cE_n(K_\ast)\ra \Emb^{\rm fr}(\coprod_I D^n, S^{n-1}\times\RR)$, given by translating the disks so that the disk labeled by $\ast$ moves to the origin, and this map is a homotopy equivalence. Thus, the terms in the two simplicial objects above are equivalent; it can be easily seen in addition that the maps are same. We therefore obtain the equivalence $U_A \simeq \int_{S^{n-1}} A$, since they are computed as the geometric realizations of equivalent simplicial objects.
\end{proof}

\begin{remark} This result confirms the intuition that since an $\cE_n$-$A$-module structure and a left action of $\int_{S^{n-1}}A$ both consist of an $S^{n-1}$ family of left $A$-module structures, the two should be equivalent.
\end{remark}


We briefly note a corollary of Proposition \ref{chiralenveloping}:

\begin{cor} Let $A = \free_{\cE_{\!n}}(V)$ be the $\cE_n$-algebra freely generated by $V$ in $\cC$. Then there is an equivalence \[\int_{S^{n-1}} A \simeq A \ot \free_{\cE_{1}}(V[n-1]).\]

\end{cor}

\begin{proof} Combining Proposition \ref{chiralenveloping} and Proposition \ref{envfree}, we compose the equivalences $\int_{S^{n-1}}A \simeq U_A \simeq A\ot \free_{\cE_1}(V[n-1])$.\end{proof}

\begin{remark} There is a generalization of the formula in the preceding corollary where one replaces each occurrence of `1' with `$i$', to compute $\int_{S^{n-i}}A$. This will be considered in \cite{facthomology}.
\end{remark}

Returning to our equivalence $U_A\simeq\int_{S^{n-1}}A$, we see that there is, consequently, an $\cE_1$-algebra structure on $\int_{S^{n-1}\times\RR}A$ corresponding to the canonical algebra structure of $U_A$. This has a very simple geometric construction, and generalization, which we now describe. Let $M^k$ be a $k$-manifold with a $B$-framing of $M\times \RR^{n-k}$. There is a space of embeddings $\coprod_I M\times\RR^{n-k}_i \ra M\times\RR^{n-k}$ parametrized by $\cE_{n-k}(I)$: More precisely, there is a map of operads $\cE_{n-k} \ra \cE{\rm nd}_{M\times \RR^{n-k}}$ from the $\cE_{n-k}$ operad to the endomorphism operad of the object $M\times \RR^{n-k}$ in $\mfld_n^B$. As a consequence, $M\times\RR^{n-k}$ attains the structure of an $\cE_{n-k}$-algebra in the $\oo$-category $\mfld_n^{B}$. Passing to factorization homology, this induces a multiplication map $$m: \int_{\coprod_I M\times\RR^{n-k}}A \ \simeq \ \Bigl(\int_{M\times\RR^{n-k}}A\Bigr)^{\ot I} \longrightarrow \int_{M\times\RR^{n-k}} A$$ for each point of $\cE_{n-k}(I)$. This gives $\int_{M\times\RR^{n-k}} A$ an $\cE_{n-k}$-algebra structure. In other words, the factorization homology functor $\int A: \mfld^{B}_n \ra \cC$ is symmetric monoidal, so it defines a functor from $\cE_{n-k}$-algebras in $\mfld_n^{B}$ to $\cE_{n-k}$-algebras in $\cC$; $M\times\RR^{n-k}$ is an $\cE_{n-k}$-algebra in $\mfld_n^B$, therefore $\int_{M\times \RR^{n-k}}A$ is an $\cE_{n-k}$-algebra in $\cC$. In the particular case of $M=S^{n-1}$, this can be seen to be equivalent to the usual $\cE_1$-algebra structure of $U_A$.

\medskip

We now turn to the problem of defining an analogue of Hochschild homology for $\cE_n$-algebras. As we shall see shortly, in the case $n=1$ there is an equivalence $\int_{S^1}A\simeq \hh_*(A)$, between the factorization homology of the circle and Hochschild homology. It might be tempting to attempt to define the $\cE_n$-Hochschild homology of $\cE_n$-algebra $A$ as the factorization homology of the $n$-sphere $S^n$ with coefficients in $A$. However, unless $n$ is $1, 3,$ or $7$, the $n$-sphere is not a parallelizable manifold, and so the construction requires some modification in order to be well-defined. Our modification will make use of the following basic observation.

\begin{lemma} Let $A$ be an $\cE_n$-algebra in $\cC$. There is an equivalence of $\cE_1$-algebras $\int_{S^{n-1}}A \simeq (\int_{S^{n-1}}A)^{\op}$, between the factorization homology of $S^{n-1}$ with coefficients in $A$ and its opposite algebra, induced by the product map $\tau: S^{n-1}\ra S^{n-1}\times\RR$ of reflection in a hyperplane with reflection about the origin on $\RR$.
\end{lemma}

\begin{proof} Let $\cE_1(I) \times \coprod_I S^{n-1}\times\RR \xra m S^{n-1}\times\RR$ be our $\cE_1(I)$ family of embeddings. The opposite $\cE_1$-algebra of $(\int_{S^{n-1}\times\RR}A)^{\op}$ is defined by the action of $\Sigma_2$ on the operad $\cE_1$, which we now define. A configuration $f:\coprod_I D^1 \hookrightarrow D^1$ defines an ordering of the set $I$, e.g., the left-to-right order of the labeled disks in $D^1$. For $\sigma\in\Sigma_2$, the nontrivial element, the map $\sigma(f)$ has the same image, but reverses the ordering of the disks. The action of the element $\sigma$ intertwines with reflection about a hyperplane $\tau: S^{n-1}\times\RR\ra S^{n-1}\times\RR$ as follows
\[\xymatrix{
\cE_1(I) \times \coprod_I S^{n-1}\times\RR\ar[d]_{\sigma\times {\rm id}}\ar[r]^{{\rm id} \times \tau}& \cE_1(I) \times \coprod_I S^{n-1}\times\RR\ar[dd]^m\\
\cE_1(I) \times\coprod_I S^{n-1}\times\RR\ar[d]_{m}\ar[d]&\\
S^{n-1}\times\RR\ar[r]^\tau & S^{n-1}\times\RR\\}\]\noindent
where the diagram commutes up to a canonical homotopy (which translates and scales the concentric punctured $n$-disks). Passage to factorization homology thereby gives a subsequent commutative diagram
\[\xymatrix{
\cE_1(I) \ot \int_{\coprod_I S^{n-1}\times\RR}A\ar[d]_{\sigma\ot {\rm id}}\ar@/_4.5pc/[dd]_{m^{\op}}\ar[r]^{{\rm id} \ot \tau}& \cE_1(I) \ot \int_{\coprod_I S^{n-1}\times\RR}A\ar[dd]^m\\
\cE_1(I) \ot \int_{\coprod_I S^{n-1}\times\RR}A\ar[d]_{m}\ar[d]&\\
\int_{S^{n-1}\times\RR}A\ar[r]^\tau &\int_{S^{n-1}\times\RR}A\\}\]\noindent
in which the left hand vertical map $m^{\op}$ defines the opposite $\cE_1$-algebra structure on $\int_{S^{n-1}\times\RR}A$, and the right hand vertical arrow $m$ denotes its usual $\cE_1$-algebra structure. This is exactly the condition that the factorization homology map $\tau: (\int_{S^{n-1}\times\RR}A)^{\op}\ra \int_{S^{n-1}\times\RR}A$ is a map of $\cE_1$-algebras, and the map is clearly an equivalence since the reflection map is a homeomorphism. \end{proof}

\begin{example} When $A$ is an $\cE_1$-algebra, this result is quite familiar: $\int_{S^0} A$ is equivalent as an $\cE_1$-algebra to $A\ot A^{\op}$, and its opposite  is $(A\ot A^{\op})^{\op}\simeq A^{\op}\ot A$. There is an obvious equivalence of $\cE_1$-algebras $\tau: A^{\op}\ot A \ra A\ot A^{\op}$ switching the two factors, which is precisely the self map of the $0$-sphere $S^0$ given by reflection about the origin.
\end{example}

As a consequence of this lemma, any left $\int_{S^{n-1}}A$-module $M$ attains a canonical right $\int_{S^{n-1}}A$-module structure $M^\tau$ obtained by the restriction, $M^\tau:={\rm Res}_\tau M$.

\begin{definition} Let $A$ be an $\cE_n$-algebra in $\cC$. Considering $A$ as an $\cE_n$-$A$-module in the usual fashion, and using Proposition \ref{chiralenveloping}, $A$ has the structure of a left $\int_{S^{n-1}}A$-module with $A^\tau$ the corresponding right $\int_{S^{n-1}}A$-module. The $\cE_n$-Hochschild homology of $A$ is the tensor product
\[\hh^{\cE_{\!n}}_*(A) := A^{\tau}\ot_{\int_{S^{n-1}}A}A\]\noindent
and, for $M$ an $\cE_n$-$A$-module, the $\cE_n$-Hochschild homology of $A$ with coefficients in $M$ is the tensor product $\hh^{\cE_{\!n}}_*(A,M) := M^{\tau}\ot_{\int_{S^{n-1}}A}A.$
\end{definition}

If $A$ has an $\cE_{n+1}$-algebra refinement, then there is an equivalence $\hh^{\cE_{\!n}}_*(A)\simeq \int_{S^n}A$. This is because, for a $A$ an $\cE_{n+1}$-algebra, there is an equivalence of $\cE_n$-$A$-modules $A\simeq A^{\tau}$, from which we obtain $\hh^{\cE_{\!n}}_*(A)\simeq A\ot_{\int_{S^{n-1}}A}A\simeq \int_{S^n}A$.

\begin{remark} The preceding construction of $\hh^{\cE_{\!n}}_*(A)$ actually applies somewhat more generally. If $M$ is a topological $n$-manifold with a framing of the tangent microbundle $\tau_M$ after restricting to the complement of a point $x\in M$ and for which there is a framed homeomorphism $U\smallsetminus x \cong S^{n-1}\times\RR$ for a small neighborhood $U$ of $x$ and the standard framing on $S^{n-1}\times\RR$, then one can construct an object $\int_M A$ for an $\cE_n$-algebra $A$. This $\int_M A$ can be defined as $$\int_M A :=A^\tau\ot_{\int_{S^{n-1}}A}\int_{M\smallsetminus\{x\}}A.$$ The group $\Top^{\rm fr}_*(M)$ of based homeomorphisms that preserve the framing away from $M\smallsetminus\{x\}$ (equivalently, the pullback $\Top_*(M) \times_{\Top(M\smallsetminus\{x\})} \Top^{\rm fr}(M\smallsetminus\{x\})$) acts naturally on $\int_MA$.

\end{remark}

Let $N$ be a framed manifold with boundary $\partial N$. Then $\int_N A$ has the structure of both a left and right $\int_{\partial N}A$-module, by the preceding constructions. We can now formulate the following gluing, or $\ot$-excision, property of factorization homology.

\begin{prop}\label{excision} Let $M$ be a $B$-framed $n$-manifold expressed as a union $M\cong U\cup_V U'$ by $B$-framed embeddings of $B$-framed submanifolds, and in which $V\cong V_0\times\RR$ is identified as the product of an $(n-1)$-manifold with $\RR$. Then, for any $\cE_B$-algebra $A$ in $\cC$, there is a natural equivalence \[\int_MA\simeq \int_U A\ot_{\int_VA}\int_{U'}A.\]

\end{prop}

\begin{proof}

Given a bisimplicial object $X_{\bullet\ast}: \Delta^{\op}\times \Delta^{\op} \ra \cC$, one can compute the colimit of $X_{\bullet\ast}$ in several steps. In one way, one can can first take the geometric realization in one of the horizontal direction, which gives a simplicial object $|X_\bullet|_\ast$ with $n$-simplices given by $|X_\bullet|_n = |X_{\bullet, n}|$, and then take the geometric realization of the resulting simplicial object. In the other way, one can take the geometric realization in the vertical simplicial direction to obtain a different simplicial object $|X_\ast|_\bullet$ with $n$-simplices given by $|X_\ast|_n = |X_{n,\ast}|$, and then take the geometric realization of $|X_\ast|_\bullet$. These both compute $\colim_{\Delta^{\op}\times\Delta^{\op}}X_{\bullet\ast}$.

Both $\int_U A\ot_{\int_V A}\int_{U'}A$ and $\int_MA$ are defined as geometric realizations of simplicial objects, the two-sided bar construction and the coend, respectively. To show their equivalence, therefore, we will construct a bisimplicial object $X_{\bullet\ast}$ such that the realization in the horizontal direction gives the two-sided bar construction computing the relative tensor product, and the realization in the vertical direction gives the simplicial object computing the coend $\EE_M\ot_{\cE_B} A$. To do so, observe that each of the terms in the two-sided bar construction, $\int_U A \ot (\int_V A)^{\ot i}\ot\int_{U'}A$, is given as the geometric realization of a simplicial object. That is, define $X_{ij}$ to be the $i$th term in the simplicial object computing the tensor products of coends $\int_U A \ot (\int_VA)^{\ot j}\ot \int_{U'}A$. For fixed a vertical degree $j$, $X_{\bullet,j}$ forms a simplicial object whose colimit is  $|X_{\bullet,j}|\simeq\int_U A \ot (\int_VA)^{\ot j}\ot \int_{U'}A$. We will now define the vertical maps. We begin with the case of the $0$th column. Note the equivalence $$X_{0,j} \simeq \coprod_J \EE_U(J) \ot A^{\ot J} \ot \Bigl(\coprod_I \EE_V(I)\ot A^{\ot I}\Bigr)^{\ot j}\ot \coprod_K \EE_{U'}(K)\ot A^{\ot K}.$$ 
We will show that the $X_{0,j}$ form a simplicial object as $j$ varies, and that the realization $|X_{0,\ast}|$ is equivalent to $\coprod_I \EE_M(I) \ot A^{\ot I}$. First, we can write the colimit $M\cong U\cup_V U'$ as a geometric realization of the simplicial object
$$\xymatrix{M & \ar[l] U\sqcup U'& U\sqcup V \sqcup U' \ar@<.5ex>[l]\ar@<-.5ex>[l] &U \sqcup V\sqcup V\sqcup U' \ar@<1ex>[l]\ar@<-1ex>[l]\ar[l]\ldots\\}$$
The degeneracy maps in this simplicial diagram, induced by $U\sqcup V \rightarrow U$ and $V\sqcup U'\ra U'$, are not quite embeddings, and hence do not quite define maps of embedded disks $\EE_U(J) \times \EE_V(I) \dashrightarrow \EE_U(J\coprod I)$, because the disks may intersect. However, this is easily rectified: Choose an embedding $\jmath: U \hookrightarrow U$ that contracts $U$ into the complement of a closed neighborhood of the boundary $\partial U$. Replacing the identity map ${\rm id}_U$ with $\jmath$, the map $\jmath\sqcup f: U\sqcup V \hookrightarrow U$ is now an embedding, hence induces a map $\EE_U(J)\times \EE_V(I) \ra \EE_U(J\sqcup I)$ for all finite sets $J$ and $I$, and likewise for $U'$. Using these maps, we can write $\EE_M(I)$, for each $I$, as a geometric realization of the embedding spaces of the pieces $U$, $V$, and $U'$:
\[\xymatrix{\EE_M(I) & \coprod_{I\ra \{1,2\}} \EE_U(I_1) \times \EE_{U'}(I_2)\ar[l] &  \coprod_{I \ra \{\ast, 1,2\}}\EE_U(I_1)\times \EE_V(I_*) \times \EE_{U'}(I_2)\ar@<.5ex>[l]\ar@<-.5ex>[l]& \ar@<1ex>[l]\ar@<-1ex>[l]\ar[l]\ldots\\}\]

By tensoring with $A^{\ot I}$, which preserves geometric realizations, we obtain a simplicial object computing $\EE_M(I)\ot A^{\ot I}$:
\[\xymatrix{\coprod \EE_U(I_1) \ot A^{\ot I_1}\ot \EE_{U'}(I_2)\ot A^{\ot I_2} &  \coprod \EE_U(I_1)\ot A^{\ot I_1}\ot \EE_V(I_*)\ot A^{\ot I_*} \ot \EE_{U'}(I_2)\ot A^{\ot I_2}\ar@<.5ex>[l]\ar@<-.5ex>[l]& \ar@<1ex>[l]\ar@<-1ex>[l]\ar[l]\ldots\\}\]

Taking the direct sum over all $I$, the $j$th term of the resulting simplicial object is equivalent to $X_{0,j}$. Thus, the $X_{0, \ast}$ has the structure of a simplicial object, and the realization $|X_{0,\ast}|$ is equivalent to $\coprod_I \EE_M(I)\ot A^{\ot I}$, the $0$th term of the simplicial object computing $\int_MA$. An identical argument gives each $X_{i, \ast}$ a simplicial structure whose realization is the $i$th term of the simplicial object computing $\int_M A$.
\end{proof}

The preceding proposition has several easy, but important, consequences:

\begin{cor} For $A$ an $\cE_1$-algebra, there is an equivalence $\int_{S^1}A\simeq \hh_*(A)$ between the factorization homology of the circle with coefficients in $A$ and the Hochschild homology of $A$.

\end{cor}
\begin{proof} We have the equivalences \[\int_{S^1}A\simeq \int_{D^1}A\underset{{\underset{{S^0}}\int A}}\ot\int_{D^1}A\simeq A\ot_{A\ot A^{\op}}A\simeq \hh_*(A).\]\end{proof}
\begin{remark} With a more sophisticated proof, \cite{dag6}, one can see further that the simplicial circle action in the cyclic bar construction computing $\hh_*(A)$ agrees with the topological circle action by rotations on $\int_{S^1}A$.
\end{remark}

By exactly the same method of proof, we can obtain:

\begin{cor}\label{tensor} Let $A$ be an $\cE_\infty$-algebra, which obtains an $\cE_B$-algebra structure by restriction along the map of operads $\cE_B\ra \cE_\infty$, and let $M$ be a $B$-manifold. Then there is a natural equivalence $$\int_M A\simeq M\ot A$$ \noindent between the factorization homology of $M$ with coefficients in $A$ and the tensor of the space $M$ with the $\cE_\infty$-algebra $A$. More generally, the following diagram commutes:
$$\xymatrix{
{\rm Mflds}^B_n \times \cE_\infty\alg(\cC) \ar[r]\ar[d] & \Space \times \cE_\infty\alg(\cC)\ar[r]^\ot &\cE_\infty\alg(\cC)\ar[d]^{\rm forget}\\
{\rm Mflds}^B_n \times\cE_B\alg(\cC)\ar[rr]_\int &&\cC\\}$$
\end{cor}

\begin{proof}

The proof is a standard induction on a handle decomposition of $M$, in the style of proofs of the h-principle. There is a slight complication in that if $M'$ is a nonsmoothable topological 4-manifold, then $M'$ will not admit a handle decomposition: However the 5-manifold $M'\times\RR$ can be decomposed  into handles. Since $A$ is an $\cE_\infty$-algebra, it is an $\cE_5$-algebra, and $\int_{M'} A$ can therefore be calculated as the factorization homology $\int_{M\times\RR}A$ of the 5-manifold $M\times \RR$. Thus, in this case we can instead perform induction on the handle decomposition of $M:=M'\times\RR$.

The base case of the induction, $M\cong \RR^n$, is immediately given by the equivalences $\int_{\RR^n}A\simeq A \simeq \RR^n\ot A$. Since both operations send disjoint unions to tensor products, we also have the equivalence $\int_{\RR^n\sqcup\RR^n}A\simeq A\ot A\simeq (\RR^n\sqcup\RR^n)\ot A$. Applying Proposition \ref{excision} to the presentation of $S^1\times\RR^{n-1}$ as a union of $\RR^n\cup_{S^0 \times\RR^n}\RR^n$, we obtain the equivalence $\int_{S^1\times\RR^{n-1}}A\simeq A\ot_{A\ot A} A\simeq S^1\ot A$. We can continue the induction to show that for all thickened spheres, the factorization homology $\int_{S^k\times\RR^{n-k}}A$ is equivalent to the tensor $S^k \ot A$.

Now we show the inductive step: Let $M$ be obtained from $M_0$ by adding a handle of index $q+1$. Therefore $M$ can be expressed as the union $M\cong M_0  \cup_{S^q\times\RR^{n-q}}\RR^n$, where $\RR^n$ is an open neighborhood of the $(q+1)$-handle in $M$. Again applying Proposition \ref{excision}, we can compute the factorization homology $\int_MA$ as $$\int_MA \simeq \int_{M_0}A\bigotimes_{\underset{S^q\times\RR^{n-q}}\int A} \int_{\RR^n}A \simeq M_0\ot A\bigotimes_{S^q\times\RR^{n-q}\ot A}\RR^n\ot A \simeq M\ot A$$ where the middle equivalence holds by the inductive hypothesis, the base case, and our previous examination of the result in the special case of thickened spheres.
\end{proof}

The preceding equivalences between the enveloping algebra $U_A\simeq S^{n-1}\ot A$ in the case that $A$ is an $\cE_\infty$-algebra, the tensor $S^{n-1}\ot A$, allow for some further interpretation of the sequence $U_A \ra A \ra L_A[n]$ of Theorem \ref{main1}:

\begin{remark} Let $A$ be an $\cE_\infty$-algebra. Recall from our discussion of the first derivative in Goodwillie calculus that the $\cE_\infty$-cotangent complex of $A$ can be calculated as a sequential colimit \[A\oplus L_A \simeq \varinjlim \Omega_A^n \Sigma_A^n (A\otimes A)\simeq \varinjlim \Omega_A^n (S^n\ot A)\] where $A\ot A$ is regarded an an augmented $\cE_\infty$-algebra over $A$, and $\Omega^n_A$ and $\Sigma^n_A$ are the iterated loop an suspension functors in $\cE_\infty\alg^{\rm aug}_A$. Using the equivalence, $U^{\cE_{\!n}}_A\simeq S^{n-1}\ot A$, between the $\cE_n$-enveloping algebra of $A$ and the tensor of $A$ with the $(n-1)$-sphere, this sequence can now be seen to be comparable to the expression of $L_A$ as the sequential colimit $\varinjlim L_A^{\cE_{\!n}}$ of the $\cE_n$-cotangent complexes of $A$, using the description of $L_A^{\cE_{n}}$ as the kernel of $(S^{n-1}\ot A)[1-n] \ra A[1-n]$.\end{remark}

We have already seen that the factorization homology $\int_{M\times\RR^k}A$ has the structure of an  $\cE_{k}$-algebra; the preceding proposition allows us to see that that algebra structure can be made to be $A$-linear. Before stating the following corollary, we recall by Theorem \ref{dunn} that the $\oo$-category $\m_A(\cC)$ is an $\cE_{n-1}$-monoidal $\oo$-category; the definition of an $\cO$-algebra in a symmetric monoidal $\oo$-category can be slightly modified to make sense in an $\cO$-monoidal $\oo$-category, see \cite{dag3}, \cite{thez}. Then:

\begin{cor} For $M^{n-k}$ a framed manifold, $k\geq 1$, and $A$ an $\cE_n$-algebra in $\cC$, then the factorization homology $\int_{M^{n-k}\times\RR^{k}}A$ has the structure of an $\cE_{k}$-algebra in $\m_A(\cC)$.
\end{cor}

\begin{proof} We describe the case $k=1$. A framed embedding $\RR^{n-1}\hookrightarrow M$ gives $\int_{M\times\RR}A$ an $\int_{\RR^{n-1}\times\RR^1}A$-module structure. The tensor product relative $A$ of \[\int_{M\times\RR^1}\!A\underset{A}\ot \int_{M\times\RR^1}\!A\simeq\int_{M\times(-1,0]\cup_{\RR^{n-1}} M\times[0,1)}\!A\] where the union on the right hand side is taken over the image $\RR^{n-1}\hookrightarrow M\times\{0\}$, and is homotopic to the wedge $M\vee M$. There is family of embedding of $M\times(-1,0]\cup_{\RR^{n-1}} M\times[0,1)$ into $M\times \RR$ parametrized by $\cE_1(2)$ (which are homotopic to the coproduct map $M\vee M \ra M$), giving a space of maps from $\int_{M\times\RR^1}A \ot_A \int_{M\times\RR^1} A$ to $\int_{M\times\RR^1}A$. Extending this construction to $I$-fold tensor products can be seen to give $\int_{M\times\RR^1}A$ an $\cE_1$-algebra structure in $A$-modules.
\end{proof}

With some of the technical tools of factorization homology now at hand, we now return to an outstanding problem from the previous section, the calculation relating the $n$-times iterated bar construction and the infinitesimal Hochschild homology of an augmented $\cE_n$-algebra. That is, we now prove the equivalence $\ba^{(n)}A \simeq \hh^{\cE_{\!n}}_*(A,k)$. This will be a basic argument that certain simplicial objects defining tensor products calculate the same objects; it helpful to first state the following trivial lemma.

\begin{lemma}\label{maneuver} For $A$ an augmented $\cE_n$-algebra in $\cC$ and $R$ an $\cE_1$-algebra in $A$-modules,  then there is an equivalence \[k\ot_R M \simeq k\bigotimes_{k\ot_A R}k\ot_A M\] for every $R$-module $M$.
\end{lemma}
\begin{proof} Note that $k\ot_A R$ obtains an $\cE_1$-algebra structure in $\cC$, since the induction functor $k\ot_A-$ is monoidal hence preserves algebra structure. Letting the $R$-module input ``$M$" vary, we obtain two linear functors $\m_R(\cC)\ra \cC$. To prove that two linear functors agree, it suffices to check on a generator for the $\oo$-category, which is $M=R$. In this case, there is an obvious cancellation to both sides, which are equivalent to $k$.\end{proof}

\begin{proof}[Proof of Lemma \ref{bar}]
We prove the result by induction. The case of $n=1$ is already familiar, but we restate to motivate the argument for higher $n$. The $\cE_n$-Hochschild homology is calculated as $k\ot_{\int_{S^0} A}A = k\ot_{A\ot A} A$. We may then apply the reasoning of Lemma \ref{maneuver} to obtain $$k\ot_{A\ot A} A \simeq k\bigotimes_{k\ot_A (A\ot A)}k\ot_A A \simeq k\ot_A k = \ba^{(1)}A.$$ For the inductive step, we now assume the equivalence $\ba^{(i)}A \simeq \hh_*^{\cE_{\!n}}(A,k)$ and show the equivalence for $i+1$. By definition, $\ba^{(i+1)}A$ is equivalent to the tensor product $k\ot_{\ba^{(i)}A} k$, and we now show the same iteration produces the infinitesimal Hochschild homology. The essential input is Proposition \ref{chiralenveloping}, which reduces the problem to factorization homology, and Proposition \ref{excision}, which allows for induction by successively dividing spheres along their equators. This allows the calculation
\[\hh_*^{\cE_{i+1}}(A,k) = k \ot_{\int_{S^{i}}A}A  \simeq k \bigotimes_{k\ot_A \int_{S^{i}}A}k\ot_A A \simeq k\bigotimes_{k\ot_A \int_{S^{i}}A}k\] applying Lemma \ref{maneuver} for the middle equivalence. The algebra in the last term $k\ot_A \int_{S^i}A$ is equivalent to $k\ot_{\int_{S^{i-1}}A}A$, again using Proposition \ref{excision}, which is equivalent to $\ba^{(i)}A$, using the inductive hypthesis. We thus obtain the equivalence of the iterative simplicial objects calculate $\hh^{\cE_{\!n}}_*(A,k)$ and $\ba^{(n)}A$.
\end{proof}

\begin{example} The preceding lemma has a clear interpretation when $A$ is an $\cE_\infty$-algebra: In this case, there is an equivalence $\ba^{(n)}A \simeq \Sigma^n_kA$ between the $n$-fold bar construction and the $n$-fold suspension of $A$ in the $\oo$-category of augmented $\cE_\infty$-algebras; likewise, there is an equivalence between the infinitesimal $\cE_n$-Hochschild homology with $k\ot_{S^{n-1}\ot A}A \simeq k\ot_A A\ot_{S^{n-1}\ot A}A \simeq k\ot_A (S^n \ot A)$. The equivalence $\Sigma_k^n A \simeq k\ot_A (S^{n}\ot A)$ is then implied by a basic observation for pointed topological spaces: The based $n$-fold loops $\Omega^nX$ is equivalent to the fiber $*\times_X X^{S^n}$ of the space of all maps over the base point in $X$.

\end{example}

\begin{remark} The equivalence of Lemma \ref{bar}, $\ba^{(n)}A \simeq \hh^{\cE_{\!n}}_*(A,k)$, may be thought of as instance of the pushforward for factorization homology: $\int_{M\times N} A \simeq \int_M\int_NA$. $ \hh^{\cE_{\!n}}_*(A,k)$ calculates the factorization homology of the $n$-disk $D^n$ with coefficients in the {\it pair} $(A,k)$. Likewise, $\ba^{(1)}A= k\ot_A k$ calculates the factorization homology of the interval $D^1$ with coefficients in the pair $(A,k)$. Thus, $\ba^{(n)}A$ can be seen to equivalent to $\hh^{\cE_{\!n}}_*(A,k)$, by expressing the $n$-disk as a product $D^n\cong (D^1)^n$, and using the pushforward formula $n-1$ times. We will give a fuller discussion in \cite{facthomology}.
\end{remark}

We will make significant use of the following result in Section 4, in studying moduli problems.

\begin{prop}\label{commutes} For a framed $(n-1)$-manifold $M$, and an $\cE_n$-algebra $A$ in $\cC$, there is a natural equivalence $$\m_{\int_{M\times\RR}A}(\cC)\simeq \int_{M}\m_A(\cC)$$ where the framing of $M\times\RR$ is the product of the given framing on $M$ and a framing of $\RR$.

\end{prop}

\begin{proof} We again prove the equivalence by induction on a handle decomposition of $M$. The two sides are equivalent in the case of $M\cong \RR^{n-1}$. By Proposition \ref{excision}, the factorization homology $\int_M\m_A(\cC)$ glues by tensor products, decomposing $M$ into glued together Euclidean spaces. The right hand side does as well, using the result, $\m_{A\ot_BC} \simeq \m_A\ot_{\m_B}\m_C$, a consequence of, e.g., Theorem 4.7 of \cite{qcloops} in the special case of algebras, i.e., affine stacks.
\end{proof}

\begin{remark} The preceding proposition will be important for our purposes in the case $M = S^{n-1}\times\RR$, in Proposition \ref{iterate}.

\end{remark}

We end with the following conceptual characterization of factorization homology. Note that a symmetric monoidal functor $H:\mfld_n^B\ra \cC$ gives $H(N^{n-1}\times\RR)$ the structure of an $\cE_1$-algebra in $\cC$. Let $\cC$ be a presentable symmetric monoidal $\oo$-category whose monoidal structure distributes over geometric realizations and filtered colimits.

\begin{definition} ${\bf H}(\mfld_n^B, \cC)$, the $\oo$-category of homology theories for $B$-framed $n$-manifolds with coefficients in $\cC$, is the full $\oo$-category of symmetric monoidal functors $H: \mfld_n^B\ra \cC$ satisfying $\ot$-excision: $H(U\cup_{V} U') \simeq H(U)\ot_{H(V)}H(U')$ for every codimension-1 gluing, $V\cong N^{n-1}\times\RR$.
\end{definition}

Induction on a handle decomposition (excepting dimension 4) allows the proof of the following result:
\begin{theorem}[\cite{facthomology}] There is an equivalence of ${\bf H}(\mfld_n^B, \cC) \simeq \cE_B\alg(\cC)$ between $\cC$-valued homology theories for $B$-framed $n$-manifolds and $\cE_B$-algebras in $\cC$. The functor ${\bf H}(\mfld_n^B, \cC) \ra \cE_B\alg(\cC)$ is given by evaluation on $\RR^n$, and the adjoint is given by factorization homology.
\end{theorem}

We defer a proof to \cite{facthomology}, which focuses on the application of factorization homology to topology. 

\section{Moduli Problems}

In this final section, we consider some moduli functors and algebraic groups defined by certain symmetries of $\cE_n$-algebras. Previously, we showed that for an $\cE_n$-algebra $A$, there exists a fiber sequence $A[n-1]\ra T_A\ra \hh^*_{\cE_{\!n}}(A)[n]$. However, we did not exhibit any algebraic structure on the sequence, and we gave no rhyme or reason as to why it existed at all. Our goal is to complete the proof of Theorem \ref{big}, by giving a moduli-theoretic interpretation of this sequence. Namely, we will show that this sequence arises as the Lie algebras of a very natural fiber sequence of derived algebraic groups $B^{n-1}A^\times\ra \Aut_A \ra \fB^nA$, and relatedly, a sequence of $\cE_{n+1}$-moduli problems. As a consequence of this interpretation, the sequence $A[n-1]\ra T_A\ra \hh^*_{\cE_{\!n}}(A)[n]$ will obtain a Lie algebra structure and, relatedly, a nonunital $\cE_{n+1}[-n]$-algebra structure.

Our construction of the group $\Aut_{\fB^nA}$ of automorphisms of an enriched $(\oo,n)$-category $\fB^nA$ will rely on a basic result, Corollary \ref{basicenriched2}, for which we rely on a preprint of Gepner, \cite{gepner}. intuitively, $\fB^nA$ has a single object and a single $k$-morphism $\phi^k$ for $1\leq k<n$, and the hom-object $\Mor(\phi^{n-1}, \phi^{n-1}) = A$.

We now begin our treatment. For the remainder of this work, $\cC$ is a stable presentable symmetric monoidal $\oo$-category whose monoidal structure distributes over colimits. We assume further that $\cC\simeq \ind(\cC)$ is generated under filtered colimits by a {\it small} $\oo$-subcategory $\cC_\diamond\subset \cC$ of compact objects (i.e., $\cC$ is compactly generated), and the compact objects coincide with the dualizable objects in $\cC$. We make use of the notions of the cotangent and tangent complexes of a moduli functor, for which we give an abbreviated summary. See \cite{hag1}, \cite{hag2}, \cite{toen}, and \cite{dag5} for a general treatment of derived algebraic geometry and \cite{moduli} and \cite{thez} for derived algebraic geometry for $\cE_n$-rings. In the following, $\cO$ is an operad for which $\cO(1)$ is the unit.

\begin{definition} For $X$ a functor from $\cO\alg(\cC)$ to the $\oo$-category of spaces, the $\oo$-category of $\cO$-quasicoherent sheaves on $X$ is $$\qc^\cO_X = \lim_{R\in(\cO\alg(\cC))^{\op}_{/X}}\m^{\cO}_R(\cC) \simeq \ulhom_{\Fun}(X, \m)$$\noindent
where $X$ is regarded as a functor to $\oo$-categories by composing with the inclusion of spaces into $\oo$-categories, $X:\cO\alg(\cC)\ra\Space\ra\Cat$, and $\m: \cO\alg(\cC)\ra\Cat$ is the covariant functor assigning to $R$ the $\oo$-category $\m^{\cO}_R(\cC)$, and to a morphism $R\ra R'$ the induction functor $U_R'\ot_{U_R}(-)$.
\end{definition}

In other words, an $\cO$-quasicoherent sheaf $M$ on $X$ is an assignment of an $R$-module $\eta^*M$ for every $R$-point $\eta \in X(R)$, compatible with base change.

\begin{example} Denote by $\Spec R$, the functor of points $\Map_{\cO\alg}(R,-)$ associated to $R$. In this case, the above limit is easy to compute: The $\oo$-category $\cO\alg(\cC)^{\op}_{/\Spec R}$ has a final object, namely $\Spec R$ itself, and as a consequence, there is a natural equivalence $\qc^\cO_{\Spec R} \simeq\m^\cO_R(\cC)$.
\end{example}

We can now make the following definition of the relative cotangent complex, which generalizes our previous notion of the cotangent complex of a map of $\cO$-algebras.

\begin{definition} Let $X$ and $Y$ be functors from $\cO\alg(\cC)$ to the $\oo$-category of spaces, and let $f:X\ra Y$ be a map from $X$ to $Y$. The relative cotangent complex $L_{X|Y}$ of $f$, if it exists, is the $\cO$-quasicoherent sheaf on $X$ for which there is natural equivalence $$\Map_R(\eta^*L_{X|Y}, M)\simeq {\rm Fiber}_\eta(X(R\oplus M) \ra X(R)\times_{Y(R)}Y(R\oplus M)).$$
\end{definition}

\begin{remark} This definition is likely difficult to digest on first viewing: Intuitively, the relative cotangent complex is a linear approximation to the difference between $X$ and $Y$, and it provides a linear method of calculating the value of $X$ on a split square-zero extension $X(R\oplus M)$ given knowledge of $Y$ and $X(R)$.

\end{remark}

If $X$ and $Y$ both admit absolute cotangent complexes (i.e., cotangent complexes relative to $\Spec k$), then there is a cofiber sequence $f^*L_Y\ra L_X\ra L_{X|Y}$, known as the transitivity sequence.

Our particular focus will be on the case of the tangent complex associated to a $k$-point of a moduli functor $e:\Spec k \ra X$. It is convenient for our examples {\it not} to define the tangent space at this point in terms of the cotangent complex, because it might be the case that the tangent complex exists while the cotangent complex fails to exist. That is, we can make sense of the notion of the tangent space at the following extra generality.

\begin{definition} The $\oo$-category $\cM_{\cO}(\cC)$ of infinitesimal $\cO$-moduli problems over $\cC$ consists of the full $\oo$-subcategory of all functors $\cF \in \Fun(\cO\alg^{\rm aug}(\cC_\diamond), \spaces)$ for which: \begin{enumerate}
\item $\cF(k)$ is equivalent to a point;
\item The restriction $\cF(k\oplus -): \cC_\diamond \ra \spaces$ to split square-zero extensions preserves finite limits;
\item $\cF$ preserves products, i.e., the map $\cF(R\times_k R') \ra \cF(R)\times \cF(R')$ is an equivalence for all $R$, $R' \in \cO\alg^{\rm aug}(\cC_\diamond)$.
\end{enumerate}
\end{definition}

The first condition allows us to specify a single point $\Spec k \ra \cF$ to study; the second condition, we next show, implies the existence of the tangent space at that point. The third condition will allow the tangent space of the moduli problem to attain algebraic structure, as we shall see in the final section.

\begin{remark} Note that for a map of operads $\cO \ra \cQ$, restriction along the forgetful functor on algebras induces $\cM_\cO(\cC) \ra\cM_\cQ(\cC)$. In particular, restriction along the forgetful functor defined by $\cE_n\ra \cE_{n+k}$ induces functors $\cM_{\cE_{\!n}}(\cC)\ra \cM_{\cE_{n+k}}(\cC)$, for all $k\geq 0$.
\end{remark}

Given an infinitesimal moduli problem $\cF$, we construct a functor $T\cF:\cC_\diamond \ra \spectra$, from the $\oo$-category $\cC_\diamond$ of compact objects of $\cC$ to the $\oo$-category of spectra. The $i$th space of the spectrum $T\cF(M)$ is defined as $$T\cF(M)(i) = \cF(k\oplus M[i]).$$ By condition (2) of $\cF$ being infinitesimal, this sequence of spaces forms an $\Omega$-spectrum, which is a (typically nonconnective) delooping of the infinite loop space $\cF(k\oplus M)$.
 
Intuitively, the functor $T\cF$ assigns to an object $M$ the spectrum of $M$-valued derivations of $k$ on $\cF$. By the second assumption on $\cF$, $T\cF$ can be seen to preserve finite limits, i.e., $T\cF$ is an exact functor in the terminology of \cite{dag1}. We make use of the following result. See, for instance, \cite{qcloops}. 

\begin{prop} The functor $\cC \ra \Fun^{\rm ex}(\cC_\diamond, \spectra)$, defined by sending an object $M$ to the exact functor $\Map(k, M\ot -)$, is an equivalence. \end{prop}

Thus, there exists an object $T$ in $\cC$ associated to the colimit preserving functor $T\cF$, with the property that there is an equivalence of spectra $T\cF(M) \simeq \Map_\cC(k, T\ot M)$. We will abuse notation and refer to this functor and the object by the same symbols:

\begin{definition} The tangent space $T\cF$ of an infinitesimal moduli problem $\cF$  is the object of $\cC$ for which there is a natural equivalence $\Map_\cC(k, T\cF\ot M) \simeq T\cF(M)$ for all $M\in\cC$.
\end{definition}

Given a moduli functor $X:\cO\alg(\cC)\ra\spaces$ with a map $p: \Spec k \ra X$, we can reduce $X$ at the point $p$ to define a functor, $X_p: \cO\alg^{\rm aug}(\cC)\ra\spaces$, as having values $X_p(R) = {\rm Fiber}_p(X(R)\ra X(k))$, the homotopy fiber of the map $X(R)\ra X(k)$ over the point $p\in X(k)$.

\begin{definition} For a moduli functor $X$, a map $p:\Spec k \ra X$ is a formally differentiable point of $X$ if the reduction $X_p$, as defined above, is an infinitesimal moduli problem.
\end{definition}

In other words, a point is formally differentiable if it is possible to define the tangent space at that point.

\begin{remark} If the functor $T\cF$ also preserves infinite products and coproducts when restricted to $\cC$, rather than $\cC_\diamond$, then, by Yoneda representability reasoning, there will additionally exist a cotangent object $L\cF$ that corepresents $T\cF$. If $\cF$ is the reduction of a pointed moduli functor $X_p$ which has a cotangent complex $L_X$ in $\qc^\cO_X$, then there are equivalences $T{X_p}\simeq (p^*L_X)^\vee$ and $L{X_p}\simeq p^*L_X$.
\end{remark}

\begin{lemma}\label{loop} For a moduli functor with a formally differentiable point $p: \Spec k\ra X$, there is a natural equivalence $T\Omega X_p\simeq TX_p[-1]$, where $\Omega X$ is the pointed moduli functor whose $R$-points are given as the based loop space $\Omega_p(\cF(R))$ based at the point $p: \ast\ra \cF(R)$.

\end{lemma}
\begin{proof} To validate this equivalence, it suffices to determine an equivalence of the functors $T\Omega X_p$ and $\Omega T X_p$, which is immediately manifest.
\end{proof}

Now equipped with the requisite notions of a pointed moduli problem and its infinitesimal tangent complex, we turn to the particular moduli problems of interest.

\begin{definition} Let $A$ and $C$ be $\cE_n$-algebras in $\cC$. The algebraic space of morphisms, $\Mor(A,C)$, is a functor from $\cE_\infty$-algebras to spaces defined by $$\Mor(A,C)(R) = \Map_{\cE_n\alg_R}(R\ot A, R\ot C).$$
\end{definition}

This construction can be extended to make $\cE_n\alg(\cC)$ {\it enriched} over $\Fun(\cE_\infty\alg(\cC),\spaces)$, in the sense of \cite{gepner}. Note that the space $\Mor(A,A)(R)$ naturally has a composition structure for each $R$, whereby the functor $\Mor(A,A)$ can be made to take values in $\cE_1$-algebras in spaces (or, equivalently, topological monoids). In the following, we will refer to a moduli functor valued in topological groups, or loop spaces, as an algebraic group.\footnote{For more general purposes, further conditions are required, as a sheaf condition with respect to a Grothendieck topology on $\cE_\infty\alg(\cC)^{\op}$. For our restricted purpose in this work, which involves infinitesimal automorphisms and formal geometry, these extra conditions are not necessary (though they would typically be satisfied for the groups of interest).} We may now define the algebraic group of automorphisms of an $\cE_n$-algebra.

\begin{definition} The algebraic group $\Aut_A$, of automorphisms of an $\cE_n$-algebra $A$, is the functor $\cE_\infty\alg(\cC)$ to loop spaces whose $R$-points consists of all maps $\Aut_A(R)\subset\Map_{\cE_n\alg_R}(R\ot A, R\ot A)$ that are homotopy equivalences.
\end{definition}

That is, $\Aut_A$ is the open subfunctor of $\Mor(A,A)$ consisting of equivalences. The classifying functor $\BAut_A$ is the infinitesimal moduli functor associated to the composite $B\circ\Aut_A: \cE_\infty\alg(\cC)\ra \Omega\textendash\Space\ra \Space$.

\begin{remark} For the functors $\cF=\BAut_A$, the associated functor ${\rm Fiber}(\cF(R)\ra \cF(k))$, for $R$ an $\cE_\infty$-algebra over $k$, has a familiar interpretation: It is infinitesimally equivalent to the functor ${\rm Def}_A$ of deformations of $A$, in that there is a map that induces an equivalence on tangent spaces. \end{remark}

We now turn to second type of algebraic group which will be of great interest for us, the group of units of an $\cE_n$-algebra. This definition first requires the following construction. For $\cC$ any closed presentable symmetric monoidal $\oo$-category, there is functor $(-)\ot 1_\cC: \Space\ra \cC$, given by assigning to a space $X$ the tensor $X\ot 1_\cC$, where $1_\cC$ is the unit of the monoidal structure on $\cC$. This functor is symmetric monoidal, by assumption on $\cC$, and as a consequence its right adjoint $\Map(1_\cC, -): \cC \ra \Space$ is right lax symmetric monoidal. In other words, if $A$ is an $\cO$-algebra in $\cC$, for any topological operad $\cO$, then $\Map(1_\cC,A)$ attains the structure of an $\cO$-algebra in spaces.

\begin{definition} The functor $\gl: \cE_n\alg(\cC) \ra \Omega^n\textendash\Space$ is the composite of $\Map(1_\cC,-)$ with the functor $\cE_n\alg(\Space)\ra \cE_n\alg(\Space)^{\rm gp}\simeq \Omega^n\textendash\Space$ which assigns to an $\cE_n$-monoid its subspace of invertible elements (which is equivalent to an $n$-fold loop space).
\end{definition}

This allows the formulation of the algebraic units of an $\cE_n$-algebra.

\begin{definition} The algebraic group $A^\times$ of units of $A$ is the functor $\cE_\infty\alg(\cC) \ra \Omega^n\textendash\Space$ assigning to $R$ the $n$-fold loop space $A^\times(R) = \gl(R\ot A)$, where $\gl$ is the functor $\cE_n\alg(\m_R)\ra \Omega^n\textendash\Space$ of the previous definition.
\end{definition}

\begin{remark} In the example where $\cC$ is the $\oo$-category of spectra, then this notion of $\gl$ clearly coincides with the standard notion from algebraic topology. In particular, for a nonconnective $\cE_n$-ring $A$ with connective cover $\tau_{\geq 0}A$, the spaces $\gl(A)$ and $\gl(\tau_{\geq 0}A)$ will be equivalent. However, the algebraic groups $A^\times$ and $(\tau_{\geq0}A)^\times$ will differ despite their equivalence on $k$-points, due to the nonequivalence of $\tau_{\geq0}(R\ot A)$ and $\tau_{\geq0}(R\ot\tau_{\geq0}A)$ for general $R$.
\end{remark}

\subsection{Automorphisms of Enriched $\oo$-Categories}

We now consider our final type of algebraic group, $\Aut_{\fB^nA}$: It will take some preliminaries to finally arrive at the definition. First, note that our definition of the algebraic group $\Aut_A$ should apply verbatim to define an algebraic group structure on automorphisms of any object ``$\fB^nA$" as long as it can be suitably base-changed, i.e., so long as ``$R\ot \fB^nA$" can be defined for each $R\in \cE_\infty\alg(\cC)$.

Let $\cX$ be a monoidal $\oo$-category. We will use the notion of $\oo$-categories enriched in $\cX$, which we denote $\Cat(\cX)$, developed by Gepner in \cite{gepner}. We will not give the technical definitions but instead summarize the very rudimentary properties from \cite{gepner} necessary for our purposes: Given a monoidal $\oo$-category $\cX$, one constructs $\cE_1\alg_\bigstar(\cX)$, an $\oo$-category of $\cE_1$-algebras in $\cX$ with many objects; $\Cat(\cX)$, $\oo$-categories enriched in $\cX$, is a localization of $\cE_1\alg_\bigstar(\cX)$, obtained by inverting the enriched functors which are fully faithful and essentially surjective.

There is a functor, which we will denote $\fB$, $$\xymatrix{\cE_1\alg(\cX) \ar[r]^\fB & \Cat(\cX) }$$\noindent
where for any $A\in \cE_1\alg(\cX)$, $\fB A$ is an enriched $\oo$-category with a single distinguished object $* \in \fB A$, and such that there is an equivalence of the hom object $\Mor_{\fB A}(*,*) = A$ as algebras in $\cX$. We denote by $1$ the enriched $\oo$-category $\fB1_\cX$. The functor $\fB$ factors as $$\xymatrix{\cE_1\alg(\cX)\ar[r]& \Cat(\cX)^{1/}\ar[r]&\Cat(\cX)}$$\noindent
through $\Cat(\cX)^{1/}$, enriched $\oo$-categories with a distinguished object $1$. There is likewise a functor $\Cat(\cX)^{1/} \ra \cE_1\alg(\cX)$ sending an enriched $\oo$-category $\cA$ with a distinguished object $1\ra \cA$ to the endomorphism algebra object $\underline{\End}_\cA(1)$. We have the following adjunction: $$\xymatrix{\Cat(\cX)^{1/}\ar[d]^{\underline\End(1)}\\
\cE_1\alg(\cX)\ar@/^1pc/[u]^{\fB}}$$

We summarize these points in the following theorem:

\begin{theorem}[\cite{gepner}]\label{gepner} For a symmetric monoidal $\oo$-category $\cX$, there is an symmetric monoidal $\oo$-category $\Cat(\cX)$ with unit $1$, with a symmetric monoidal functor $\fB:\cE_1\alg(\cX)\ra \Cat(\cX)^{1/}$ which is fully faithful.
\end{theorem}

\begin{example} In the case that $\cX$ is $\spaces$, the $\oo$-category of spaces equipped with the Cartesian monoidal structure, then $\cE_1\alg(\spaces)$ is equivalent to the $\oo$-category of topological monoids and $\Cat(\spaces)$ is equivalent to $\Cat$. The functor $\fB$ is equivalent to the functor that assigns to a topological monoid $G$ to its simplicial nerve $N_\bullet G=\fB G$, thought of as an $\oo$-category with a single object whose endomorphisms equal $G$. (Since the usual classifying space $BG$ is equivalent to the geometric realization of the simplicial nerve $|\fB G|$, this is the motivation for the notation ``$\fB$".)
\end{example}

This theorem has the following corollary.

\begin{cor}\label{basicenriched}
For any $A$ and $C$ in $\cE_1\alg(\cX)$, there is a homotopy pullback diagram of spaces:
$$\xymatrix{
\Map_{\cE_1\alg(\cX)}(A,C)\ar[d] \ar[r]& \Map_{\Cat(\cX)}(\fB A, \fB C)\ar[d]\\
\ast \ar[r] & \Map_{\Cat(\cX)}(1, \fB C) \\}$$

\end{cor}
\begin{proof}

Since the functor $\fB: \cE_1\alg(\cX) \ra \Cat(\cX)^{1/}$ has a right adjoint and the unit of the adjunction is an equivalence, we have that $\fB$ is fully faithful. Therefore, $\Map_{\cE_1\alg}(A,C)$ is homotopy equivalent to $\Map_{\Cat(\cX)^{1/}}(\fB A, \fB C)$. The result now follows from the standard formula for mapping objects in an under category.
\end{proof}

For $\cC$ a symmetric monoidal $\oo$-category, the $\oo$-category $\cE_1\alg(\cC)$ inherits the symmetric monoidal structure of $\cC$. Thus, the construction above can be iterated to obtain a functor $\fB^2: \cE_1\alg(\cE_1\alg(\cC))\ra \Cat(\Cat(\cC))$, where $\oo$-category $\Cat(\Cat(\cC))$ is our definition for ${\rm Cat}_{(\oo,2)}(\cC)$, $(\oo,2)$-categories enriched in $\cC$, \cite{gepner}. Iterating, we obtain a fully faithful functor
$$\cE_1\alg^{(n)}\alg(\cC) \hookrightarrow {\rm Cat}_{(\oo,n)}(\cC)^{1/}$$ where $\cE_1\alg^{(n)}(\cC)$ is the $\oo$-category of $n$-times iterated $\cE_1$-algebras in $\cC$, and $1= \fB^n 1_\cC$ is the unit of ${\rm Cat}_{(\oo,n)}(\cC)$. As a consequence, we have an identical formula for mapping spaces as in Corollary \ref{basicenriched}. That is, we may now use this result to describe mapping spaces of $\cE_n$-algebras, using the theorem of Dunn, and Lurie, that an $\cE_n$-algebra is an $n$-times iterated $\cE_1$-algebra. By the same reasoning as for Corollary \ref{basicenriched}, we have the following corollary of Theorem \ref{gepner} and Theorem \ref{dunn}:

\begin{cor}\label{basicenriched2}
For any $\cE_n$-algebras $A$ and $C$ in $\cX$, there is a homotopy pullback diagram:
$$\xymatrix{
\Map_{\cE_n\alg(\cX)}(A,C)\ar[d] \ar[r]& \Map_{{\rm Cat}_{(\oo,n)}(\cX)}(\fB^n A, \fB^n C)\ar[d]\\
\ast \ar[r] & \Map_{{\rm Cat}_{(\oo,n)}(\cX)}(1, \fB^n C) \\}$$

\end{cor}

Finally, we will need the following comparison between enriched $\oo$-categories and tensored $\oo$-categories. Let $\cC$ be a presentable symmetric monoidal $\oo$-category whose monoidal structure distributes over colimits, and let $\m_A^{(n)}$ abbreviate the $n$-fold application of the $\m$-functor, defined by $$\m_A^{(k+1)}:=\m_{\m_A^{(k)}}({\rm Cat}^{\rm Pr}_{(\oo,k)}(\cC))$$ where ${\rm Cat}^{\rm Pr}_{(\oo,k)}(\cC)$ consists of those $(\oo,k)$-categories enriched in $\cC$ which are presentable. An $A$-module in $\cC$ is equivalent to an enriched functor $\fB A\ra \cC$. Likewise, there is an equivalence $\m_A^{(k+1)} \simeq \Fun(\fB^{k+1}A, {\rm Cat}^{\rm Pr}_{(\oo,k)}(\cC))$.

\begin{prop}[\cite{gepner}]\label{enrichedcomparison} The map $\Map(\fB^nA, \fB^nC)\ra \Map(\m_A^{(n)}, \m_C^{(n)})$ is full on components, and the essential image consists of those functors $F$ for which there exists an equivalence $F(\m_A^{(i)})\simeq \m_C^{(i)}$, for each $i<n$.\end{prop}

We now define the algebraic group $\Aut_{\fB^n A}$:

\begin{definition} For an $\cE_n$-algebra in a symmetric monoidal $\oo$-category $\cC$, the algebraic group of automorphisms of $\fB^nA$ is the functor $\Aut_{\fB^nA}: \cE_\infty\alg(\cC)\ra \Omega\textendash\spaces$ whose $R$-points are given by the subspace $$\Aut_{\fB^n A}(R) \subset \Map_{{\rm Cat}_{(\oo,n)}(\m_R(\cC))}(\fB^n(R\ot A),\fB^n(R\ot A))$$ of those functors which are equivalences.
\end{definition}

In order to apply the theory of infinitesimal moduli problems to our algebraic groups of interest, it is necessary to make the following observation.

\begin{lemma} The moduli functors $\Mor(A,C)$, $\Aut_A$, $A^\times$, and $\Aut_{\fB^nA}$ have formally differentiable points, i.e., the reduction of the moduli functors  at their natural basepoints form infinitesimal moduli problems.
\end{lemma}
\begin{proof} We give the proof for $f\in \Mor(A,C)$, the others cases being similar. Condition (1) is immediate. Conditions (2) and (3) are implied by the fact the tensor products in $\cC$ distribute over finite limits, and therefore the reduction $\Mor(A,C)_f$ preserves finite limits.
\end{proof}

\subsection{Infinitesimal Automorphisms of $\cE_n$-algebras}

The results of the rest of this section involve the interrelation of these algebraic groups and the cohomology theories of $\cE_n$-algebra studied earlier in this paper, and are summarized in the following theorem.

\begin{theorem}\label{main2} There is a fiber sequence of algebraic groups $B^{n-1}A^\times \ra \Aut_A \ra \Aut_{\fB^n A}$. Passing to the associated tangent spaces gives a fiber sequence $A[n-1]\ra T_A\ra \hh^*_{\cE_{\!n}}(A)[n]$. \end{theorem}

\begin{remark} This theorem elaborates on Kontsevich's conjecture from \cite{KS} and \cite{motives}. Kontsevich conjectured an equivalence of Lie algebras $\hh^*_{\cE_{\!n}}(A)[n] \simeq T_A/A[n-1]$, and we go further to interpret this equivalence at the level of moduli problems. For instance, this theorem implies, for $A$ defined over a field of characteristic zero, that the Maurer-Cartan elements of $\hh^*_{\cE_{\!n}}(A)[n]$ classify deformations of the enriched $\oo$-category $\fB^n A$, or, equivalently, certain deformations of $\m^{(n)}_A$. This is a generalization of the familiar result that Maurer-Cartan elements of the Lie bracket of usual Hochschild cohomology $\hh^*(A)[1]$ classify deformations of $\fB A$ or, equivalently, deformations of the category of $A$-modules, for $A$ an associative algebra. This $n=1$ case is very close to Keller's theorem in \cite{keller}.

\end{remark}

\begin{remark} By the Lie algebra of an algebraic group $\GG$ in $\cC$, we mean its tangent space at the identity map $\id: \Spec k\ra\GG$, which can be expressed as any of the equivalent tangent spaces,  $\Lie(\GG):= T\GG \simeq TB\GG[-1]\simeq T_{k|B\GG}$. It remains to show that this tangent space indeed possesses a Lie algebraic structure: We address this point in the final section.
\end{remark}

There is an analogue of the preceding theorem for a map $f:A\ra C$ of $\cE_n$-algebras:

\begin{theorem} There is a fiber sequence of moduli problems $\Mor(A,C)\ra \Mor(\fB^n A, \fB^n C)\ra B^nC^\times$, and given an $\cE_n$-ring map $f:A\ra C$, looping gives a corresponding sequence of algebraic groups $\Omega_f\Mor(A,C)\ra\Omega_f\Mor(\fB^nA,\fB^nC)\ra B^{n-1}C^\times$. Passage to the tangent spaces at the distinguished point gives a fiber sequence of Lie algebras $\Der(A,C)[-1]\ra\hh_{\cE_{\!n}}^*(A,C)[n-1]\ra C[n-1]$.

\end{theorem}

We will prove this piecemeal, beginning with the fiber sequence of moduli functors.

\begin{prop} There is a fiber sequence of algebraic groups $B^{n-1}A^\times\ra \Aut_A\ra \Aut_{\fB^n A}$.
\end{prop}
\begin{proof}
By Corollary \ref{basicenriched2}, there is a fiber sequence of spaces \[\Mor(A,A)(R)\ra \Mor(\fB^n A, \fB^n A)(R)\ra \Mor(\fB^n k, \fB^n A)(R)\] for every $R$. Since limits in functor $\oo$-categories are computed pointwise in the target, this implies that $\Mor(A,A)\ra \Mor(\fB^n A, \fB^n A)\ra \Mor(\fB^n k, \fB^n A)$ is a fiber sequence of moduli functors. Restricting to equivalences in the first and second terms gives rise to an additional fiber sequence $\Aut_A \ra \Aut_{\fB^n A}\ra \Mor(\fB^n k, \fB^n A)$.

Next, we identify the moduli functor $\Mor(\fB^n k, \fB^n A)(R)$ with $B^nA^\times$. It suffices to produce a natural equivalence on their $k$-points, the spaces $\Map_{{\rm Cat}_{(\oo,n)}(\cC)}(\fB^n k, \fB^n A)$ and $B^n\gl(A)$, the argument for general $R$-points being identical. By the adjunction $\Map(k,-):\cE_n\alg(\cC)\leftrightarrows \cE_n\alg(\Space):(-)\ot k$, a map out of $\fB^n k$ in ${\rm Cat}_{(\oo,n)}(\cC)$ is equivalent to a map out of the contractible, trivial category $\ast$ in ${\rm Cat}_{(\oo,n)}$. Thus, we have the equivalence $\Map_{{\rm Cat}_{(\oo,n)}(\cC)}(\fB^n k, \fB^n A) \simeq\Map_{{\rm Cat}_{(\oo,n)}}(*, \fB^n \Map(k,A))$. For the $n=1$ case of $\Cat$, there is an equivalence $\Map_{\Cat}(*, \cC) = \cC^{\sim}$, the subspace of $\cC$ consisting of all invertible morphisms. Iterating this relation to obtain the same result for all $n$, this implies the equivalence \[\Map_{{\rm Cat}_{(\oo,n)}}(*, \fB^n \Map(k,A))\simeq\Map_{{\rm Cat}_{(\oo,n)}}(*, \fB^n \gl(A))\simeq B^n\gl(A),\] which completes our argument that $\Mor(\fB^n k, \fB^n A)$ and $B^nA^\times$ define the same moduli functor. Cumulatively, we may now identify a natural fiber sequence of functors $\Aut_A \ra \Aut_{\fB^n A}\ra B^nA^\times$. The homotopy fiber of the map $\Aut_A \ra \Aut_{\fB^n A}$ can thereby be identified as the looping of the base, $\Omega B^nA^\times$, which is equivalent to $B^{n-1}A^\times$. The map $\Aut_A\ra\Aut_{\fB^nA}$ is a map of algebraic groups, and limits of algebraic groups are calculated in the underlying $\oo$-category of functors, therefore the inclusion of the fiber $B^{n-1}A^\times \ra \Aut_A$ is a map of algebraic groups.
\end{proof}

\begin{lemma}\label{fiber} For a fiber sequence of infinitesimal moduli functors $X\ra Y\ra Z$, passage to the tangent spaces results in a fiber sequence $TX\ra TY\ra TZ$ in $\cC$.
\end{lemma}
\begin{proof} To prove that $TX\ra TY\ra TZ$ is a fiber sequence, it suffices to show that $TX\ra TY \ra TZ$ is a fiber sequence of functors, i.e., that for every $M\in\cC$, that $TX(M)\ra TY(M)\ra TZ(M)$ is a fiber sequence of spectra. This, in turn, follows from the corresponding fact for the space-valued functors: If this sequence forms a fiber sequence of spaces for every $M$, then the previous sequence will be a fiber sequence of spectra, since the functor $\Omega^\infty$ preserves fibrations.\end{proof}

The next step is the identification of the tangent spaces of the individual terms in the sequence $B^{n-1}A^\times \ra \Aut_A\ra \Aut_{\fB^nA}$.

\begin{lemma}\label{vect} There is a natural equivalence $\Lie(\Aut_A) \simeq T_A$ and more generally, an equivalence of functors $T\Aut_A \simeq \Hom_{\m_A^{\cE_{\!n}}}(L_A, -\ot A)$
\end{lemma}
\begin{proof}
First, there is an equivalence of tangent spaces $T{\Aut_A}$ and $T{\Mor(A,A)}$, for the following general reason. Let $X\ra Y$ is a map of moduli functors for which $X(R)\ra Y(R)$ is an inclusion of components for every $R$, which can be thought of as a generalization of the notion of a map being formally Zariski open. In this case, the fibers of the map $X(R\oplus M) \ra X(R)\times_{Y(R)}Y(R\oplus M)$ are trivial, for all $R$ and $M$, and thus the relative cotangent and tangent complexes are trivial. In particular, the relative tangent complex of $\Aut_A \ra \Mor(A,A)$ is trivial, and from the transitivity sequence we obtain the natural equivalence $T{\Aut_A}\simeq T{\Mor(A,A)}$.

Now, let $N$ be an object of $\cC$. By definition, the space $\Map_\cC(L_{k|\Mor(A,A)}, N)$ is the loop space of the fiber of the map $\Mor(A,A)(k\oplus N)\ra \Mor(A,A)(k)$ induced by the projection map $k\oplus N \ra k$. This fiber is the mapping space $\Map_{\cE_n\alg/A}(A, A \oplus A\ot N)$, which is equivalent to $\Map_{\m_A^{\cE_{\!n}}}(L_A, A\ot N)$. Thus, we have an equivalence $\Map_{k}(L_{k|\Mor(A,A)}, N)\simeq \Omega \Map_{\m_A^{\cE_{\!n}}}(L_A, A\ot N)$ for all $N\in \cC$. By setting $N=k$, we obtain the equivalence $T_{k|\Aut_A}\simeq T_{k|\Mor(A,A)} \simeq T_A[-1]$. Applying the previous lemma to the case $\cF=\BAut_A$ completes the proof.
\end{proof}

\begin{remark} This is an derived algebraic analogue of the familiar topological fact that the Lie algebra of the diffeomorphism group of a smooth manifold is equivalent to the Lie algebra of vector fields.\end{remark}

\begin{lemma}\label{unit} There is a natural equivalence $\Lie(B^{n-1}A^\times) \simeq A[n-1]$.
\end{lemma}
\begin{proof}
This is a consequence of the following. First, for a left $A$-module $V$, let $\End_A(V)$ be the functor $\cE_\infty\alg(\cC)\ra \Space$ that assigns to $k'$ the space of maps $\End_A(V)(k')=\Map_{k'\ot A}(k'\ot V, k'\ot V)$. A standard calculation shows the equivalence $T_{k|\End_A(V)}[1] \simeq \Hom_A(V,V)$. Using that $A^\times\ra \End_A(A)$ is formally Zariski open, we obtain that tangent space of $A^\times$ at the identity is $A$, which can delooped $n-1$ times, using Lemma \ref{loop}. 
\end{proof}

Having described $\Lie(\Aut_A)$ and $\Lie(B^{n-1}A^\times)$, we are left to identify the Lie algebra of infinitesimal automorphisms of the enriched $(\oo,n)$-category $\fB^nA$. Let $A$ be an $\cE_n$-algebra in $\cC$, as before. The remainder of this section proves the following, which will complete the proof of the part of Theorem \ref{main2} that our sequence relating the tangent complex and Hochschild cohomology is the infinitesimal version of our sequence of algebraic groups relating automorphisms of $A$ and those of $\fB^nA$:

\begin{theorem}\label{liehh} There is an equivalence $$\Lie(\Aut_{\fB^nA})\simeq \hh^*_{\cE_{\!n}}(A)[n]$$ between the Lie algebra of the algebraic group of automorphisms of the $\cC$-enriched $(\oo,n)$-category $\fB^nA$, and the $n$-fold suspension of the $\cE_n$-Hochschild cohomology of $A$.
\end{theorem}

We prove the theorem by applying the following proposition.

\begin{prop}\label{crutch} There is an equivalence $\Omega^n \Aut_{\fB^nA}(k)\simeq \gl(\hh^*_{\cE_{\!n}}(A))$.

\end{prop}

\begin{proof}[Proof of Theorem \ref{liehh}]

To prove the theorem, it suffices to show an equivalence of spaces between $T\Omega^n\Aut_{\fB^nA}(N)$ and $\Map_{\m_A^{\cE_{\!n}}}(A,A\ot N)$ for every $N$ in $\cC$: By choosing $N=k[j]$ to be shifts of $k$, this would then imply the equivalence of Theorem \ref{liehh}.

The space $\Map_{\m_A^{\cE_{\!n}}}(A,A\ot N)$ is the fiber, over the identity, of the map \[\xymatrix{\Map_{\m_{A\oplus A\ot N}^{\cE_{\!n}}}(A\oplus A\ot N,A\oplus A\ot N)\ar[r]&\Map_{\m_A^{\cE_{\!n}}}(A,A),\\}\] which is equivalent to the fiber of the map \[\xymatrix{\gl(\hh^*_{\cE_{\!n}}(A\oplus A\ot N)) \ar[r]&\gl(\hh^*_{\cE_{\!n}}(A))\\}\] since the bottom row is a subspace, full on connected components, in the top row. We therefore have a map of homotopy fiber sequences:
\[\xymatrix{
T\Omega^n\Aut_{\fB^nA}(N)\ar[d]\ar[r]&\Omega^n \Map(\fB^nA\oplus A\ot N,\fB^nA\oplus A\ot N)\ar[d]^{\sim}\ar[r]&\Omega^n\Map(\fB^nA,\fB^nA)\ar[d]^{\sim}\\
\Map_{\m_A^{\cE_{\!n}}}(A, A\ot N)\ar[r]&\gl(\hh^*_{\cE_{\!n}}(A\oplus A\ot N)) \ar[r]&\gl(\hh^*_{\cE_{\!n}}(A))\\}\] Since the two right hand vertical arrows, on the base and total space, are homotopy equivalences, this implies that the left hand map, on fibers, is a homotopy equivalence: $T\Omega^n\Aut_{\fB^nA}(N)\simeq \Map_{\m_A^{\cE_{\!n}}}(A, A\ot N)$.
\end{proof}

The rest of this section will be devoted to the proof of Proposition \ref{crutch}. See \cite{dag6} for a closely related treatment of this result. The essential fact for the proof is the following:

\begin{prop}\label{iterate} There is an equivalence $\hh^*_{\cE_{n-1}}(\m_A(\cC)) \simeq \m_A^{\cE_{\!n}}(\cC)$.

\end{prop}

\begin{remark} Proposition \ref{iterate} generalizes a result proved in \cite{qcloops}, where this statement was proved in the case where $A$ had an $\cE_\infty$-algebra refinement. The argument below is essentially identical, replacing the tensor $S^k \ot A$, used in \cite{qcloops}, by the factorization homology $\int_{S^k}A$.
\end{remark}

\begin{proof} The proof is a calculation of the $\oo$-category of functors $\Fun_{\cE_{n-1}\textendash{\m_A}}(\m_A,\m_A)$, which are $\cE_{n-1}$-$\m_A$-module functors. First, we have an equivalence between $\cE_{n-1}$-$\m_A$-module $\oo$-categories and $(\int_{S^{n-2}}\m_A)$-module $\oo$-categories, as a consequence of Proposition \ref{chiralenveloping}. Thus, we obtain an equivalence $$\hh_{\cE_{n-1}}^\ast(\m_A)\simeq \Fun_{\int_{S^{n-2}}\m_A}(\m_A,\m_A).$$
\noindent
Secondly, for $M$ a stably parallelizable $k$-manifold of dimension less than $n-1$, then $\int_M \m_A \simeq \m_{\int_M A}$, by Proposition \ref{commutes}. Applying this in the case of $M= S^{n-2}$ produces the further equivalence
$$\hh_{\cE_{n-1}}^\ast(\m_A)\simeq \Fun_{\m_{\int_{S^{n-2}}A}}(\m_A,\m_A).$$
\noindent
Finally, we apply the general equivalence of $\oo$-categories $$\Fun_{\m_R} (\m_A, \m_B)\simeq \m_{A^{\op}\ot_R B}$$\noindent in the case of $R = \int_{S^{n-2}}A$ and $A=B$. Again using the basic features of factorization homology, the equivalences $A^{\op}\ot_{\int_{S^{n-2}}A}A\simeq \int_{S^{n-1}}A$ and $\int_{S^{n-1}}A\simeq U_A$ give the promised conclusion of \[\hh_{\cE_{n-1}}^\ast(\m_A) \simeq \m_{\int_{S^{n-1}}A}\simeq \m_A^{\cE_{\!n}}.\]\end{proof}

This has an immediate corollary, that the $\cE_n$-Hochschild cohomology of $A$ is equivalent to the endomorphisms of the unit of the tensor structure for the $\cE_{n-1}$-Hochschild cohomology of $\m_A$:

\begin{cor} For $A$ as above, there is a natural equivalence $\hh^*_{\cE_{\!n}}(A) \simeq \Hom_{\hh^*_{\cE_{n-1}}(\m_A)}(1,1)$.
\end{cor}

We have the following transparent lemma, which we will shortly apply.

\begin{lemma} For $\cX$ an $\oo$-category with a distinguished object $1$, and $\cX^{\sim}$ the underlying space consisting of objects of $\cX$, then there is an equivalence $\gl(\Map_{\cX}(1,1))\simeq \Omega_1\cX^{\sim}$. \end{lemma}\qed

This gives the following:

\begin{cor} For $A$ an $\cE_n$-algebra, there are equivalences $\gl(\hh^*_{\cE_{\!n}}(A)) \simeq \Omega \gl(\m_A^{\cE_{\!n}}) \simeq \Omega \gl (\hh^*_{\cE_{n-1}}(\m_A))$.

\end{cor}

We now complete the proof of Proposition \ref{crutch} (which, in turn, completes the proof of Theorem \ref{liehh}).

\begin{proof}[Proof of Proposition \ref{crutch}]

By iterating the previous corollary, we obtain \[\gl(\hh^*_{\cE_{\!n}}(A))\simeq \Omega^n\gl\Bigl(\hh^*_{\cE_0}\Bigl(\m^{(n)}_A\Bigr)\Bigr).\] For the case $n=0$, the $\cE_0$-Hochschild cohomology is given simply by endomorphisms, $\hh^*_{\cE_0}(R) = \Hom(R,R)$. Hence we have the equivalence $\gl(\hh^*_{\cE_{\!n}}(A))\simeq \Omega^n _{\rm id}\Map(\m_A^{(n)},\m_A^{(n)})$. Using Proposition \ref{enrichedcomparison}, we finally conclude that $\gl(\hh^*_{\cE_{\!n}}(A))\simeq \Omega^n_{\rm id}\Map(\fB^nA,\fB^nA)$, proving our proposition.
\end{proof}

\begin{remark} The proof of Proposition \ref{crutch} extends to show that $\Lie(\Omega_f\Mor(\fB^nA, \fB^nC))$ is equivalent to $\hh^*_{\cE_{\!n}}(A,C)[n-1]$, for $f:A\ra C$ an $\cE_n$-algebra map.

\end{remark}

It remains to equate the two sequences constructed in the main theorems of this paper are equivalent:

\begin{prop} The two sequences $\hh^*_{\cE_{\!n}}(A) \ra A \ra T_A[1-n]$ constructed in Theorems \ref{main1} and \ref{main2} are equivalent.
\end{prop}

\begin{proof}
It suffices to show that the two maps $\hh^*_{\cE_{\!n}}(A) \ra A$ are equivalent to conclude the equivalence of these two sequences, up to an automorphism of $T_A$. The map of Theorem \ref{main2} was the linearization of a map $\Omega^n \Aut_{\fB^nA} \ra A^\times$, which was the units of the usual map $\Map_{\m_A^{\cE_n}}(A,A) \ra \Map_{\m_A}(A,A)$, given by the forgetful functor $\m_A^{\cE_n}\ra \m_A$. The map in Theorem \ref{main1} is the $\cE_n$-$A$-module dual of the map $U_A\ra A$, defined by the counit of the adjunction between $\m_A^{\cE_n}$ and $\m_A$. Thus, we obtain that these two maps are equivalent.
\end{proof}

\subsection{Lie Algebras and the Higher Deligne Conjecture}

In the previous section, we showed that the fiber sequence $A[n-1]\ra T_A\ra \hh^*_{\cE_{\!n}}(A)$ could obtained as the tangent spaces associated to a fiber sequence of derived algebraic groups. The sole remaining point of discussion is to identify the algebraic structure on this sequence. As tangent spaces of algebraic groups, one should expect that this is a sequence of (restricted) Lie algebras, as Kontsevich conjectured in \cite{motives}: This is indeed the case. This sequence has more structure, however: After shifting, it is is a sequence of nonunital $\cE_{n+1}$-algebras.

We briefly explain in what sense this can be regarded as being {\it more} structured. An associative algebra can be equipped with the commutator bracket, which gives it the structure of a Lie algebra. A similar fact is the case for general $\cE_n$-algebras. If $A$ is an $\cE_n$-algebra in chain complexes over a field $\FF$, then there is a map $\cE_n(2)\ot A^{\ot 2}\ra A$. Passing to homology, and using that ${\rm H}_*(\cE_n(2),\FF)\cong {\rm H}_*(S^{n-1},\FF)\cong \FF\oplus \FF[n-1]$, we obtain a map $(\FF\oplus\FF[n-1])\ot {\rm H}_*(A)^{\ot 2}\ra {\rm H}_*(A)$. We thus obtain two different maps. The degree 0 map defines an associative multiplication, which is quite familiar; the degree $n-1$ map, on the other hand, defines a {\it Lie bracket}, as first proved by Cohen, \cite{cohen}. Thus, at least at the level of homology, one can think of an $\cE_n$-algebra structure on $A$ as consisting of a Lie algebra on the shift $A[n-1]$ together with some extra structure.

We now show that these structures exist on the tangent spaces we have discussed. That is, that the tangent space of an infinitesimal moduli problem admits same structure as that afforded by the tangent space of an augmented algebra. In particular, for a moduli problem for $\cO$-algebras, the tangent space has an $\cO^!$-algebra structure. First, we show that our $\cE_\infty$-moduli problems of interest admit refinements to $\cE_{n+1}$-moduli problems.

\begin{prop}\label{factorizes} There is a lift of the infinitesimal $\cE_\infty$-moduli problems associated to $\BAut_A$, $\BAut_{\fB^n}$, and $B^{n-1}A^\times$ to the $\oo$-category $\cM_{\cE_{n+1}}(\cC)$ of infinitesimal $\cE_{n+1}$-moduli problems.
\end{prop}

\begin{proof} It suffices to show that there is a factorization of $\GG: \cE_\infty\alg(\cC)\ra \cE_{n+1}\alg(\cC) \ra \Omega\textendash\spaces$, for each $\GG$ among the groups above. Recall Theorem \ref{monoidal}, the consequence of the theorem of Dunn, \cite{dunn} and \cite{dag6}, for an $\cE_{n+1}$-algebra $R$, the $\oo$-category $\m_R(\cC)$ has the structure of an $\cE_n$-monoidal $\oo$-category. Using this, we now define the above functors: $\Aut_A:\cE_{n+1}\alg(\cC)\ra \spaces$ takes values $$\Aut_A(R) \subset\Map_{\cE_n\alg(\m_R(\cC))}(R\ot A, R\ot A)$$ consisting of those maps that are homotopy equivalences. Likewise, $\Aut_{\fB^nA}: \cE_{n+1}\alg(\cC)\ra \spaces$ is defined by taking values $$\Aut_{\fB^nA}(R)\subset \Map_{{\rm Cat}_{\oo,n}(\m_R(\cC))}(\fB^n(R\ot A), \fB^n(R\ot A)).$$ And likewise the previous formula may be applied to define $B^{n-1}A^\times$.
\end{proof}

The argument that an infinitesimal $\cO$-moduli problem $\cF$ has a tangent space in $\cC$ is identical to the $\cE_\infty$ case: The assignment $M \rightsquigarrow \cF(k\oplus M)$ can be delooped to form a spectrum-valued functor, which is equivalent to one of the form $\Map_\cC(k,T\cF\ot -)$. We now show that the tangent space of an $\cO$-moduli problem obtains the same algebraic structure that the tangent space of an augmented $\cO$-algebra possesses. The essential idea is that infinitesimal moduli problems are expressible as geometric realizations and filtered colimits of affines (i.e., functors of the form $\Spec A = \Map_{\cO}(A,-)$): Geometric realizations and filtered colimits preserve algebraic structure, therefore the tangent space of the moduli problem retains the algebraic structure of the terms in the resolution.

\begin{prop}\label{extend} Assume that there exists a functorial $\cO^\vee$-algebra structure on the tangent space of an augmented $\cO$-algebra, for some operad $\cO^\vee$. I.e., we are given a factorization of $T: \cO\alg^{\rm aug}(\cC)^{\op}\ra \cC$, through the forgetful functor $\cO^\vee\alg^{\rm nu}(\cC)\ra \cC$. Then the tangent space of every infinitesimal moduli problem $\cF$ canonically obtains an $\cO^\vee$-algebra structure. I.e., there is a lift:
\[\xymatrix{
\cO\alg^{\rm aug}(\cC_\diamond)^{\op}\ar[r]^{\widetilde{T}}\ar[d]_{\Spec}&\cO^\vee\alg^{\rm nu}(\cC)\\
\cM_\cO(\cC)\ar@{-->}[ur]^T\\}\]
\end{prop}
\begin{proof}
To economize, we abbreviate $\cO\alg:= \cO\alg^{\rm aug}(\cC_\diamond)$ for the duration of the proof. Consider the embedding $\Spec: \cO\alg^{\op}\subset \cP(\cO\alg^{\op})$, the Yoneda embedding into presheaves. Since $\cO\alg^{\op}$ is a small $\oo$-category, this factors through the $\oo$-category of ind-objects, $\ind(\cO\alg^{\op})\subset \cP(\cO\alg^{\op})$, whose essential image in the $\oo$-category of presheaves consists of those presheaves that preserve finite limits. Consider also the $\oo$-subcategory of presheaves consisting of all those functors that preserve products, which we denote $\cP_\Sigma(\cO\alg^{\op})\subset \cP(\cO\alg^{\op})$. We thus have the following sequence of fully faithful inclusions of $\oo$-categories \[\cO\alg^{\op} \subset \ind(\cO\alg^{\op}) \subset \cM_\cO(\cC)\subset \cP_\Sigma(\cO\alg^{\op})\subset \cP(\cO\alg^{\op}).\] First, we define the functor $\widetilde{T}: \ind(\cO\alg^{\op}) \ra \cO^\vee\alg^{\rm nu}(\cC)$. Any ind-object $X$ can be realized as a filtered colimit $\varinjlim \Spec A_i$ in the $\oo$-category of presheaves, and this gives an equivalence of the tangent space $TX \simeq \varinjlim TA_i$, where the filtered colimit is computed in $\cC$. However, this is a filtered diagram of $\cO^\vee$-algebras, and the forgetful functor $\cO^\vee\alg^{\rm nu}(\cC) \ra \cC$ preserves filtered colimits, hence $TX$ obtains the structure of an $\cO^\vee$-algebra.

We now extend the functor $\widetilde T$ to each moduli problem $\cF\in\cM_\cO(\cC)$. Since $\cF$ preserves products, and $\cO\alg^{\op}$ is a small $\oo$-category, there exists a simplicial resolution of $\cF$ by ind-representables, by Lemma 5.5.8.14 of \cite{topos}. That is, there exists a simplicial presheaf $\cF_\bullet \ra \cF$ mapping to $\cF$, with each $\cF_i\simeq\varinjlim \Spec A_l$ ind-representable, and such that the map $|\cF_\bullet(R)|\ra \cF(R)$ is an equivalence for every $R\in \cO\alg$. The tangent space $T\cF$ is thereby equivalent to the geometric realization of $T\cF_\bullet$, the tangent spaces of the resolution. Since the forgetful functor $\cO^\vee\alg(\cC)\ra \cC$ preserves geometric realizations, there is an equivalence in $\cC$ between $T\cF$ and $|\widetilde T\cF_\bullet|$, hence $T\cF$ obtains the structure of an $\cO^\vee$-algebra.

To state this slightly more formally, we have the following pair of left Kan extensions:
\[\xymatrix{
\cO\alg^{\op} \ar[d]_\Spec\ar[r]^{\widetilde{T}}&\cO^\vee\alg^{\rm nu}(\cC) \ar[r] &\cC\\
\cM_\cO(\cC)\ar@/_1pc/@{-->}[ur]^{\Spec_!\widetilde{T}}\ar@/_1pc/@{-->}[urr]_{\Spec_!T}\\}\]
By the preceding, both left Kan extensions can be calculated in terms of filtered colimits and geometric realizations that are preserved by the forgetful functor $\cO^\vee\alg^{\rm nu}(\cC)\ra \cC$. As a consequence we obtain the equivalence $\Spec_! \widetilde{T}\cF \simeq \Spec_! T\cF$ for each infinitesimal moduli problem $\cF$. Since the left Kan extension $\Spec_!T\cF\simeq T\cF$ exactly calculates the usual tangent space of $\cF$, we obtain a canonical $\cO^\vee$-algebra structure on $T\cF$, for each $\cF$
\end{proof}

\begin{remark} The condition of being product-preserving is essential in the definition of an infinitesimal moduli problem: One can always left Kan extend the functor $\cO\alg^{\rm aug}(\cC_\diamond)^{\op}\ra  \cO^\vee\alg(\cC)$ along the inclusion, $\Spec$, but there would be no guarantee that the result would be the tangent space of $\cF$, due to the difference between colimits in $\cO^\vee$-algebras and colimits in $\cC$. Condition (2) is a slight weakening of a Schlessinger-type formal representability condition, \cite{schlessinger}.
\end{remark}

Proposition \ref{extend}, together with Proposition \ref{Okoszul}, implies the following:

\begin{cor} For $\cO^!=(1\circ_\cO1)^\vee$ the derived Koszul dual operad of the operad $\cO$, the tangent space of an infinitesimal $\cO$-moduli problem has the structure of an $\cO^!$-algebra. That is, there is a functor $T:\cM_\cO(\cC)\ra \cO^!\alg^{\rm nu}(\cC)$.

\end{cor}

The next result follows from Proposition \ref{extend} and Proposition \ref{koszulen}, coupled with Proposition \ref{factorizes}.

\begin{cor}\label{picture}
There is a commutative diagram
$$
\xymatrix{
\cM_{\cE_1}(\cC) \ar[r]\ar[d]^T&\ldots \ar[r] & \cM_{\cE_m}(\cC)\ar[d]^T \ar[r]& \ldots \ar[r]& \cM_{\cE_\infty}(\cC)\ar[d]^T\\
\cE_1[-1]\alg^{\rm nu}(\cC)\ar[rrd]_{\rm forget} & & \cE_m[-m]\alg^{\rm nu}(\cC) \ar[d]^{\rm forget}&& \Lie[-1]\alg(\cC)\ar[dll]^{\rm forget}\\
&&\cC\\}$$ in which the vertical arrows are given the tangent spaces of the moduli problems, and the horizontal arrows are given by restriction of moduli problems along the functors $\cE_{n+k}\alg^{\rm aug}(\cC_\diamond)\ra \cE_n\alg^{\rm aug}(\cC_\diamond)$.

\end{cor}

\begin{proof} By Proposition \ref{koszulen}, we have functors $\cE_m\alg^{\rm aug}(\cC_\diamond)^{\op}\ra \cE_m[-m]\alg^{\rm nu}(\cC)$ for each $m$. The theorem follows by applying Proposition \ref{extend} to obtain the corresponding picture for infinitesimal moduli problems.
\end{proof}

\begin{remark} If an additional hypothesis is placed on the infinitesimal moduli problems, and $\cC=\m_\FF$ is chain complexes over a field $\FF$, then Lurie, in \cite{moduli}, has outlined an argument that the functor from $\cE_n$-moduli problems to nonunital $\cE_n$-algebras is an equivalence. At the level of generality of the present work, however, it seems very possible that such modification does not produce an equivalence between moduli problems and their tangent spaces.
\end{remark}

This completes the proof of Theorem \ref{big}: Since the reduction of the functors $B^{n}A^\times \ra \BAut_A\ra \BAut_{\fB^nA}$ defines a fiber sequence in $\cM_{\cE_\infty}(\cC)$, therefore their tangent spaces attain the structure of Lie algebras, after shifting by 1. Since this fiber sequence lifts to a sequence of infinitesimal $\cE_{n+1}$-moduli problems, their tangent spaces attain the structure of nonunital $\cE_{n+1}$-algebras, after shifting by $n$. Applying Lemma \ref{fiber}, these tangent complexes form a fiber sequence, the terms of which are $A[-1]$, $T_A[-n]$ and $\hh^*_{\cE_{\!n}}(A)$ by the calculations of Lemma \ref{unit}, Lemma \ref{vect}, and Theorem \ref{liehh}.

Despite showing the existence of an interesting nonunital $\cE_{n+1}$-algebra structure on $A[-1]$, and, more generally, a nonunital $\cE_{n+i}$-algebra structure on $A[-i]$, we have not offered a satisfying description of it. We suggest the following. 

\begin{conj}\label{last} Given the equivalence of operads $\cE_n[-n]\simeq \cE_n^!$, then the nonunital $\cE_{n+1}$-algebra structure on $A[-1]$ given in Theorem \ref{big} is equivalent, after suspending by $n$, to that defined by restricting the $\cE_n$-algebra structure of $A$ along the map of operads \[\cE_{n+1}[-n-1]\simeq\cE_{n+1}^! \longrightarrow \cE_n^! \simeq \cE_n[-n],\] Koszul dual to the usual map of operads $\cE_n \ra \cE_{n+1}$. Likewise, the Lie algebra structure constructed on the sequence $A[n-1] \ra T_A \ra \hh^*_{\cE_{\!n}}(A)[n]$ is equivalent to that obtained from the sequence of nonunital $\cE_{n+1}$-algebras $A[-1] \ra T_A[-n] \ra \hh^*_{\cE_{\!n}}(A)$ by restricting along a map of operads $\Lie \ra \cE_n[1-n]$, which is the Koszul dual of the usual operad map $\cE_{n+1}\ra \cE_\infty$.
\end{conj}

We conclude with several comments.

\begin{remark} This gives a new construction and interpretation of the $\cE_{n+1}$-algebra on $\cE_n$-Hochschild cohomology $\hh^*_{\cE_{\!n}}(A)$, the existence of which is known as the higher Deligne conjecture, \cite{motives}, and was previously proved in full generality in \cite{dag6} and \cite{thomas}, in \cite{hkv} over a field, and in \cite{mccluresmith} for $n=1$, among many other places. To summarize,  when $A$ is an $\cE_n$-algebra in a stable $\oo$-category, then $\hh^*_{\cE_{\!n}}(A)$ has an $\cE_{n+1}$-algebra structure because it can be identified with the tangent complex to a $\cE_{n+1}$-moduli problem. This should not be thought of as a genuinely new proof, however, because it uses the same essential ingredient as \cite{dag6}, \cite{hkv}, and possibly all of the other proofs: that $\cE_n$-algebras are $n$-iterated $\cE_1$-algebras. The particular benefit of this construction of the $\cE_{n+1}$-algebra structure is that it relates it to deformation theory, as well as to nonunital $\cE_{n+1}$-algebra structures on $T_A[-n]$ and $A[-1]$. In characteristic zero, a nonunital $\cE_{n+1}$-algebra structure on $T_A[-n]$ was previously constructed in \cite{tamarkin}.
\end{remark}

\begin{remark} A consequence of Theorem \ref{big} is, in characteristic zero, solutions to the Maurer-Cartan equation for the Lie bracket on $\hh^*_{\cE_{\!n}}(A)[n]$ classify deformations of $\fB^nA$, hence $\m_A^{(n)}$. It would interesting to have a direct construction of these deformations.
\end{remark}

\end{document}